\documentclass[11pt,letterpaper]{amsart}

\usepackage{amsmath, amscd, amssymb}
\usepackage[frame,cmtip,arrow,matrix,line,graph,curve]{xy}
\usepackage{graphpap, color}
\usepackage[mathscr]{eucal}
\usepackage{cancel}
\usepackage{verbatim}
\usepackage{adjustbox}
\usepackage{tikz}
\usepackage{float}
\usepackage{booktabs}
\usepackage{multirow}
\usepackage{tabularx}

\numberwithin{equation}{section}

\newtheorem{theorem}{Theorem}[section]
\newtheorem{corollary}[theorem]{Corollary}
\newtheorem{proposition}[theorem]{Proposition}
\newtheorem{lemma}[theorem]{Lemma}

\theoremstyle{definition}
\newtheorem{definition}[theorem]{Definition}
\newtheorem{question}[theorem]{Question}
\newtheorem{example}[theorem]{Example}
\newtheorem{remark}[theorem]{Remark}

\newtheorem{ques}[theorem]{Question}

\newcommand{\Z}{\mathbb{Z}}
\newcommand{\Q}{\mathbb{Q}}

\newcommand{\C}{\mathbb{C}}

\newcommand{\PP}{\mathbb{P}}

\def\CC{\mathbb{C}}

\def\QQ{\mathbb{Q}}

\def\RR{\mathbb{R}}

\def\ZZ{\mathbb{Z}}

\def\sO{{\mathscr O}}

\def\sK{\mathscr{K}}
\def\sP{\mathscr{P}}
\def\sQ{\mathscr{Q}}
\def\sR{\mathscr{R}}

\def\sT{\mathscr{T}}

\def\sX{\mathscr{X}}

\newcommand{\cal}{\mathcal}
\def\cA{{\cal A}}

\def\cC{{\cal C}}
\def\cD{{\cal D}}

\def\cL{{\cal L}}
\def\cM{{\cal M}}

\def\sO{{\cal O}}

\def\fp{\mathfrak{p}}
\def\fq{\mathfrak{q}}
\def\fQ{\mathfrak{Q}}

\def\fc{\mathfrak{c}}

\def\tX{{\widetilde{X}}}

\def\hbar{\overline{h}}

\def\Mbar{\overline{M}}
\def\mzn{\overline{\cM}_{0,n} }
\def\mzno{\overline{\cM}_{0,n+1} }
\def\mgn{\overline{\cM}_{g,n}}

\def\GL{\mathrm{GL} }

\def\PGL{\mathrm{PGL} }

\def\Ind{\mathrm{Ind} }
\def\Res{\mathrm{Res} }
\def\Stab{\mathrm{Stab} }
\def\dim{\mathrm{dim} }

\def\Pic{\mathrm{Pic} }
\def\log{\mathrm{log} }

\def\ch{\mathrm{ch} }

\def\and{\quad{\rm and}\quad}
\def\lra{\longrightarrow }

\def\beq{\begin{equation}}
\def\eeq{\end{equation}}
\def\ben{\begin{enumerate}}
\def\een{\end{enumerate}}

\def\DM{Deligne-Mumford }

\def\and{\quad\text{and}\quad}

\def\a{\alpha}
\def\d{\delta}
\def\e{\epsilon}
\def\w{\omega}

\title{Representations on the cohomology of $\overline{\mathcal{M}}_{0,n}$}
\date{October 9, 2023}

\author{Jinwon Choi}
\address{Department of Mathematics and Research Institute of Natural Science, Sookmyung Women's University, Seoul 04310, Korea}
\email{jwchoi@sookmyung.ac.kr}

\author{Young-Hoon Kiem}
\address{School of Mathematics, Korea Institute for Advanced Study, 85 Hoegiro, Dongdaemun-gu, Seoul 02455, Korea}
\email{kiem@kias.re.kr}

\author{Donggun Lee}
\address{Center for Complex Geometry, Institute for Basic Science (IBS), 55 Expo-ro, Yuseong-gu, Daejeon 34126, Korea}
\email{dglee@ibs.re.kr}

\thanks{JC was partially supported by NRF grant 2018R1C1B6005600. YHK was partially supported by NRF grant 2021R1F1A1046556. DL was supported by NRF grant 2019H1A2A1076706.}

\begin{document}

\begin{abstract}
The moduli space $\overline{\mathcal{M}}_{0,n}$ of $n$ pointed stable curves of genus $0$ admits an action of the symmetric group $S_n$ by permuting the marked points.
We provide a closed formula for the character of
the $S_n$-action on the cohomology of $\overline{\mathcal{M}}_{0,n}$. This is achieved by studying wall crossings of the moduli spaces of quasimaps which provide us with a new inductive construction of $\overline{\mathcal{M}}_{0,n}$, equivariant with respect to the symmetric group action.
Moreover we prove that $H^{2k}(\overline{\mathcal{M}}_{0,n})$ for $k\le 3$ and $H^{2k}(\overline{\mathcal{M}}_{0,n})\oplus H^{2k-2}(\overline{\mathcal{M}}_{0,n})$ for any $k$ are permutation representations.
Our method works for related moduli spaces as well and we provide a closed formula for the character of the $S_n$-representation on the cohomology of the Fulton-MacPherson compactification $\PP^1[n]$ of the configuration space of $n$ points on $\PP^1$ and more generally on the cohomology of the moduli space  $\overline{\mathcal{M}}_{0,n}(\PP^{m-1},1)$ of stable maps.
\end{abstract}

\maketitle
\tableofcontents

\section{Introduction}

The moduli space $$\cM_{0,n}=[(\PP^1)^n-\Delta]/\mathrm{Aut}(\PP^1), \quad \Delta=\bigcup_{i\ne j}\{p_i=p_j\}$$of $n$ distinct points $p_1,\cdots, p_n$ on projective line $\PP^1$ up to projective equivalence is one of the most studied objects in algebraic geometry since the nineteenth century and its compactification $\overline{\mathcal{M}}_{0,n}$ by \DM \cite{DM} has played a major role in enumerative algebraic geometry and related areas.
The symmetric group acts on $\overline{\mathcal{M}}_{0,n}$ by permuting the marked points and hence on the cohomology $H^*(\overline{\mathcal{M}}_{0,n})$.
Thanks to \cite{BM, Get, Has, Kap, Kee, KM, Manin, RS} to name a few, a lot is known about $H^*(\overline{\mathcal{M}}_{0,n})$ but we still cannot answer simple questions like
\begin{ques}\label{q1}
\begin{enumerate}
\item Can we give a closed formula for the character of the $S_n$-representation on $H^{2k}(\overline{\mathcal{M}}_{0,n})$ for each $k$?
\item For each $k$, is $H^{2k}(\overline{\mathcal{M}}_{0,n})$ a permutation representation of $S_n$? (\cite[Question 1]{RS}.)
\end{enumerate}
\end{ques}
\noindent Here a representation $\phi:S_n\to \GL_N(\CC)$ is called a permutation representation if  $\phi$ factors through the subgroup $S_N\subset \GL_N(\CC)$ of permutation matrices.

The purpose of this paper is to provide answers to the above questions through a new inductive construction of the moduli space $\overline{\mathcal{M}}_{0,n}$ motivated by our investigation \cite{CK} on the Landau-Ginzburg/Calabi-Yau correspondence.
More precisely,
\begin{itemize}
\item (Theorem \ref{Thm.ClosedForm}) we provide a closed formula for the $S_n$-character on $H^*(\overline{\mathcal{M}}_{0,n})$;
\item (Corollary \ref{Cor.SingleDeg}) we prove that $H^{2k}(\overline{\mathcal{M}}_{0,n})$ is a permutation representation for $k\le 3$ or $k\ge n-6$;
\item (Theorem \ref{Cor.SubseqDeg}) we also prove that $H^{2k}(\overline{\mathcal{M}}_{0,n})\oplus H^{2k-2}(\overline{\mathcal{M}}_{0,n})$ is a permutation representation for each $k$.
\end{itemize}
In principle, our closed formula should enable us to answer Question \ref{q1} (2) in full generality but the combinatorics seems too complicated to work out for the moment.
Note that the $S_n$-representation on $H^{2k}(\overline{\mathcal{M}}_{0,n})$ was proved to be a permutation representation for $k=1$ or $k=n-4$ by Farkas-Gibney \cite{FG} and for $k=2$ or $k=n-5$ by Ramadas-Silversmith \cite{RS} through completely different methods.

We remark that our method works for other moduli spaces including the Fulton-MacPherson compactification (cf. \cite{FM})
$$\PP^1[n]=\overline{\mathcal{M}}_{0,n}(\PP^1,1)$$
and enables us to provide a closed formula (Theorem \ref{FMmain}) for the character of the $S_n$-representation on the cohomology of $\overline{\mathcal{M}}_{0,n}(\PP^{m-1},1)$ for any $m\ge 2$.

In this paper, all varieties and schemes are defined over the complex number field $\CC$ and we will use rational coefficients $\QQ$ for cohomology groups.

\subsection{Cohomology of $\overline{\mathcal{M}}_{0,n}$}

Let us recall previously known results about the cohomology of $\overline{\mathcal{M}}_{0,n}$.

An $n$-pointed stable curve of genus 0 is a connected projective curve $C$ of arithmetic genus 0 with at worst nodal singularities together with $n$ distinct smooth points $p_1,\cdots, p_n$ such that there is no nontrivial automorphism of $C$ preserving the marked points $p_i$ (cf. \cite{DM}).
For convenience, we will often write $C$ instead of $(C,p_1,\cdots, p_n)$ when the meaning is clear from the context.
The moduli functor of $n$-pointed stable curves of genus $0$ is represented by a smooth projective variety of dimension $n-3$ and an explicit inductive construction was first provided by Knudsen
in \cite{knu} in 1983.

Motivated by enumerative geometry and string theory, the cohomology ring of $\overline{\mathcal{M}}_{0,n}$ was investigated in 1990s by several authors. In \cite{Manin}, Manin used the virtual Poincar\'e polynomial coming from Deligne's mixed Hodge structure to give an inductive formula for the Betti numbers of the cohomology $H^*(\overline{\mathcal{M}}_{0,n})$.
In \cite{Kee}, Keel discovered an alternative inductive construction of $\overline{\mathcal{M}}_{0,n+1}$ by applying a sequence of blowups to $\overline{\mathcal{M}}_{0,n}\times \PP^1$ along smooth subvarieties.
By the blowup formulas \cite[p.605]{GrHa} and \cite[Theorem 6.7]{Ful}, he proved that
the cycle class map $$A^*(\overline{\mathcal{M}}_{0,n})\lra H^*(\overline{\mathcal{M}}_{0,n})$$ from the Chow group to the cohomology is an isomorphism and $H^*(\overline{\mathcal{M}}_{0,n})$ is generated by boundary divisors with only some evident relations like WDVV.
In particular, there is no odd degree cohomology part and the cohomology group $H^{2k}(\overline{\mathcal{M}}_{0,n})$ coincides with the Chow group $A^k(\overline{\mathcal{M}}_{0,n})$.

As mentioned above, the symmetric group $S_n$ acts on $\overline{\mathcal{M}}_{0,n}$ by permuting the marked points $p_1,\cdots, p_n$. In fact, there are no other automorphisms of $\overline{\mathcal{M}}_{0,n}$ for $n\ge 5$ by \cite{BrM}.
Thus it is natural to think of the cohomology of $\overline{\mathcal{M}}_{0,n}$ as an $S_n$-representation. Unfortunately, Keel's construction consists of morphisms and blowups which are not $S_n$-equivariant, and hence it cannot be used to compute the representations on the cohomology of $\overline{\mathcal{M}}_{0,n}$.

In \cite{CGK}, Chen, Gibney and Krashen studied a different but similar construction of $\overline{\mathcal{M}}_{0,n+1}$ by applying a sequence of blowups to a projective bundle over $\overline{\mathcal{M}}_{0,n}$.\footnote{We thank the referee who kindly pointed out that \cite{CGK} contains a similar construction.} See \cite[Theorem 3.3.1 and Proposition 3.4.3]{CGK}. However, their construction is not $S_n$-equivariant although obviously $S_{n-1}$-equivariant. This is basically because their interpretation detects only an $S_{n-1}$-action on $\overline{\mathcal{M}}_{0,n}$ which fixes one marked point, called the \emph{root} in their terminology (\cite[Definition 2.0.1]{CGK}). See Appendix~\ref{App.CGK} for details.

In \cite{Get}, Getzler gave an algorithm for the characters of $H^*(\overline{\mathcal{M}}_{0,n})$ using the language of operads, and managed to compute them for $n\le 6$.
In 1994, Kapranov discovered a new construction of $\overline{\mathcal{M}}_{0,n}$ by a sequence of blowups starting with $\PP^{n-3}$. In fact, after fixing $n-1$ general points $q_1,\cdots, q_{n-1}$ in $\PP^{n-3}$, we blow up $\PP^{n-3}$ along the $n-1$ points $\{q_i\}_{1\le i<n}$, along the proper transforms of the $\binom{n-1}{2}$ lines $\{\overline{q_iq_j}\}$, along the proper transforms of $\binom{n-1}{3}$ planes $\{\overline{q_iq_jq_k}\}$, and so on. After these $n-2$ blowups, we end up with $\overline{\mathcal{M}}_{0,n}$. For each permutation of the points $q_1,\cdots, q_{n-1}$, there is a unique projective transformation of $\PP^{n-3}$ which extends the permutation and hence we can think of $S_{n-1}$ as a subgroup of $\PGL_{n-2}(\CC)$. Obviously, the blowup centers are invariant with respect to this subgroup $S_{n-1}$ and the induced action of $S_{n-1}$ on $\overline{\mathcal{M}}_{0,n}$ coincides with the action of the subgroup $S_{n-1}$ of $S_n$ permuting $\{2,\cdots, n\}$ with $1$ fixed.
By applying the blowup formula, we can calculate the cohomology $H^*(\overline{\mathcal{M}}_{0,n})$ as an $S_{n-1}$-representation.
This calculation will be explicitly worked out in \S\ref{Sec.11}.

In 2003, Hassett introduced in \cite{Has} the notion of weighted pointed stable curves and showed that Kapranov's blowups are actually the natural morphisms among a special type of moduli spaces of weighted pointed stable curves of genus 0. Moreover the blowup centers are also the moduli spaces of suitable weighted pointed stable curves of genus 0.
Using Hassett's moduli spaces and morphisms, Bergstr\"om and Minabe in \cite{BM}
gave a recursive algorithm to compute the characters of the $S_n$-representations $H^*(\overline{\mathcal{M}}_{0,n})$ and managed to compute them for $n\le 20$ (cf. \cite{RS}).
The algorithm starts with a result of the second named author and Moon in \cite{KM} that factorizes the natural morphism $$\overline{\mathcal{M}}_{0,n}\lra (\PP^1)^n/\!/\PGL_2(\CC)$$ to the geometric invariant theory (GIT) quotient into a sequence of explicit blowups.
The algorithm uses the blowup formula and pushes the computation little by little to $\overline{\cM}_{0,k}$ with $k<n$.
However, no algorithms known so far have provided us with a closed formula for the characters of
the $S_n$-representation $H^*(\overline{\mathcal{M}}_{0,n})$.
The first goal of this paper is to provide a closed formula for these characters
(Theorem \ref{Thm.ClosedForm}). Our result answers Question \ref{q1} (1) positively.

Recall that an $S_n$-representation $V$ is called a \emph{permutation representation} if it has a basis $B$ (called a \emph{permutation basis}) such that an $S_n$-action on $B$ induces the $S_n$-action on $V$.
An interesting question related to $\overline{\mathcal{M}}_{0,n}$ is whether the cohomology group $A^k(\overline{\mathcal{M}}_{0,n})=H^{2k}(\overline{\mathcal{M}}_{0,n})$ is a permutation representation or not for each $k$ (Question \ref{q1} (2)).
Here is some evidence on why the answer to this question may be yes.
\begin{enumerate}
\item A recent work \cite{CT} of Castravet and Tevelev which is again based on \cite{KM} tells us that the derived category of $\overline{\mathcal{M}}_{0,n}$ admits a full exceptional collection which is invariant under the action of $S_n$.
Hence the Grothendieck group of coherent sheaves on $\overline{\mathcal{M}}_{0,n}$ is a permutation representation of $S_n$ and so is the total cohomology
$\oplus_{k=0}^{n-3}A^k( \overline{\mathcal{M}}_{0,n})$ by Grothendieck-Riemann-Roch (GRR).
However, as GRR mixes the degrees, it does not imply that $A^k( \overline{\mathcal{M}}_{0,n})$ itself is a permutation representation for each $k$.
\item
For $k=1$ (and hence for $k=n-4$ by Poincar\'e duality), Farkas and Gibney \cite[Lemma 2]{FG} found a permutation basis for the Picard group of $ \overline{\mathcal{M}}_{0,n}$. Ramadas \cite{Ram} found another permutation basis for $A^1( \overline{\mathcal{M}}_{0,n})$ which is not isomorphic to that in \cite{FG}.
\item Ramadas and Silversmith in \cite{RS} extended the technique of \cite{Ram} to give a permutation basis for $A^2( \overline{\mathcal{M}}_{0,n})$.
\item Ramadas and Silversmith checked that $A^k( \overline{\mathcal{M}}_{0,n})$ is a permutation representation for $n\le 10$ and for all $k$ by using (unpublished) computational data of Bergstr\"om-Minabe.
\end{enumerate}
The second goal of this paper is to answer Question \ref{q1} (2) as much as we can.
Using the closed formula (Theorem \ref{Thm.ClosedForm}), we prove (Corollary \ref{Cor.SingleDeg}) that
\beq\label{15}
A^k(\overline{\mathcal{M}}_{0,n})\text{ is a permutation representation of }S_n\text{ for }k\le 3\text{ or }k\ge n-6.\eeq
Our method is completely different from known approaches even for $k=1, 2$. Furthermore, we prove (Theorem \ref{Cor.SubseqDeg}) that
\beq\label{16}
A^k( \overline{\mathcal{M}}_{0,n})\oplus  A^{k-1}( \overline{\mathcal{M}}_{0,n})\text{ is a permutation representation of }S_n\text{ for all }k\text{ and }n.\eeq

We prove the above by a new inductive construction of $\overline{\mathcal{M}}_{0,n}$ which arises from
the wall crossings of the moduli spaces of $\delta$-stable quasimaps developed in \cite{CK}.
Our method works well with other interesting moduli spaces like the stable map space $\overline{\mathcal{M}}_{0,n}(\PP^{m-1},1)$ by simply considering more sections for quasimaps (Theorem \ref{FMmain}).

\subsection{Strategy of the proofs}
Recall that Kapranov's construction of $\overline{\mathcal{M}}_{0,n+1}$ by applying a sequence of blowups to $\PP^{n-2}$ is $S_n$-equivariant. The blowup centers are also obtained by applying sequences of blowups to projective spaces of lower dimensions. By the blowup formula (Lemma \ref{Lem.BlowupFormula}), we can calculate the $S_n$-representation on $H^*(\overline{\mathcal{M}}_{0,n+1})=A^*(\overline{\mathcal{M}}_{0,n+1})$.
The combinatorics of the blowups is handled by \emph{weighted rooted trees} (Definition \ref{Def.RootedTree}) which are trees with certain decorations. Letting
\beq\label{10} Q_n=\sum_k\ch_n(A^k(\overline{\mathcal{M}}_{0,n+1}))t^k\eeq
where $\ch_n$ denotes the character of $S_n$, we will show (Proposition \ref{Prop.RepKap}) that
\beq\label{11}
Q_n=\sum_{k\ge 0} \sum_{T\in \sT_{n,k}/S_n}\ch_n(U_{T})t^k
\eeq
where $\sT_{n,k}$ denotes the set of weighted rooted trees with $n$ inputs and weight $k$ while
$$U_T=\mathrm{Ind}^{S_n}_{\mathrm{Stab}(T)}e$$ is the induced $S_n$-representation from the trivial representation of $\mathrm{Stab}(T)$.

Now our goal is to find
\beq\label{12} P_n=\sum_k\ch_n(A^k(\overline{\mathcal{M}}_{0,n}))t^k\eeq
from \eqref{11}. So we want to compare $A^*(\overline{\mathcal{M}}_{0,n})$ with $A^*(\overline{\mathcal{M}}_{0,n+1})$ as $S_n$-representation spaces.
For this, we use a new inductive construction of $\overline{\mathcal{M}}_{0,n}$ which arises from the wall crossings of moduli spaces of $\delta$-stable quasimaps in \cite{CK}

To build $\overline{\mathcal{M}}_{0,n+1}$ from $\overline{\mathcal{M}}_{0,n}$, we should add one extra point to a stable curve $(C,p_1,\cdots,p_n)$ with $n$ marked points. If we think of a point on $C$ as a pair $(L,s)$ of a line bundle $L$ of degree $1$ on $C$ and a nonzero section $s$ of $L$, we gain freedom to choose a stability condition for the pair $(L,s)$. Le Potier in \cite{LeP2} discovered that there is a sequence of stability conditions for such a pair over a smooth projective curve and Huybrechts-Lehn in \cite{HuLe} generalized Le Potier's stability conditions and moduli spaces to arbitrary projective schemes.
In \cite{CK}, the construction was further extended to the moduli spaces of triples $(C,L,s)$ where $C$ is now allowed to vary.  See \S\ref{Sec.qmap} for details.

As a consequence, we have a sequence of $S_n$-equivariant blowups (Theorem \ref{26})
\begin{equation}\label{wc}
  \overline{\mathcal{M}}_{0,n+1}\cong \fQ^{\d=\infty}\xrightarrow{~\fq_2~ } \cdots \xrightarrow{~\fq_\ell~} \fQ^{\d=0^+}
\end{equation}
whose blowup centers are disjoint unions of $\binom{n}{h}$ copies of $$\overline{\mathcal{M}}_{0,h+1}\times \overline{\mathcal{M}}_{0,n-h+1}, \quad 2\le h\le \ell=\lfloor\frac{n-1}{2}\rfloor.$$
Using this, by the blowup formula, we will show (Proposition \ref{Prop.SnRep}) that
\beq\label{13}
Q_n=P_{\fQ^{\d=0^+}}+t\sum^{\ell}_{h=2}Q_hQ_{n-h}
\eeq
where $P_{\fQ^{\d=0^+}}$ denotes the character polynomial of the moduli space $\fQ^{\d=0^+}$ of $\d=0^+$-stable quasimaps.
Together with \eqref{11}, \eqref{13} gives us a closed formula for $P_{\fQ^{\d=0^+}}$.

The last moduli space $\fQ^{\d=0^+}$ in \eqref{wc} is a $\PP^1$-bundle over $\overline{\mathcal{M}}_{0,n}$ when $n$ is odd.
When $n$ is even, $\fQ^{\d=0^+}$ is the blowup of a $\PP^1$-bundle over $\overline{\mathcal{M}}_{0,n}$ along $\frac12\binom{n}{\ell+1}$ disjoint copies of $\overline{\mathcal{M}}_{0,\ell+2}\times \overline{\mathcal{M}}_{0,\ell+2}$.
We will show (Corollary \ref{CMindel}) that
\beq\label{14}
P_{\fQ^{\d=0^+}}=
(1+t)P_n+ts_{(1,1)}\circ Q_{\frac{n}{2}}
\eeq
where $Q_{\frac{n}{2}}$ is defined to be zero if $n$ is odd.
Here $s_{(1,1)}\circ Q_{\frac{n}{2}}$ denotes the antisymmetric part in $ Q_{\frac{n}{2}}. Q_{\frac{n}{2}}$. See \S\ref{sec:prelim} for the notations in the $S_n$-representation theory.

From \eqref{11}, \eqref{13} and \eqref{14}, we obtain a closed formula for $P_n$ (Theorem \ref{Thm.ClosedForm}).
Then a direct computation using our closed formula for $P_n$ proves \eqref{15} (cf. Corollary \ref{Cor.SingleDeg}).
By showing that all the blowup contributions in \eqref{13} cancel with factors of $Q_n$, we obtain \eqref{16}. Moreover, by taking the invariant parts, we can calculate the cohomology of the moduli space $\mzn/S_n$ of stable curves of genus 0 with unordered marked points (Corollary \ref{63}).

Although it is ideal to find a permutation basis for $A^k( \overline{\mathcal{M}}_{0,n})$, it is not easy from our computation to find an explicit permutation basis even for $A^k( \overline{\mathcal{M}}_{0,n})\oplus  A^{k-1}( \overline{\mathcal{M}}_{0,n})$. The main reason is that the cancellation of the blowup contributions in \eqref{13} is proved in a purely combinatorial manner. Nevertheless, when $k$ is small, the combinatorial complexity is manageable and the computation for $A^k( \overline{\mathcal{M}}_{0,n})$ can be done by hand for all $n$. See Corollary \ref{Cor.SingleDeg}.

The notion of a quasimap is quite flexible and it can be adapted for many other problems.
By simply increasing the number $m$ of sections of a line bundle $L$ of degree 1, our method applies to the computation of the $S_n$-representation on the cohomology of the moduli space
$\overline{\mathcal{M}}_{0,n}(\PP^{m-1},1)$
of stable maps of degree 1 to $\PP^{m-1}$ and gives a closed formula (Theorem \ref{FMmain}). In particular,
we can calculate the $S_n$-representation on the cohomology of the Fulton-MacPherson space $\PP^1[n]$ constructed in \cite{FM}.

The $\d$-wall crossing was originally developed to provide a moduli theoretic framework for the Landau-Ginzburg/Calabi-Yau correspondence (cf. \cite{CK}). The wall crossings in \eqref{wc} are for the Calabi-Yau side of the correspondence. On the Landau-Ginzburg side, we have the moduli spaces of quasimaps on twisted curves. An analogous investigation of the representation theory on the latter side will be also an interesting line of research for the future.

\subsection{Layout of the paper}

This paper is organized as follows.
In \S\ref{Sec.qmap}, we review the wall crossings of the moduli spaces of quasimaps.
In \S\ref{sec:prelim}, we review basic facts on representations of the symmetric group $S_n$.
In \S\ref{Sec.13}, we prove \eqref{13} and \eqref{14} which enable us to calculate $P_n$ from $Q_n$.
In \S\ref{Sec.11}, we prove \eqref{11} which gives us $Q_n$.
In \S\ref{Sec.CF}, we obtain closed formulas and prove Theorem \ref{Thm.ClosedForm} as well as Theorem \ref{FMmain}.
In \S\ref{Sec.PermRep}, we prove \eqref{16} and complete our proof of \eqref{15}.
In \S\ref{sec:cuspidalblock}, we discuss about the K-theoretic cuspidal block of $\overline{\mathcal{M}}_{0,n}$ as an application.
In Appendix~\ref{App.CGK}, we compare \eqref{wc} with the blowup construction of $\overline{\mathcal{M}}_{0,n+1}$ from $\overline{\mathcal{M}}_{0,n}$ studied in \cite{CGK}.
In Appendices~\ref{App.A} and \ref{App.B}, we record the $S_n$-characters of the cohomology of $\overline{\mathcal{M}}_{0,n}$ and $\PP^1[n]$ for $n\le 11$.

\bigskip

\noindent\textbf{Acknowledgement}.
We thank Han-Bom Moon and Hyeonjun Park for useful discussions and comments.

\bigskip

\section{Wall crossings of the moduli spaces of quasimaps}\label{Sec.qmap}
In this section, we review necessary facts about stable quasimaps and their moduli defined and studied in \cite{CK}. The study of wall crossing behavior of these moduli spaces as the stability condition varies results in a new factorization of the morphism forgetting the first marked point
$$\mzno\lra \mzn$$
into the sequence \eqref{wc} of smooth blowups and a projection $\fQ^{\d=0^+}\to \mzn$ which are $S_n$-equivariant.
Moreover, we will see that the wall crossings of quasimap moduli spaces give us a factorization of the forgetful morphism
$$\mzn(\PP^{m-1},1)\lra \mzn$$
from the moduli space of stable maps of degree 1 to $\mzn$ for $m\ge 2$.

\subsection{$\d$-stability of pairs} \label{Subsec.Delta.Pairs}

In this subsection, we review the notion of $\d$-stability for a
pair $(E,\alpha)$ of a coherent sheaf $E$ and a multi-section $\a$ from \cite{LeP2, HuLe}.

Let $m\in \mathbb{N}$. Let $C$ be a connected projective curve with at worst nodal singularities. Let us call such a curve \emph{prestable}.
Fix an ample line bundle $\sO_C(1)$.
Let $E$ be a coherent sheaf on $C$. A \emph{multi-section} of $E$ refers to a homomorphism \[\a:\sO_C^{\oplus m}\longrightarrow E\]
of $\sO_C$-modules.
When $E$ is a line bundle and $\alpha$ is surjective, the pair $(E,\a)$ defines a morphism $|\a|:C\to \PP^{m-1}$.
For $\a\ne 0$, we call a pair $(E,\a)$ a \emph{quasimap} from $C$ to $\PP^{m-1}$ when $E$ is a line bundle.

For a subsheaf $E'\subset E$, we let $\theta(E',\a)=1$ if $\a$ factors through $E'$, and $\theta(E',\a)=0$ otherwise. For $\d\in \RR_{>0}$, we write
\[P^\d_{E',\a}(t):=P_{E'}(t)+\theta(E',\a)\d,\]
where $P_{E'}(t)=r(E')t+\chi(E')$ is the Hilbert polynomial of $E'$ with respect to $\sO_C(1)$. If $r(E')\neq 0$, we define the \emph{reduced Hilbert polynomial} of $(E',\a)$ by
\[p^\d_{E',\a}(t):=\frac{P^\d_{E',\a}(t)}{r(E')}=t+\frac{\chi(E')}{r(E')}+\theta(E',\a)\frac{\d}{r(E')}.\]

Recall that a coherent sheaf $E$ is said to be \emph{pure} if for every nonzero subsheaf $E'$ of $E$, the support of $E'$ has the same dimension as the support of $E$. Let $E$ be a 1-dimensional sheaf on $C$ and let $\a$ be a nonzero multi-section of $E$.

\begin{definition}\label{dstdef}
For $\d> 0$, a pair $(E, \a)$ on $C$ is called \emph{$\d$-(semi)stable} with respect to $\sO_C(1)$ if $E$ is pure and for every nonzero subsheaf $E'\neq E$,
\beq\label{dsteq} p^\d_{E',\a}<p^\d_{E,\a} \quad \text{ (resp. } p^\d_{E',\a}\leq
p^\d_{E,\a}).\eeq
\end{definition}

Note that when $m=1$, the notion of $\d$-stability of pairs agrees with the stability defined by Le Potier in \cite{LeP2}.

A morphism $\varphi: (E_1, \a_1) \xrightarrow{}(E_2,\a_2)$ of pairs refers to a homomorphism $\varphi: E_1 \xrightarrow{}E_2$ of $\sO_C$-modules satisfying $\varphi \circ \a_1=\a_2$. We say $\varphi$ is an \emph{isomorphism} if it is an isomorphism of $\sO_C$-modules.

Nowadays it is standard to construct the moduli space of $\d$-stable pairs by GIT.
\begin{theorem} \cite{HuLe}
For $\d>0$, there is a projective coarse moduli scheme of (S-equivalence classes of) $\delta$-semistable pairs on $C$ if we fix the Hilbert polynomial of $E$ with respect to $\sO_C(1)$.
\end{theorem}
So far, we considered sheaves and sections on a fixed curve $C$.
In the next subsection, we will also vary the curve $C$.

\subsection{$\delta$-stable quasimaps}
Let $d, m\in \mathbb{N}$ and $g,n\in \ZZ_{\ge 0}$. Recall that the degree of a line bundle on a prestable curve is defined to be the sum of the degrees of its restrictions to the irreducible components.

For a prestable curve $C$ with $n$ smooth points $p_1,\cdots,p_n\in C$,
the log canonical line bundle is defined as
$$\omega_{C}^{\log}:=\omega_{C}(\sum_{i=1}^n p_i).$$
A prestable curve $C$ with {distinct} smooth marked points $p_1,\cdots,p_n$ is \emph{stable} if and only if $\omega_{C}^{\log}$ is ample.

In this paper, a \emph{quasimap} to $\PP^{m-1}$ of degree $d$ and genus $g$ with $n$ marked points refers to a triple $(C, L, s)$ of an $n$-pointed prestable curve $C$ of arithmetic genus $g$, a line bundle $L\in \Pic^d(C)$ of degree $d$ on $C$ and a nonzero multi-section $s\in H^0(C,L)^{\oplus m}$.
We recall the following definition from \cite{CK}.

\begin{definition} \label{Def.Delta.Stability} \cite[Definition 5.2]{CK}
    For $\delta{> 0}$, a quasimap $(C, L, s)$ is called \emph{$\delta$-(semi)stable} if
        \begin{enumerate}
            \item[(i)] $\omega_C^{\log}\otimes L^a$ is ample for any $a\in \QQ_{>0}$, and
            \item[(ii)] for every line bundle $A\in \Pic(C)$ such that $\omega_C^{\log}\otimes A^\epsilon$ is ample for any $\epsilon >0$, a pair $(L,s)$ is $\delta$-(semi)stable with respect to $\omega_C^{\log}\otimes A^\e$ in the sense of Definition \ref{dstdef} for arbitrarily small $\e>0$.
        \end{enumerate}

        We say that a quasimap is $\emph{strictly $\d$-semistable}$ if it is $\d$-semistable but not $\d$-stable.
\end{definition}

A pointed prestable curve is called \emph{semistable} if every rational irreducible component has at least two special points (i.e. nodes or marked points). A rational component of a semistable curve is called a \emph{rational bridge} if it has precisely two nodes and no marked points, or is called a \emph{rational tail} if it has only one node and one marked point. A pointed semistable curve is called \emph{quasi-stable} if it has no rational tails and any chain of rational bridges is of length at most 1. We remark the following two simple facts (see \cite[p.176]{CK}).
\begin{enumerate}
    \item[1.] If a quasimap $(C,L,s)$ is $\delta$-semistable for some $\d>0$, then $C$ is quasi-stable. If $E\subset C$ is a rational bridge, then $L|_E\simeq \sO_E(1)$.
    \item[2.] The condition (ii) in Definition~\ref{Def.Delta.Stability} can be replaced by
    \begin{itemize}
        \item[(ii$'$)] the pair $(\rho_*L, \rho_*s)$ is $\delta$-(semi)stable on $\bar{C}$ with respect to $\omega_{\bar{C}}^{\log}$,
    \end{itemize}
    where $\rho:C\xrightarrow{} \bar{C}$ is the stabilization morphism and $\omega_{\bar{C}}^{\log}:=\omega_{\bar{C}}(\sum \rho(p_i))$ for the marked points $p_i$ of $C$.
\end{enumerate}
The stabilization morphism $\rho$ contracts rational bridges and for all $t$, we have $\chi(L\otimes (\w_C^\log)^t)=\chi(\bar{L}\otimes(\w_{\bar{C}}^\log)^t)$ where $\bar{L}=\rho_*L$.

\subsection{Moduli of $\d$-stable quasimaps} \label{18}
In this subsection, we define the moduli stack $\fQ^{\d}$ of $\d$-semistable quasimaps for each $\d>0$ and define the notion of a wall. When $\d$ is not a wall, $\fQ^\d$ is a proper \DM stack of finite type over $\C$.

Let $\fQ=\fQ_{g,n,d,m}$ denote the stack of quasimaps to $\PP^{m-1}$ of degree $d$ and genus $g$ with $n$ marked points, whose section over a scheme $S$ consists of
\begin{enumerate}
\item a flat proper morphism $\cC\xrightarrow{} S$ of finite type with $n$ disjoint smooth sections whose geometric fibers are $n$-pointed prestable curves of genus $g$,
\item a line bundle $\cL$ on $\cC$ whose restriction to each fiber over $S$ is of degree $d$,
\item a homomorphism $\sO_{\cC}^{\oplus m}\xrightarrow{} \cL$ which is nonzero on each fiber over $S$.
\end{enumerate}
By \cite[Theorem 2.5]{CK}, $\fQ$ is an algebraic stack.
In fact $\fQ$ is a huge stack and by fixing a stability condition, one can cut out
many interesting open substacks.
For instance, the moduli stack
$\mgn(\PP^{m-1},d)$ of stable maps is an open substack of $\fQ$ consisting of quasimaps $(C,L,s)$ such that $\omega_{C}^{\log}\otimes L^3$ is ample and $s:\sO_C^{\oplus m}\to L$ is surjective, by \cite[Lemma 2.4]{CK}.

For each $\delta>0$, $\fQ^\delta=\fQ^\d_{g,n,d,m}$ is now defined as the open substack consisting of $\delta$-semistable quasimaps \cite[Remark 5.7]{CK}.
For our purpose, it is important to see how $\fQ^\delta$ varies as $\d$ changes.

\begin{definition} \cite[Definition 5.8]{CK} We say that $\d$ is \emph{general} with respect to a polynomial $P(t)=rt+\chi\in\mathbb{Z}[t]$ if there are no strictly $\d$-semistable quasimaps $(C,L,s)$ such that $P_{\bar{L}}(t)=P(t)$ where $\rho:C\xrightarrow{ }\bar{C}$ is the stabilization morphism and $\bar{L}=\rho_*L$. We say $\d$ is a \emph{wall} if it is not general.
\end{definition}
By \cite[Lemma 5.9]{CK}, for a fixed quadruple $(g,n,d,m)$, there are only finitely many walls and the $\d$-stability of quasimaps remains constant between two consecutive walls.

For every quasimap $(C,L,s)\in \fQ^\d_{g,n,d,m}$, the Hilbert polynomial of $\bar{L}$ is
$$P_{g,n,d}:=P_{\bar{L}}(t)=(2g-2+n)t+d+1-g.$$
\begin{theorem} \cite[Theorem 5.10]{CK} \label{Thm.DM}
    Suppose that $2g-2+n>0$ and $d+\delta \geq g-1$. For $\d>0$ general with respect to $P_{g,n,d}$, $\fQ^\d=\fQ^\d_{g,n,d,m}$ is a proper \DM stack of finite type over $\C$.
\end{theorem}

In this paper, we are only interested in the case of $g=0$.
\begin{lemma} Let $g=0$ and $\d$ be general with respect to $P_{0,n,d}$. Then the moduli stack $\fQ^\d=\fQ^\d_{0,n,d,m}$ is a smooth projective variety of dimension $n+m(d+1)-4$. \end{lemma}
\begin{proof}
As $\d$ is general, there are no strictly $\d$-semistable quasimaps and hence no nontrivial automorphisms on $\fQ^\d$.

By \cite[Theorem 2.5]{CK}, the obstruction sheaf for $\fQ^\d$ is
$R^1\pi_*\cL^{\oplus m}$ where $\cL$ is the universal line bundle over the universal curve $\pi:\cC\to \fQ^\d$.
Since $H^1(L)^{\oplus m}$ is zero whenever $g=0$ and $d>0$, $R^1\pi_*\cL^{\oplus m}=0$ by the base change property. Hence $\fQ^\d$ is smooth.

The dimension count comes from \cite[Corollary 2.6]{CK}.
\end{proof}

\subsection{Wall crossings for  $m=d=1$} \label{Sec.Deltawc.Genus0}

Our computation of $A^k(\overline{\mathcal{M}}_{0,n})$ utilizes the wall crossings of the moduli spaces of $\d$-stable quasimaps when $g=0$ and $m=d=1$.

We write $\d=\infty$ when $\d$ is sufficiently large so that there are no walls larger than $\d$ and $\d=0^+$ when $\d$ is a sufficiently small positive number.

\begin{lemma} \label{Lem.delta.infty}
    When $m=d=1$, $\fQ^{\d=\infty}_{g,n,1,1}\simeq \overline{\mathcal{M}}_{g,n+1}$.
\end{lemma}
\begin{proof}
Let $(C,L,s)\in \fQ^{\d=\infty}$. If $C$ is stable, then $(L,s)$ can be an arbitrary pair except that $s$ does not vanish at the nodes of $C$. If $C$ is not stable, then $C$ is quasi-stable with a unique rational bridge $E\simeq \PP^1$, and $s$ is uniquely determined by a section of $L|_E\simeq \sO_E(1)$ which does not vanish at the nodes on $E$.

Relativizing this, to each family $(\cC,\cL,s)$ of $\d=\infty$-stable  quasimaps parametrized by a scheme $S$, one can associate an extra section of $\cC \xrightarrow{}S$, which parametrizes the vanishing locus of the section $s$. Indeed, the vanishing locus is flat over $S$ by \cite[Lemma 31.18.9]{Stacks} and fiberwisely a smooth point, so the structure morphism to $S$ is an isomorphism.

As the universal curve over $\overline{\mathcal{M}}_{g,n}$ is $\overline{\mathcal{M}}_{g,n+1}$, we
thus have a morphism
\beq\label{17}\fQ^{\d=\infty}\lra \overline{\mathcal{M}}_{g,n+1}.\eeq

To get the inverse, consider a family $\cC\xrightarrow{}S$ of stable curves of genus $g$ with $n+1$ sections. Then the image of the first section is a relative Cartier divisor of degree 1, so corresponds to a pair $(\cL,s)$ of a line bundle $\cL\in \Pic(\cC)$ of fiberwise degree 1 and a nonzero section $s$. This gives us a family $(\cC,\cL,s)$ of $\d=\infty$-stable quasimaps and hence an inverse to \eqref{17}. \end{proof}

If furthermore $g=0$, the wall crossings of the moduli spaces of $\d$-stable quasimaps, which we call \emph{$\d$-wall crossings}, result in a new construction of $\overline{\mathcal{M}}_{0,n+1}$ from $\overline{\mathcal{M}}_{0,n}$.
Let $$\ell=\lfloor\frac{n-1}{2}\rfloor.$$ By \cite[Lemma 7.9]{CK}, there exist $\ell-1$ walls
$$\d_h=\frac{n-2}{h-1}-2, \quad \text{for }2\leq h\leq \ell.$$

Let $\d_h^+ \in (\d_{h-1}, \d_h)$, $\d_2^+:=\infty$ and $\d_{\ell+1}^+:=0^+$. Let $\pi:\overline{\mathcal{M}}_{0,n+1}\xrightarrow{} \overline{\mathcal{M}}_{0,n}$ denote the forgetful morphism which forgets the first marking.
\begin{theorem} \label{26} Suppose $g=0$ and $m=d=1$. Let $\fQ^{\d}=\fQ^{\d}_{0,n,1,1}$ and $\ell=\lfloor\frac{n-1}{2}\rfloor$. Then the $\d$-wall crossings factorize the morphism $\pi$ as
\beq\label{22}
\overline{\mathcal{M}}_{0,n+1}\cong \fQ^{\d=\infty}\xrightarrow{~\fq_2~ } \fQ^{\d_3^+}\xrightarrow{~\fq_3~ }\cdots \xrightarrow{~\fq_{\ell-1}~ }\fQ^{\d_\ell^+} \xrightarrow{~\fq_\ell~} \fQ^{\d=0^+}\xrightarrow{~\fp~}\overline{\mathcal{M}}_{0,n}~\eeq
where
\begin{enumerate}
    \item for $2\leq h \leq \ell$, the morphism $\fq_h$ at the wall $\d_h$ is the blowup along a disjoint union of $\binom{n}{h}$ copies of $\overline{\mathcal{M}}_{0,h+1}\times \overline{\mathcal{M}}_{n-h+1}$,
    \item the last morphism $\fp$ is the forgetful morphism. Moreover,
    \begin{enumerate}
        \item if $n=2\ell+1$, then $\fp$ is a $\PP^1$-bundle;
        \item if $n=2\ell+2$, then $\fp$ is
             the blowup of a $\PP^1$-bundle over $\overline{\mathcal{M}}_{0,n}$ along a disjoint union of $\frac12\binom{n}{\ell+1}$ copies of $\overline{\mathcal{M}}_{0,\ell+2}\times \overline{\mathcal{M}}_{0,\ell+2}$.
    \end{enumerate}
\end{enumerate}
All $\fq_h$ and $\fp$ are $S_n$-equivariant.
\end{theorem}
\begin{proof} The identification of the forgetful morphism $\fQ^{\d=\infty}\xrightarrow{} \overline{\mathcal{M}}_{0,n}$ with $\pi$ follows from Lemma~\ref{Lem.delta.infty}. Now the assertion is a special case of \cite[Theorem 7.11]{CK} for $m=1$.
\end{proof}
\begin{remark} \label{rmk:delta0}
For each $I\subset [n]:=\{1,2,\cdots, n\}$ with $2\leq |I|\leq n-2$, we denote by $D_{I}=D_{I^c}$ the boundary divisor of $\overline{\mathcal{M}}_{0,n}$ which generically parametrizes nodal curves with two components, one with markings by $I$ and the other with markings by $I^c$.

Let $C\in \overline{\mathcal{M}}_{0,n}$. Unless $n$ is even and $C\in D_{I}$ for some $I\subset [n]$ with $|I|=\frac{n}{2}$, $C$ has a unique irreducible component whose complement has no connected subcurves with $\ge\frac{n}{2}$ markings. We call this component the \emph{balanced component} of $C$.

When $n=2\ell+2$ is even and $C\in D_{I}$ for some $I\subset [n]$ with $|I|=|I^c|$, $C$ has the unique node separating markings by $I$ and $I^c$. We also call the two irreducible components containing this node the \emph{balanced components} of $C$.

By the stability condition, for a quasimap $(C, L, s)\in \fQ^{\d=0^+}$, the line bundle $L$ can have degree one only on balanced components. So, the moduli $\fQ^{\d=0^+}$ parametrizes the balanced components of the universal curve over $\overline{\mathcal{M}}_{0,n}$.

When $n$ is odd, the morphism $\fp$ is a $\PP^1$-bundle. When $n$ is even, the morphism $\fp$ then can be described as the blowup of a $\PP^1$-bundle over $\overline{\mathcal{M}}_{0,n}$ which parametrizes a consistent choice of a balanced component of the universal curve over the boundary divisor $\bigsqcup_{I\subset [n], |I|=n/2}D_{I}$.
\end{remark}

\subsection{Wall crossings for $d=1$ and $m\ge 2$} \label{Sec.Epsilonwc}
In this subsection, we generalize Theorem \ref{26} to the case of $m\ge 2$. As we saw in \S\ref{18}, the moduli space $\mzn(\PP^{m-1},1)$ of stable maps to $\PP^{m-1}$ is an open substack of $\fQ=\fQ_{0,n,1,m}$.
By wall crossings in $\fQ$, the forgetful morphism
\beq\label{19} \mzn(\PP^{m-1},1)\lra \mzn\eeq
is now factorized as follows.

\begin{theorem} \cite[Theorem 7.11]{CK} \label{24}
Suppose $m\geq2$, $d=1$ and $g=0$. Let $\fQ^{\d}=\fQ^{\d}_{0,n,1,m}$ and $\ell=\lfloor\frac{n-1}{2}\rfloor\geq 1$.
Then \eqref{19} is the composition of morphisms
\beq\label{20} \mzn(\PP^{m-1},1)\xrightarrow{~\fc~}\fQ^{\e=0^+}\xrightarrow{~\fc_0~} \fQ^{\d=\infty}\xrightarrow{~\fq_2~ } \cdots \xrightarrow{~\fq_\ell~} \fQ^{\d=0^+}\xrightarrow{~\fp~}\overline{\mathcal{M}}_{0,n}~\eeq
where
\begin{enumerate}
    \item if $m=2$, $\fc$ is an isomorphism;\\
    if $m>2$, $\fc$ is the blowup along a $\PP^{m-1}$-bundle over $\overline{\mathcal{M}}_{0,n+1}$,
    \item $\fc_0$ is the blowup along $n$ copies of  a $\PP^{m-1}$-bundle over $\overline{\mathcal{M}}_{0,n}$,
    \item for $2\leq h \leq \ell$, the morphism $\fq_h$ is the blowup along a disjoint union of $\binom{n}{h}$ copies of a $\PP^{m-1}$-bundle over $\overline{\mathcal{M}}_{0,h+1}\times \overline{\mathcal{M}}_{n-h+1}$,
    \item
    \begin{enumerate}
    \item if $n=2\ell+1$, $\fp$ is a $\PP^{2m-1}$-bundle;
       \item if $n=2\ell+2$, $\fp$ is the blowup of a $\PP^{2m-1}$-bundle over $\overline{\mathcal{M}}_{0,n}$ along a disjoint union of $\frac{1}{2}\binom{n}{\ell+1}$ copies of a $\PP^{m-1}$-bundle over $\overline{\mathcal{M}}_{0,\ell+2}\times \overline{\mathcal{M}}_{0,\ell+2}$.
    \end{enumerate}
\end{enumerate}
All morphisms $\fc$, $\fc_0$, $\fq_h$ and $\fp$ are equivariant with respect to the action of $S_n$ permuting the marked points.
\end{theorem}

\begin{remark}
(1) It is well-known that $\Mbar_{0,n}(\PP^1,1)$ is isomorphic to the Fulton-MacPherson compactified configuration space $\PP^1[n]$ of $n$ points on $\PP^1$ (see \cite{FM} for definition). Therefore when $m=2$, Theorem \ref{24} gives a new construction of $\PP^1[n]$.

(2) $\fQ^{\e=0^+}$ in \eqref{20} is the moduli space of $\e=0^+$-stable quasimaps. See \cite{MOP, Tod, CK} for the $\e$-stability.
\end{remark}

\bigskip

\section{Preliminaries on representations of the symmetric groups}\label{sec:prelim}

In this section, we collect necessary facts on the representation theory of  $S_n$.
The main reference is \cite{Mac}.

\subsection{Characteristic map}
Let $R$ be the ring of representations of the symmetric groups and let $\Lambda:=\varprojlim\Z[x_1,\cdots,x_n]^{S_n}$ be the ring of symmetric functions. It is well known that $\Lambda \otimes \QQ =\QQ[p_1, p_2, \cdots]$, where $p_n= \sum x_i^n$ is the $n$-th power sum. For a representation $V$ of $S_n$, we define its character by
\[\ch_n(V)= \frac{1}{n!}\sum_{\sigma\in S_n} \mathrm{Tr}_V(\sigma) p_{\sigma}, \]
where $\mathrm{Tr}_V(\sigma)$ is the trace of the action of $\sigma$ on $V$ and $p_{\sigma}=p_{\lambda_1}\cdots p_{\lambda_r}$ if $\sigma$ has the cycle type $\lambda=(\lambda_1, \cdots, \lambda_r)$. For example if $\sigma=(1,3,5)(2,4)$, then $p_\sigma= p_3 p_2$.

Recall that irreducible representations of $S_n$ bijectively correspond to partitions of $n$. Let $V_\lambda$ be the irreducible representation corresponding to a partition $\lambda$. Then we have $\ch_n(V_\lambda)= s_\lambda\in \Lambda$, the Schur function.

The characteristic map $\ch:R\xrightarrow{}\Lambda$ is a ring isomorphism, that is,
\begin{equation}
    \label{eq:charmult}
    \ch_i(V)\ch_j(W)=\ch_{i+j}\big(\Ind^{S_{i+j}}_{S_i\times S_j}V\otimes W\big),
\end{equation}
where $\Ind$ denotes the induced representation. Hence, we sometimes do not distinguish a representation $V$ of $S_n$ and its character $\ch_n(V)$. For example, $V.W$ denotes $\Ind^{S_{i+j}}_{S_i\times S_j}V\otimes W$ for $S_i$-representation $V$ and $S_j$-representation $W$.

\subsection{Plethysm}
In $\Lambda$, there is another multiplication, called \emph{plethysm}. For two symmetric functions $f$ and $g$, the plethysm is denoted by $f \circ g$ and it is uniquely defined by the following properties:
\begin{enumerate}
    \item $f\circ p_n=p_n\circ f= f(x_1^n, x_2^n, \cdots )$,
    \item $(f+g)\circ h= f\circ h + g\circ h$,
    \item $(fg)\circ h= (f\circ h)(g\circ h)$
\end{enumerate}
for symmetric functions $f, g$ and $h$. Note that $p_1=s_{(1)}$ acts as a two sided identity. The plethysm is not commutative in general, but it is associative.
For more details, see \cite[I. 8]{Mac}.

On the representation ring $R$, we can define the corresponding product. Namely, given an $S_n$-representation $V$ and $S_m$-representation $W$, the plethysm is defined by
\[W\circ V = \Ind^{S_{mn}}_{(S_n)^m\rtimes S_m} \big(W\otimes T^m(V)\big),\]
where $T^m(V)$ is the $m$-th tensor power of $V$ and $S_m$ acts on $T^m(V)$ by permuting its factors. Then it is well known that
\begin{equation}
    \label{eq:charple}
    \ch(W\circ V) = \ch(W)\circ \ch(V).
\end{equation}
In this paper, we will use plethysm when $W$ is the trivial representation ($\ch(W)=s_{(m)}$) or the alternating representation ($\ch(W)=s_{(1^m)}$). We also do not distinguish a representation and its character when we use plethysm.

\begin{example} \label{Ex.Plethysm}
Let $V$ be a representation of $S_n$. Consider a natural action of $S_2$ on $V.V$ switching two copies of $V$. Then $s_{(2)}\circ V$ and $s_{(1,1)}\circ V$ are the symmetric and the antisymmetric parts of $V.V$ respectively. From the defining properties of plethysm, one can easily check the following:
\begin{enumerate}
    \item $s_{(2)}\circ V= \Ind^{S_{2n}}_{(S_n\times S_n)\rtimes S_2}V\otimes V=(V.V)^{S_2}$ and
    \item $s_{(1,1)}\circ V= V.V-s_{(2)}\circ V$.
\end{enumerate}
Furthermore, we have an identity of symmetric functions (\cite[I, (8.8)]{Mac}) $$s_{(1,1)}\circ (f+g) = fg + s_{(1,1)}\circ f + s_{(1,1)}\circ g.$$
Hence, if $V$ has a decomposition $V=\bigoplus^r_{i=1} V_i$ into $S_n$-representations $V_i$, then we have
\begin{equation} \label{Eq.AntiSymm}
    s_{(1,1)}\circ \left(\bigoplus_{i=1}^r V_i\right)=\bigoplus_{i<j}V_i.V_j\oplus \bigoplus_{i=1}^r \left(s_{(1,1)}\circ V_i\right).
\end{equation}
\end{example}

\medskip

\subsection{Permutation representation}
\begin{definition}
    A representation $V$ of $S_n$ is a \emph{permutation representation} if there exists an $S_n$-invariant basis of $V$. If further such a basis is one orbit of $S_n$, then $V$ is said to be \emph{transitive}.
\end{definition}
Every permutation representation decomposes into a direct sum of transitive ones. While such a decomposition is not necessarily unique, every decomposition of a given permutation representation has the same number of transitive components because it is equal to the dimension of the invariant subspace.

Every transitive permutation representation of $S_n$ is isomorphic to the induced representation of the trivial representation of some subgroup $H\subset S_n$, which we denote by
$$U_H:=\Ind^{S_n}_{H}e$$
where $e$ is the trivial $H$-representation. Equivalently, $U_H$ is isomorphic to the permutation representation of $S_n$ generated by the cosets in $S_n/H$ with the $S_n$-action given by left multiplication. This is also isomorphic to the $H$-invariant subspace of the regular representation of $S_n$ with respect to the action of $H$ induced by the right multiplication.

We introduce the following notations.
\begin{definition} \label{Def.TransPermRep}
\begin{enumerate}
    \item When $n=\sum a_i$, from $\prod S_{a_i}\subset S_n$, we write $$U_{a_1,\cdots,a_k}:=U_{\prod S_{a_i}}.$$
    \item Let $n=ma+\sum a_i$. Consider $((S_a)^m\rtimes S_m)\times \prod S_{a_i}\subset S_n$, where $S_m$ permutes factors of $(S_a)^m$. We write \[U_{a^m,a_1,\cdots,a_k}^{S_m}:=U_{((S_a)^m\rtimes S_m)\times \prod S_{a_i}},\]
    where $a^m$ denotes the sequence of $m$ copies of $a$.
\end{enumerate}
\end{definition}
The representation $U_{a^m,a_1,\cdots,a_k}^{S_m}$ is the invariant subspace of $U_{a^m,a_1,\cdots,a_k}$ under an action of $S_m$, which commutes with the action of $S_n$. Note that by \eqref{eq:charmult} and \eqref{eq:charple}, we have
$$\ch(U_{a_1,\cdots,a_k})=\prod_{i=1}^k s_{(a_i)}$$ and $$\ch(U_{a^m,a_1,\cdots,a_k}^{S_m})=(s_{(m)}\circ s_{(a)}).\prod_{i=1}^k s_{(a_i)}.$$

\bigskip

\section{Representations on the cohomology of $\fQ^\delta$} \label{Sec.13}

In this paper, we will only deal with varieties or stacks $X$ whose cycle class map is an isomorphism  $A^k(X)\cong H^{2k}(X)$ for each $k$.
Recall that the cohomology $H^*([Y/G])$ of a quotient stack is the equivariant cohomology $H^*_G(Y)=H^*(EG\times_G Y)$.

We define the \emph{$S_n$-equivariant Poincar\'{e} polynomial} by
\beq\label{21}
P^{S_n}_X(t)=\sum_{k}\ch_n(A^k(X))t^k\in \Lambda[t].\eeq
Sometimes we will simply write $P^{S_n}_X(t)$ as $P_X(t)$ when the group $S_n$ is obvious from the context.
Let
\beq\label{64} P_n:=P^{S_n}_{\overline{\mathcal{M}}_{0,n}},\quad Q_n:=P^{S_n}_{\overline{\mathcal{M}}_{0,n+1}}\quad \in \Lambda[t]\eeq
be the $S_n$-equivariant Poincar\'e polynomials of $\overline{\mathcal{M}}_{0,n}$ and
$\overline{\mathcal{M}}_{0,n+1}$ respectively.
The $S_n$-action on $\overline{\mathcal{M}}_{0,n+1}$
in the rest of this paper consists of the permutations of the last $n$ markings with the first marking fixed.

\subsection{$\d$-wall crossings and the blowup formula}
By Theorem~\ref{26}, the blowup formula (\cite[p.605]{GrHa} or \cite[Theorem 6.7]{Ful}) yields a proof of \eqref{13}.

\begin{proposition} \label{Prop.SnRep} For $n\geq 3$, let $\ell=\lfloor\frac{n-1}{2}\rfloor.$
Then we have    \beq\label{45}
    Q_n=P_{\fQ_{0,n,1,1}^{\d=0^+}}+t\sum^{\ell}_{h=2}Q_hQ_{n-h}.\eeq
For $m\ge 2$, we have
\beq\label{51}
P_{\mzn(\PP^{m-1},1)}=P_{\fQ_{0,n,1,m}^{\d=0^+}}+
\eeq \[
\frac{1-t^m}{1-t}\left(\frac{t-t^{m-1}}{1-t} Q_n+\frac{t-t^m}{1-t} s_{(1)}.Q_{n-1}+\frac{t-t^{m+1}}{1-t} \sum_{h=2}^{\ell}Q_hQ_{n-h}\right).
\]
\end{proposition}

\begin{proof} The blowup centers for $\fq_h$ in \eqref{22} are $D_{I}\cong \overline{\mathcal{M}}_{0,\{\bullet\}\cup I}\times \overline{\mathcal{M}}_{0,\{\bullet\}\cup I^c}$ for each $I\subset [n]=\{1,2,\cdots,n\}$ with $|I|=h$.
Now \eqref{45} (resp. \eqref{51}) follows directly from the blowup formula and Theorem \ref{26} (resp. Theorem \ref{24}) for $m=1$ (resp. $m\ge 2$). \end{proof}

By \eqref{45}, we can calculate $P_{\fQ^{\d=0^+}_{0,n,1,1}}$, once we know $Q_n$ which will be calculated in \S\ref{Sec.11}.
The remainder of this section is devoted to a proof of \eqref{14} which enables us to calculate $P_n$ from $P_{\fQ^{\d=0^+}}$.

\subsection{GIT for $\mzn$}
In this subsection, we give a geometric description of $\fQ^{\delta=0^+}$ by GIT.

Let $G=\PGL_2(\CC)$ and $$\PP^1[n]=\mzn(\PP^1,1)$$ denote the Fulton-MacPherson compactification of the configuration space of $n$ points on $\PP^1$ (cf. \cite{FM}).
The obvious action of $G$ on $\PP^1$ induces an action of $G$ on $\PP^1[n]$ which commutes with that of $S_n$.
We first recall the following important fact.
\begin{theorem}\label{23} \cite[Theorem 4.1]{KM}
There is a suitable linearization of the $G$-action on $\PP^1[n]$ such that the GIT quotient
of $\PP^1[n]$ by $G$ and its partial desingularization are isomorphic to $\mzn$, i.e.
\beq\label{37} \PP^1[n]/\!/G\cong \widetilde{\PP^1[n]}/\!/G \cong\mzn.\eeq
\end{theorem}

Kirwan's partial desingularization $\widetilde{\PP^1[n]}/\!/G$ of $\PP^1[n]/\!/G$ is defined as follows (cf. \cite{Kir2}).
Let \beq\label{27} Y=(\PP^1[n])^{ss}\supset (\PP^1[n])^{s}=Y^s\eeq
denote the semistable and stable parts of $\PP^1[n]$.
For $n$ odd, we have $Y^{ss}=Y^s$ and there is nothing to do.
Now let $n$ be even.

For $I\subset [n]=\{1,2,\cdots,n\}$, let $\cD_I\subset \PP^1[n]$ denote the boundary divisor which generically parametrizes prestable curves $C=C_1\cup C_2$ with two irreducible components glued at a node $\bullet$ such that
$$C_1\in \overline{\cM}_{0,I\cup \{\bullet\}} \and C_2\in \PP^1[I^c\cup\{\bullet\}].$$
Here $\overline{\cM}_{0,I\cup \{\bullet\}} \cong \overline{\cM}_{0,|I|+1}$ is the moduli space of $|I|+1$ pointed stable curves of genus 0 with markings by $I\cup \{\bullet\}$ and
$\PP^1[I^c\cup\{\bullet\}]$ is likewise.
Then it is straightforward to see from \cite{KM} that
$$Y=\PP^1[n]-\bigcup_{|I|>n/2} \cD_I \supset \PP^1[n]-\bigcup_{|I|\ge n/2} \cD_I=Y^s.$$

The locus of infinite stabilizers in $Y$ is
$$Z:=\bigsqcup_{|I|=\frac{n}{2}}(\cD_I\cap \cD_{I^c})$$
and the blowup $\widetilde{Y}=\mathrm{bl}_ZY$ of $Y$ along $Z$ has no strictly semistable point with respect to a suitable linearization (cf. \cite[\S2.2]{KM}) by \cite{Kir2}.
The partial desingularization of $\PP^1[n]/\!/G$ is the quotient
\beq\label{30}\widetilde{\PP^1[n]}/\!/G =\widetilde{Y}^s/G\eeq
of the stable part $\widetilde{Y}^s$ by $G$ and
\beq\label{35}
\widetilde{Y}-\widetilde{Y}^s=\bigsqcup_{|I|=\frac{n}{2}}\widetilde{\cD}_I\eeq
is the disjoint union of the proper transforms $\widetilde{\cD}_I$ of $\cD_I$.
Since the codimension of $Z$ is two and $\cD_I$ are divisors, $\widetilde{\cD}_I\cong \cD_I$.

Using the above geometry, we can compute the $S_n$-equivariant Poincar\'e polynomial of the stack $[Y/G]$.
\begin{proposition}\label{28}
For $n\ge 4$ even, the $S_n$-equivariant Poincar\'e polynomial $P_{[Y/G]}$ of the stack $[Y/G]$ is
\beq\label{29}
P_{[Y/G]}=P_n+\frac{t}{1+t}s_{(1,1)}\circ Q_{\frac{n}{2}}+\frac{t^2}{1-t^2}Q^2_{\frac{n}{2}}.
\eeq
\end{proposition}
\begin{proof}
By the blowup formula, we find that
\beq\label{33}
P_{[\widetilde{Y}/G]}=P_{[Y/G]}+tP_{[Z/G]}.\eeq
Let $E$ be the exceptional divisor in $\widetilde{Y}$. Then its GIT quotient
$$E/\!/G=E^s/G=\bigsqcup_{|I|=\frac{n}{2}, 1\in I} \overline{\cM}_{0,I\cup\{\bullet\}}\times \overline{\cM}_{0,I^c\cup\{\bullet\}}$$
and hence
\beq\label{31}P_{E^s/G}=s_{(2)}\circ Q_{\frac{n}{2}}.\eeq
On each fiber $\PP^1$ of $E\to Z$, there are two unstable points and we find
$$E-E^s=\bigsqcup_{|I|=\frac{n}{2}} \overline{\cM}_{0,I\cup\{\bullet\}}\times \overline{\cM}_{0,I^c\cup\{\bullet\}}$$
 so that
 \beq\label{32} P_{[(E-E^s)/G]}=\frac{1}{1-t}Q^2_{\frac{n}{2}}.\eeq
Since the stratification $E=E^s\sqcup(E-E^s)$ is equivariantly perfect by \cite{Kir2} and the codimension of $E-E^s$ is one,
we have
$$P_{[E/G]}=s_{(2)}\circ Q_{\frac{n}{2}}+\frac{t}{1-t}Q^2_{\frac{n}{2}}.$$
Since $E\to Z$ is a $\PP^1$-bundle, we thus have
$$P_{[Z/G]}=\frac{1}{1+t}\left( s_{(2)}\circ Q_{\frac{n}{2}}+\frac{t}{1-t}Q^2_{\frac{n}{2}} \right).$$ and by \eqref{33} this gives us
\beq\label{34} P_{[\widetilde{Y}/G]}=P_{[Y/G]}+\frac{t}{1+t}\left( s_{(2)}\circ Q_{\frac{n}{2}}+\frac{t}{1-t}Q^2_{\frac{n}{2}} \right).\eeq

By \eqref{35}, \eqref{37} and \eqref{30}, we have
\beq\label{36}
P_{[\widetilde{Y}/G]}=P_{[\widetilde{Y}^s/G]}+\frac{t}{1-t}Q^2_{\frac{n}{2}}
=P_n+\frac{t}{1-t}Q^2_{\frac{n}{2}}.
\eeq
Combining \eqref{34} and \eqref{36}, we have
$$ P_{[Y/G]}=P_n+\frac{t}{1-t}Q^2_{\frac{n}{2}}-\frac{t}{1+t}\left( s_{(2)}\circ Q_{\frac{n}{2}}+\frac{t}{1-t}Q^2_{\frac{n}{2}} \right)$$
which gives us \eqref{29} by Example \ref{Ex.Plethysm} (2).
\end{proof}

\subsection{Geometry of $\fQ^{\delta=0^+}$}
The goal of this subsection is to prove the following, by a GIT description of $\fQ_{0,n,1,m}^{\delta=0^+}$ for $m\ge 1$.
\begin{proposition} \label{25}
For $n\ge 3$,
\beq\label{38}
P_{\fQ_{0,n,1,m}^{\d=0^+}}=\frac{1-t^{2m}}{1-t}P_{[Y/G]}-t^{m+1}\frac{1-t^m}{(1-t)^2}Q^2_{\frac{n}{2}}\eeq
where $Q_{\frac{n}{2}}$ is set to be 0 for $n$ odd.
\end{proposition}
The case of odd $n$ is a direct consequence of Theorem \ref{26} (2a) and Theorem \ref{24} (4a).

Now let $n=2\ell+2$ be even. Let
$$X=Y\times \PP^{2m-1}=(\PP^1[n])^{ss} \times \PP H^0(\PP^1,\sO_{\PP^1}(1)^{\oplus m})$$
on which $G$ acts diagonally. The linearization $\sO_Y(1)$ of $Y$ induces a linearization $\sO_Y(1)\boxtimes \sO_{\PP^{2m-1}}(\e)$ for $\e>0$ sufficiently small.
The Hilbert-Mumford criterion then tells us that
\beq\label{39}
X^s=X-\bigsqcup_{|I|=\frac{n}{2}}(\cD_I\times \PP^{m-1})\eeq
and we have a geometric quotient $X^s/G$.

For $(f,\a)\in X^s$ with $(f:C\to \PP^1)\in Y\subset \mzn(\PP^1,1)$, it is straightforward to check that
$$(C, L=f^*\sO_{\PP^1}(1),s=f^*\a)$$
is $\delta=0^+$-stable. Conversely, given $(C,L,s)\in\fQ_{0,n,1,m}^{\d=0^+}$, the linear system $|L|$ gives us a morphism $f:C\to \PP^1$ and $s$ induces $\a\in H^0(\sO_{\PP^1}(1)^{\oplus m})$, such that $(f,s)\in X^s$.
Upon relativizing these, we obtain a proof of the following.
\begin{proposition}\label{40} For $n$ even,
$$\fQ^{\d=0^+}_{0,n,1,m}\cong X^s/G.$$
\end{proposition}

Now we can prove Proposition \ref{25} for $n$ even.
\begin{proof}[Proof of \eqref{38}]
Since the stratification $X=X^s\sqcup (X-X^s)$ is equivariantly perfect by \cite{Kir, Kir2}, we have
\beq\label{41}
P_{\fQ_{0,n,1,m}^{\d=0^+}}=P_{X^s/G}=\frac{1-t^{2m}}{1-t} P_{[Y/G]} - t^{m+1}\frac{1}{1-t} \frac{1-t^m}{1-t}Q^2_{\frac{n}{2}}\eeq
since the codimension of $\cD_I\times \PP^{m-1}$ in $X$ is $m+1$ for each $I$.
\end{proof}

\subsection{$S_n$-representations on $A^*(\fQ^{\d=0^+})$}

Propositions \ref{28} and \ref{25} prove the following.
\begin{proposition}\label{PMindel} For $n\ge 3$ and $m\ge 1$,
    \beq\label{43}
    P_{\fQ^{\d=0^+}_{0,n,1,m}}=\frac{1-t^{2m}}{1-t}P_{n}+ts_{(1,1)}\circ \frac{1-t^m}{1-t}Q_\frac{n}{2},\eeq
    where $Q_{\frac{n}{2}}$ is set to be zero for $n$ odd.
\end{proposition}
\begin{proof}
For an odd $n$, the proposition follows from \eqref{38} since $Y/G=\mzn$.
For $n$ even, by \eqref{29} and \eqref{38}, we have
\begin{equation*}
        \begin{split}
            P_{\fQ^{\d=0^+}}=\frac{1-t^{2m}}{1-t}P_{n}+ \frac{t-t^{2m+1}}{1-t^2}s_{(1,1)}\circ Q_\frac{n}{2}+\frac{t^2(1-t^{m-1})(1-t^m)}{(1-t)(1-t^2)}Q_{\frac{n}{2}}^2.
        \end{split}
    \end{equation*}
    By applying \eqref{Eq.AntiSymm} to $(1+t+\cdots+t^{m-1})Q_\frac{n}{2}$, the proposition follows.
\end{proof}

When $m=1$, \eqref{43} is \eqref{14}.
\begin{corollary} \label{CMindel} For $n\geq3$, we have
\beq\label{44}
P_{\fQ_{0,n,1,1}^{\d=0^+}}=
(1+t)P_n+ts_{(1,1)}\circ Q_{\frac{n}{2}}
\eeq
where $ Q_{\frac{n}{2}}$ is set to be 0 for $n$ odd.
\end{corollary}

By \eqref{45} and \eqref{44}, we can calculate $P_n$ from $Q_n$.
\begin{theorem}\label{46} For $n\ge 3$,
\beq\label{47}
P_n=\frac{1}{1+t} Q_n-\frac{t}{1+t}\left(\sum_{h=2}^\ell Q_hQ_{n-h}+s_{(1,1)}\circ Q_{\frac{n}{2}}
\right)\eeq
where $\ell=\lfloor\frac{n-1}{2}\rfloor$ and we set $Q_{\frac{n}{2}}=0$ if $n$ is odd.
\end{theorem}

In the next section, we will calculate $Q_n$.

\bigskip

\section{$S_n$-representations on the cohomology of $\overline{\mathcal{M}}_{0,n+1}$} \label{Sec.11}

By Theorem \ref{46}, in order to compute $P_n$, it suffices to compute $Q_n$.
In this section, we calculate $Q_n$ in terms of a purely combinatorial gadget, called \emph{weighted rooted trees}  (Proposition \ref{Prop.RepKap}). The idea is to apply the blowup formula (Lemma \ref{Lem.BlowupFormula}) for transversal subvarieties to the Kapranov construction.

\subsection{Blowup formula for transversal subvarieties}
Note that the blowup centers for Kapranov's blowups are usually not smooth and hence we cannot directly apply the usual blowup  formula \cite[p.605]{GrHa} or \cite[Theorem 6.7]{Ful}.
In this section, we recall the blowup formula for transversal subvarieties.

Let $X$ be a smooth variety and $$Y=\bigcup_{a \in A}Y_a$$ be a transversal union of smooth irreducible subvarieties $Y_a$ of $X$, that is, every nonempty intersection
$$Y_{A'}:=\bigcap_{a \in A'}Y_a, \quad \text{for }A'\subset A$$
is transversal. For each subset $A'\subset A$, we let
    \[V_{A'}=\bigotimes_{a\in A'}V_{a}\]
where $V_a:=A^+(\PP^{\mathrm{codim}(Y_a)-1})$ is the cohomology of $\PP^{\mathrm{codim}(Y_a)-1}$ of positive degrees.

\begin{lemma} \cite[Proposition 6.1]{BM} \label{Lem.BlowupFormula}  Let $\tX$ be the blowup of $X$ along $Y=\bigcup_{a\in A}Y_a$. Then,
    \begin{equation}\label{Eq.BlowupFormula}
        A^*(\tX)=A^*(X)\oplus \bigoplus_{\emptyset \neq A'\subset A} A^*(Y_{A'})\otimes V_{A'},
    \end{equation}
    where $A'$ runs over the nonempty subsets of $A$ with $Y_{A'}\neq \emptyset$.
\end{lemma}
\begin{proof}
    $\tX$ is isomorphic to the iterated smooth blowup of $X$ along (the proper transforms of) $Y_a$. The assertion follows from a repeated application of the blowup formula for smooth irreducible centers.
\end{proof}

Note that if a finite group $G$ acts on $X$ and $Y$ is $G$-invariant, then there are natural actions of $G$ on $\tX$ and $A$ as $G$ permutes $Y_a$. In this case, \eqref{Eq.BlowupFormula} is $G$-equivariant.

The Kapranov map $\overline{\mathcal{M}}_{0,n+1}\xrightarrow{} \PP^{n-2}$ breaks into a sequence of blowup morphisms. The intermediate spaces and blowup centers are suitable Hassett moduli spaces of weighted pointed stable curves \cite{Has}. The weight of the first marking is fixed to be one. We call this marking the \emph{0-th marking}. The remaining $n$ markings are indexed by the set $[n]=\{1,2,\cdots,n\}$ and the symmetric group $S_n$ acts on the moduli spaces by permuting them. All the blowup morphisms are $S_n$-equivariant by construction, and so we can apply the blowup formula to compute the $S_n$-equivariant Poincar\'e polynomial. We will introduce the notion of weighted rooted trees to encode the combinatorial data of the blowup centers.

\subsection{Hassett moduli of weighted stable curves} \label{Subsubsec.Hassett}
In \cite{Has}, Hassett introduced the notion of \emph{weighted} stable curves which generalizes that of stable curves.
Recall that for each weight
\[\mathcal{A} = (\mathcal{A}(1), \cdots, \mathcal{A}(n)) \in (\Q\cap (0,1])^{\oplus n}, \]
a pair $(C,p_1, \cdots, p_n)$ of a prestable curve $C$ with (not necessarily distinct) smooth marked points $p_i\in C$ is called $\cA$-stable if
\begin{enumerate}
    \item $\omega_C(\sum_{i=1}^n \cA(i) p_i) $ is ample, and
    \item for any $p\in C$, $\sum_{p_i=p} \cA(i) \le 1$.
\end{enumerate}
Note that when $\cA(i)=1$ for all $i$, we recover the notion of the usual stable curve. Hassett proved that for each weight $\mathcal{A}\in (\Q\cap (0,1])^{\oplus n}$ with $\sum_{i=1}^n\mathcal{A}(i)+2g-2>0$ there exists a smooth projective \DM stack $\overline{\mathcal{M}}_{g,\mathcal{A}}$ parametrizing $\mathcal{A}$-weighted stable curves of genus $g$. If $g=0$, every such Hassett moduli space is a smooth projective variety.

When $\mathcal{A},\mathcal{B}\in (\Q\cap (0,1])^{\oplus n}$ are weights with $\mathcal{A}(i)\geq\mathcal{B}(i)$ for all $i$, we write $\mathcal{A}\geq \mathcal{B}$. For weights $\mathcal{A}\geq \mathcal{B}$, there exists the \emph{reduction morphism} $\rho_{\mathcal{A},\mathcal{B}}:\overline{\mathcal{M}}_{g,\mathcal{A}}\xrightarrow{}\overline{\mathcal{M}}_{g,\mathcal{B}}$ which stabilizes $\mathcal{A}$-stable curves with respect to the weight $\mathcal{B}$.
\begin{proposition} \cite{Has} \label{prop:reduction}
Reduction morphisms have the following properties.
\begin{enumerate}
    \item For weights $\mathcal{A}\geq \mathcal{B} \geq \mathcal{C}$, we have $\rho_{\mathcal{A}, \mathcal{C}}=\rho_{\mathcal{B},\mathcal{C}}\circ \rho_{\mathcal{A},\mathcal{B}}$.
    \item For weights $\mathcal{A}\geq \mathcal{B}$, there are only finitely many $t\in (0,1]$ such that
    \begin{equation} \label{Ineq.Hassett}
    \sum_{i\in I}\mathcal{A}(i)>\sum_{i\in I}(1-t)\mathcal{A}(i)+t\mathcal{B}(i)=1\geq \sum_{i\in I} \mathcal{B}(i)
    \end{equation}
    for some $I\subset [n] $ with $|I|\geq 3$.
    \item If there exists only one such $t$ in (2), then $\rho_{\mathcal{A}, \mathcal{B}}$ is the blowup along the transversal union of the loci $Z_I$ along which the $I$-markings collide, where $I$ runs over the subsets of $[n]$ satisfying \eqref{Ineq.Hassett} and $|I|\geq 3$. Moreover, $Z_I\simeq \overline{\mathcal{M}}_{g,\mathcal{B}_I}$ where $\mathcal{B}_I\in (\Q\cap (0,1])^{\oplus (n-|I|+1)}$ is the weight obtained from $\mathcal{B}$ by adding up $\{\mathcal{B}(i)\}_{i\in I}$ as one entry.
    \item The intersection of some of the loci $Z_{I_1},\cdots,Z_{I_r}$ in (3) is nonempty if and only if $I_j$ are mutually disjoint. Furthermore, in this case $\bigcap_j Z_{I_j}\simeq \overline{\mathcal{M}}_{g,\mathcal{B}_{I_1,\cdots,I_r}}$ where $\mathcal{B}_{I_1,\cdots,I_r}$ denotes the weight obtained from $\mathcal{B}$ by adding up $\{\mathcal{B}(i)\}_{i\in I_j}$ as one entry for each $j=1,\cdots,r$.
\end{enumerate}
\end{proposition}
By (1) and (2), every reduction morphism is the composition of the blowups in (3) where the intersections of the irreducible components of the blowup centers are again isomorphic to some Hassett moduli spaces.

In the rest of this section, we let $g=0$. We consider Hassett moduli spaces of weighted stable curves with $n+1$ marked points with the weight of the 0-th marking fixed to be one.

For a weight $\mathcal{A}\in (\Q\cap (0,1])^n$, we write
\begin{enumerate}
    \item[(i)] $a^r\cdot\mathcal{A}=\left(a,\cdots,a,\mathcal{A}(1),\cdots, \mathcal{A}(n)\right)\in (\Q\cap (0,1])^{n+r}$ for $r\geq 1$,
    \item[(ii)] $\mathcal{A}=a^n~$ if $\mathcal{A}=(a,\cdots,a)$,
    \item[(iii)] $\frac{1}{k}\mathcal{A}=\left(\frac{\mathcal{A}(1)}{k},\cdots,\frac{\mathcal{A}(n)}{k}\right)$ for $k\geq \max_i \mathcal{A}(i)$.
\end{enumerate}
Especially when $r=1$ in (i) we simply write as $a\cdot \mathcal{A}$ and consider $a$ as its 0-th entry for convenience.

\begin{example} \label{Ex.Reduction}
    Let $\mathcal{A}=(2,2,3,3,3,5)$. Then the reduction morphism \[\overline{\mathcal{M}}_{0,1\cdot\frac{1}{8}\mathcal{A}}\longrightarrow\overline{\mathcal{M}}_{0,1\cdot \frac{1}{9}\mathcal{A}}\]
    is the blowup along a transversal union of two subvarieties. One is the locus in $\overline{\mathcal{M}}_{0,1\cdot \frac{1}{9}\mathcal{A}}$ where the 3 marked points of weights $\frac{2}{9}, \frac{2}{9}, \frac{5}{9}$ collide and hence isomorphic to $\overline{\mathcal{M}}_{0,1^2\cdot(\frac{1}{3})^3}$. The other is the locus where the 3 marked points of weight $\frac{1}{3}$ collide and thus isomorphic to $\overline{\mathcal{M}}_{0,1^2\cdot (\frac{2}{9})^2\cdot (\frac{5}{9})}$. The intersection of these two is the locus $\overline{\mathcal{M}}_{0,3}=\mathrm{pt}$ where each triple collides.
\end{example}

We also recall the following two general facts.
\begin{remark}\label{Ex.Hassett} Let $\mathcal{A}, \mathcal{B}\in (\Q\cap (0,1])^n$ be weights.
\begin{enumerate}
    \item If $\mathcal{A}\geq \mathcal{B}$ and $|I|=2$ for every $I\subset [n]$ satisfying \eqref{Ineq.Hassett}, then the reduction morphism $\rho_{\mathcal{A}, \mathcal{B}}$ is an isomorphism \cite[Corollary 4.7]{Has}.
    \item Suppose that $\sum_i \mathcal{A}(i)>1\geq \sum_{i\neq j}\mathcal{A}(i)$ for all $1\leq j\leq n$. Then $\overline{\mathcal{M}}_{0,1\cdot \mathcal{A}}\simeq \PP^{n-2}$ \cite[\S6.2]{Has}.
\end{enumerate}
\end{remark}

\subsection{Kapranov's blowups}

By Kapranov's construction, we have a morphism
\beq\label{48}
\overline{\mathcal{M}}_{0,n+1}\lra \PP^{n-2}\eeq
which factors as the composition
\begin{equation} \label{Map.Kap}
    \overline{\mathcal{M}}_{0,n+1}\xrightarrow{\,\cong\,} \overline{\mathcal{M}}_{0,1\cdot(\frac{1}{2})^n}\longrightarrow \cdots \longrightarrow \overline{\mathcal{M}}_{0,1\cdot(\frac{1}{n-1})^n}\cong \PP^{n-2}
\end{equation}
of the reduction morphisms among the Hassett moduli spaces $\overline{\mathcal{M}}_{0,1\cdot(\frac{1}{k})^n}$.
To apply the blowup formula to each reduction morphism $$\overline{\mathcal{M}}_{0,1\cdot (\frac{1}{k-1})^n}\xrightarrow{}\overline{\mathcal{M}}_{0,1\cdot (\frac{1}{k})^n},$$
we have to consider closed subvarieties in $\overline{\mathcal{M}}_{0,1\cdot(\frac{1}{k})^n}$ which naturally correspond to partitions of $[n]$.

\begin{definition} \label{Def.Partition} Let $S$ be a finite set and let $k$ be an integer.
\begin{enumerate}
    \item A set $P:=\{I_1, \cdots ,  I_r\}$ of mutually disjoint nonempty subsets of $S$ is called a \emph{partition} of $S$ if $S=\sqcup_jI_j$. We call each $I_j$ a \emph{part} of $P$. When $r=1$, i.e. $P=\{S\}$, we call $P$ the \emph{trivial} partition of $S$.
    \item Let
    \[\hspace*{2em} P=\{I_{1,1}, \cdots , I_{1,m_1},I_{2,1}, \cdots , I_{2,m_2}, \cdots,I_{r,1}, \cdots , I_{r, m_r},J_{1}, \cdots J_s\}\] be a partition of $S$. Let
    \[P'=\{I_{1}, \cdots , I_{r},J_{1}, \cdots J_s\}\]
    be a partition of $S$ such that \[I_j=I_{j,1}\sqcup\cdots\sqcup I_{j,m_j}\] is a nontrivial partition of $I_j$ for each $j$. Then we say $P$ is a \emph{refinement} of $P'$ \emph{of type $\{m_1,\cdots, m_r\}$}. If further $r=1$, we say it is \emph{of a simple type}.
    \item Let $P,P'$ be as in (2). We write $P>_k P'$ if $m_j\geq 3$ and $|I_j|=k$ for all $j$. We write $P\geq_k P'$ if $P>_kP'$ or $P=P'$. We write $P>P'$ if $P>_kP'$ for some $k$.
\end{enumerate}
Note that when $P>_k P'$ is of type $\{m_1,\cdots,m_r\}$, there exist $r$ distinct partitions $P'_1,\cdots, P'_r \geq_k P'$ such that $P>_k P'_j$ is of a simple type $\{m_j\}$ for each $j$. These $P'_j$ are precisely the maximal ones among the partitions $P''$ satisfying $P>_k P''\geq_k P'$. They uniquely determine $P'$, and vice versa.
\end{definition}

Let $1\leq k<n$ be an integer and let $P=\{I_1,\cdots,I_r\}$ be a partition of $[n]$ with $\max_{j}|I_j|\leq k$. Consider the locus $\overline{\mathcal{M}}^k_{P}\subset \overline{\mathcal{M}}_{0,1\cdot(\frac{1}{k})^n}$ where the marked points indexed by $I_j$ collide for each $j$. Then
\[\overline{\mathcal{M}}^k_{P}\cong\overline{\mathcal{M}}_{0,1\cdot \frac{1}{k}(|I_1|,\cdots,|I_r|)}.\]
Hence when $\max_j|I_j|< k$, we have the reduction morphism $\rho^k_P$
\begin{equation} \label{Diag.Kap}
    \xymatrix{\overline{\mathcal{M}}^{k-1}_{P}\ar[r]^-{\rho^k_P} \ar[d] & \overline{\mathcal{M}}^k_{P}\ar[d]\\
\overline{\mathcal{M}}_{0,1\cdot(\frac{1}{k-1})^n}\ar[r]  & \overline{\mathcal{M}}_{0,1\cdot(\frac{1}{k})^n}
}
\end{equation}
which commutes with the reduction morphism in the bottom row where the vertical morphisms are the inclusions. Furthermore, when $P$ is a refinement of another partition $P'$ of type $\{m_1,\cdots,m_r\}$, we have $\overline{\mathcal{M}}^k_{P'}\subset \overline{\mathcal{M}}^k_P$ which is of codimension $\sum_j(m_j-1)$ and $\rho^k_{P'}$ is compatible with \eqref{Diag.Kap}.

From Proposition \ref{prop:reduction} (3) and (4), it follows that
\begin{enumerate}
\item if $\{P_a\,|\, a \in A\}$ is the collection of partitions with $P>_kP_a$ of simple type, then $\rho^k_P$ is the blowup along the transversal union $\bigcup_{a\in A}\overline{\mathcal{M}}^k_{P_a}$
\item for a nonempty $A'\subset A$, we have
\[\bigcap_{a\in A'}\overline{\mathcal{M}}^k_{P_{a}}=\overline{\mathcal{M}}^k_{P_{A'}}, \]
where $P_{A'}$ is the maximal partition satisfying $P_a >_k P_{A'}$ for all $a\in A'$.
The intersection is empty if there is no such partition $P_{A'}$. Hence nonempty intersections among $\{\overline{\mathcal{M}}^k_{P_a}~|~ a\in A\}$ correspond bijectively to the partitions $P'$ with $P >_k P'$.
\end{enumerate}

\begin{example}
    Let $P=\{I_1,\cdots,I_6\}$ be a partition of $[18]=\{1, \cdots, 18\}$ having $6$ parts
    \begin{equation*}
        \begin{split}
             I_1=\{&1,2\},~ I_2=\{3,4\},~ I_3=\{5,6,7\},~ I_4=\{8,9,10\},\\
            & I_5=\{11,12,13\},~ I_6=\{14,15,16,17,18\}
        \end{split}
    \end{equation*}
    of cardinalities $2,2,3,3,3,5$ respectively. Then we have $\overline{\mathcal{M}}^k_P\simeq \overline{\mathcal{M}}_{0,1\cdot \frac{1}{k}\mathcal{A}}$ for $\mathcal{A}=(2,2,3,3,3,5)$ and $k\geq 5$, and the reduction morphism
    \[\rho^9_P: \overline{\mathcal{M}}^8_P\longrightarrow \overline{\mathcal{M}}^9_P\]
    is a blowup as we saw in Example~\ref{Ex.Reduction}. The blowup center is the transversal union of two closed subvarieties $\overline{\mathcal{M}}^9_{P'_{\{1,2,6\}}},~ \overline{\mathcal{M}}^9_{P'_{\{3,4,5\}}}\subset \overline{\mathcal{M}}^9_P$, where
    \[P'_{\{i,j,k\}}:=\{I_{i}\sqcup I_{j}\sqcup I_{k}\}\cup \{ I_l ~|~l\neq i,j,k\}.\]
    There are three partitions $P'$ satisfying $P >_9 P'$. Two of them are of simple types, namely $P'=P'_{\{1,2,6\}},~P'_{\{3,4,5\}}$. The other is $P'=P'_{\{1,2,6\},\{3,4,5\}}:=\{I_1\sqcup I_2 \sqcup I_6,~ I_3\sqcup I_4 \sqcup I_5\}$, where the corresponding $\overline{\mathcal{M}}^9_{P'}$ is the nonempty intersection of the previous two, as we saw
    \[\overline{\mathcal{M}}^9_{P'_{\{1,2,6\}}}\cap~ \overline{\mathcal{M}}^9_{P'_{\{3,4,5\}}}=~\overline{\mathcal{M}}^9_{P'_{\{1,2,6\},\{3,4,5\}}}=\mathrm{pt}\]
    in Example~\ref{Ex.Reduction}.
\end{example}
\begin{remark}
    (1) In the above example, $P'_{\{2,3,4\}}$ satisfies $P >_8 P'_{\{2,3,4\}}$ while $P \not>_9 P'_{\{2,3,4\}}$ since $|I_1\sqcup I_3\sqcup I_4|=8$. Hence the corresponding subvariety $\overline{\mathcal{M}}^8_{P'_{\{2,3,4\}}}\subset \overline{\mathcal{M}}^8_P$ is an irreducible component of the blowup center of $\rho^8_P$, not of $\rho^9_P$.

    (2) Furthermore, $P'_{\{3,6\}}:=\{I_3\sqcup I_6\}\cup \{I_j~|~j\neq 3,6\}$ does not contribute to the blowup $\rho^9_P$ while $|I_3\sqcup I_6|=9$, since $\overline{\mathcal{M}}^9_{P'_{\{3,6\}}}\subset \overline{\mathcal{M}}^9_{P}$ is of codimension 1. This is the reason why we exclude the case $m_j=2$ for some $j$ when we consider the partial orders on partitions.
\end{remark}

Now the blowup formula for $\rho^k_P$ can be written in terms of partitions. For $P> P'$ of type $\{m_1,\cdots, m_r\}$, we define a graded vector space $V_{P,P'}$ by
\begin{equation} \label{Eq.BlowupFactor}
    V_{P,P'}:=\bigotimes^r_{j=1} A^+(\PP^{m_j-2})
\end{equation}
where $A^+(\PP^{m_j-2})\subset A^*(\PP^{m_j-2})$ denotes the subspace of positive degrees. We let $V_{P,P}:=\Q$.

\begin{lemma} \label{66} Let $k$ be an integer and let $P$ be a partition of $[n]$ with $\max_{I\in P}|I|<k<n$. Then there is an isomorphism of graded vector spaces
\[ A^*(\overline{\mathcal{M}}^{k-1}_{P})=\bigoplus_{P \geq_k P'} A^*(\overline{\mathcal{M}}^k_{P'})\otimes V_{P,P'}\]
where $P'$ runs over the partitions of $[n]$ satisfying $P\geq_k P'$.
\end{lemma}
\begin{proof}
    The lemma is a direct consequence of the blowup formula (Lemma~\ref{Lem.BlowupFormula}) applied to $\rho^k_P$. Note that for each $a\in A$ and $A'\subset A$,
    \[Y_a=\overline{\mathcal{M}}^k_{P_a},\quad Y_{A'}=\overline{\mathcal{M}}^k_{P_{A'}} \quad \text{ and }\quad V_{A'}=V_{P,P_{A'}},\]
    where $A$ parametrizes $P'$ such that $P>_k P'$ are of simple types.
\end{proof}

\subsection{Weighted rooted trees}
Let us introduce combinatorial objects called rooted trees which bijectively correspond to the boundary strata of $\overline{\mathcal{M}}_{0,\{0\}\cup [n]}\simeq \overline{\mathcal{M}}_{0,n+1}$ that contribute in applying the blowup formula to \eqref{Map.Kap}.
We also impose weight functions on the vertices of rooted trees to encode the cohomology degrees. Explicitly, the notion of a weighted rooted tree is equivalent to a descending sequence of partitions  (Lemma~\ref{Lem.DescendingPartition}).

Recall that a \emph{graph} $G$ is data $(F(G),Ver(G),\iota)$ consisting of a set $F(G)$ of \emph{flags}, a set $Ver(G)$ of \emph{vertices} and an involution $\iota:F(G)\xrightarrow{}F(G)$, equipped with an \emph{attaching map} $F(G) \xrightarrow{} Ver(G)$. An unordered pair of flags which forms a 2-cycle under $\iota$ is called an \emph{edge}, and a flag fixed by $\iota$ is called a \emph{leg}. The \emph{valency} $val(v)$ of a vertex $v$ is the number of flags attached to $v$. A graph is called a \emph{tree} if it is connected and simply connected, that is, every pair of vertices is connected by edges and there are no loops or cycles.

\begin{definition} \label{Def.RootedTree}
\begin{enumerate}
    \item A \emph{rooted tree} is a tree in which one leg is designated as the \emph{output}. The unique vertex where the output is attached is called the \emph{root} and denoted by $v_0$. The other legs are called \emph{inputs}.

    We assume that every rooted tree satisfies a condition that is more strict than the usual stability assumption: $val(v_0)\geq 3$ and $val(v)\geq 4$ for all vertices $v\neq v_0$.
    \item A \emph{weighted rooted tree} is a rooted tree $T$ equipped with a map
    \[w:Ver(T) \longrightarrow \Z\]
    such that $0\leq w(v_0)\leq val(v_0)-3$ and $1\leq w(v)\leq val(v)-3$ for $v\neq v_0$. Let $w(T):=\sum_{v\in Ver(T)}w(v)$ be the \emph{weight} of a tree $T$.
    \item For each rooted tree $T$, there are two canonical maps
    \begin{equation*}
        \begin{split}
            D_T:& ~Ver(T)\setminus \{v_0\} \longrightarrow 2^{[n]}\\
            (\text{resp. }d_T:& ~Ver(T)\setminus \{v_0\} \longrightarrow \Z)
        \end{split}
    \end{equation*}
    which assign to each vertex $v\neq v_0$ the set (resp. the number) of legs whose paths to $v_0$ pass through $v$.
\end{enumerate}
\end{definition}

We remark that the root is allowed to have weight 0 since the weight of the root will encode the cohomology degree of the base in the blowup formula. The weights of the other vertices will encode the cohomology degrees of intersections of the irreducible components of the blowup centers.
\begin{figure}[h]
    \centering
    \begin{tikzpicture} [scale=0.8,auto=left,every node/.style={scale=0.8}]
      \tikzset{Bullet/.style={circle,draw,fill=black,scale=0.5}}
      \node[Bullet] (n0) at (2,0.5) {};
      \node[Bullet] (n1) at (2,-0.5) {};
      \node[] (n2) at (0.5,-1.5) {};
      \node[] (n3) at (1.5,-1.5) {};
      \node[] (n4) at (2.5,-1.5) {};
      \node[] (n5) at (3.5,-1.5) {};

      \draw[black] (n0) -- (2,1.5);
      \draw[black] (n0) -- (n1);
      \draw[black] (n1) -- (n2);
      \draw[black] (n1) -- (n3);
      \draw[black] (n1) -- (n4);
      \draw[black] (n1) -- (n5);
      \draw[black] (n0) -- (0.5,-0.5);
      \draw[black] (n0) -- (3.5,-0.5);
      \draw[] (2,1.5) node[right] {the output};
      \draw (2.1,0.5) node[right] {$v_0$};
      \draw (2.1,-0.5) node[right] {$v$};
      \draw[] (0.4,-0.5) node {$1$};
      \draw[] (3.6,-0.5) node {$2$};
      \draw[] (0.4,-1.6) node {$3$};
      \draw[] (1.4,-1.6) node {$4$};
      \draw[] (2.6,-1.6) node {$5$};
      \draw[] (3.6,-1.6) node {$6$};
    \end{tikzpicture}
    \caption{A rooted tree with 6 inputs}
    \label{fig:RootedTree}
\end{figure}
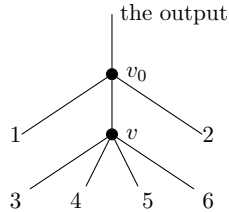

For example, the rooted tree in Figure \ref{fig:RootedTree} has 6 inputs and 2 vertices $v_0, v$ of valencies 4 and 5 respectively. We remark that our notion of a rooted tree is the same as the notion of a tree defined in \cite[\S1.1.1]{GK} except that we require a more restrictive condition on valencies of vertices.

\begin{definition}
Let $\sT^{o}_n$ (resp. $\sT_{n,p}$) denote the set of rooted trees with $n$ inputs (resp. weighted rooted trees with $n$ inputs and weight $p$).
\end{definition}

The notion of a rooted tree is in fact equivalent to a sequence of partitions of a set satisfying certain descending properties. To see this, let $\sP_n$ be the set of sequences
    \[P_2:=\{\{1\},\cdots,\{n\}\},~ P_3,~  \cdots ,~ P_{n-1}\]
of partitions of $[n]$ with $P_{k-1}\geq_k P_k$ for each $k$. Note that for such a sequence, we have $\max_{I\in P_k}|I|\leq k$ for each $k$.

\begin{lemma} \label{Lem.DescendingPartition}
    There is a canonical bijection
    \[\sT^{o}_n \longrightarrow \sP_n.\]
    Furthermore, for each $(P_2,\cdots,P_{n-1})\in \sP_n$ , $P_{k-1}\geq_k P_k$ is a refinement of type $\{val(v)-1\}_{v\in d_T^{-1}(k)}$ for all $k$ where $T$ is the corresponding rooted tree.
\end{lemma}
\begin{proof}
Let $T \in \sT^{o}_n$ be a rooted tree. Then one can assign a sequence of partitions $P_2=\{\{1\},\cdots, \{n\}\},~ P_3,~ \cdots ,~ P_{n-1}$ of $[n]$ by inductively obtaining $P_k$ from $P_{k-1}$ by uniting parts of $P_{k-1}$ to one whenever they unite to form a subset $D_T(v)\subset [n]$ for some $v\in d_T^{-1}(k)$ for each $k$. Together with the assumption on the valencies (Definition \ref{Def.RootedTree} (1)), this implies that either we have $P_{k-1}=P_k$ or $P_{k-1}>_k P_k$ is of type $\{val(v)-1\}_{v\in d^{-1}_T(k)}$.
\end{proof}

For example, the rooted tree in Figure \ref{fig:RootedTree} corresponds to the sequence of partitions
\[P_2=\{\{1\},\cdots,\{6\}\}=P_3>_4 P_4=\{\{1\},\{2\},\{3,4,5,6\}\}=P_5.\]
Note that the number of parts in $P_{n-1}$ is $val(v_0)-1\ge 2$.

\subsection{$\Res^{S_{n+1}}_{S_n}A^*(\overline{\mathcal{M}}_{0,n+1})$ as a sum over trees}
In this subsection, we compute $Q_n$ and prove \eqref{11}.

The symmetric group $S_n$ acts on the set $\sT_{n,p}$ by permuting the inputs. Let  $\sT_{n,p}/S_n$ denote the set of the $S_n$-orbits in $\sT_{n,p}$.
For $T\in \sT_{n,p}/S_n$, we let
$$U_T:=\Ind^{S_n}_{\Stab(T)}e$$
be the induced representation of $S_n$ from the trivial representation of $\Stab(T)$
where $\mathrm{Stab}(T)$ denotes the stabilizer group of $T$ as a weighted rooted tree.
Note that this is well-defined since $\Ind^{S_n}_H$ depends only on the conjugacy class of $H$.
For example, for any weighted rooted tree $T$ whose underlying rooted tree is  Figure~\ref{fig:RootedTree}, we have $U_T=U_{2,4}$.

\begin{proposition} \label{Prop.RepKap}
For $k\ge 0$, there is an isomorphism
\begin{equation}\label{eq:Mn+1}
    \mathrm{Res}^{S_{n+1}}_{S_n}A^k(\overline{\mathcal{M}}_{0,n+1})\cong \bigoplus_{T\in \sT_{n,k}/S_n}U_T
\end{equation}
of $S_n$-representations. In other words, the $S_n$-representation of $A^k(\overline{\mathcal{M}}_{0,n+1})$ is the permutation representation generated by the $S_n$-set $\sT_{n,k}$. Therefore, the $S_n$-equivariant Poincar\'e polynomial $Q_n=P^{S_n}_{\mzno}$ satisfies
\beq\label{52}
Q_n= \sum_{k\ge 0}\sum_{T\in \sT_{n,k}/S_n}\ch_n(U_T)t^k.\eeq
\end{proposition}

\begin{proof}
 We apply Lemma~\ref{Lem.BlowupFormula} to \eqref{Map.Kap}. Note that when $P$ is the most refined partition $P_2=\{\{1\}, \cdots, \{n\}\}$, we have $\overline{\mathcal{M}}^2_{P_2}\simeq\overline{\mathcal{M}}_{0,n+1}$ by Remark~\ref{Ex.Hassett} (1). Furthermore, for each $(P_2,\cdots,P_{n-1})\in \sP_n$, we have $\overline{\mathcal{M}}^{n-1}_{P_{n-1}}\simeq \PP^{val(v_0)-3}$ by Remark~\ref{Ex.Hassett} (2) where $v_0$ denotes the root of the rooted tree corresponding to $(P_i)$ under the bijection in Lemma~\ref{Lem.DescendingPartition}. A repeated application of Lemma \ref{66} results in the graded isomorphisms
\begin{equation} \label{Eq.SumOverTrees}
    \begin{split}
        A^*(\overline{\mathcal{M}}_{0,n+1})\cong &~\bigoplus_{(P_2,\cdots, P_{n-1})\in \sP_n}\Big(A^*(\overline{\mathcal{M}}^{n-1}_{P_{n-1}})\otimes \bigotimes_{i=3}^{n-1} V_{P_{i-1},P_{i}}\Big)\\
        \cong &~\bigoplus_{T\in \sT^{o}_n}\Big(A^*(\PP^{val(v_0)-3})\otimes \bigotimes_{v\in Ver(T)\setminus \{v_0\}}A^+(\PP^{val(v)-3})\Big),
    \end{split}
\end{equation}
where the second isomorphism follows from Lemma \ref{Lem.DescendingPartition} and \eqref{Eq.BlowupFactor}.

The choice of cohomology degrees of  $A^+(\PP^{val(v)-3})$ is equivalent to assigning weights in Definition \ref{Def.RootedTree} to the vertices of each rooted tree except the root. On the other hand, for each $T\in \sT_n^o$, the possible weights of the root can be bijectively associated to the degrees of $A^*(\PP^{val(v_0)-3})$, namely $0, 1,\cdots, val(v_0)-3$. Taking the degree $k$ parts, we find that \eqref{Eq.SumOverTrees} is equal to a sum
\begin{equation*}
    A^k(\overline{\mathcal{M}}_{0,n+1})=\bigoplus_{T\in \sT_{n,k}}A^{w(v_0)}(\PP^{val(v_0)-3})
\end{equation*}
over weighted rooted trees in $\sT_{n,k}$.
This is $S_n$-equivariant due to \eqref{Diag.Kap}. Now regrouping the $S_n$-orbits, we have
\begin{equation*}
    \mathrm{Res}^{S_{n+1}}_{S_n}A^k(\overline{\mathcal{M}}_{0,n+1})=\bigoplus_{T\in \sT_{n,k}/S_n}\mathrm{Ind}^{S_n}_{\mathrm{Stab}(T)}A^{w(v_0)}(\PP^{val(v_0)-3})=\bigoplus_{T\in \sT_{n,k}/S_n}U_T
\end{equation*}
as $A^{w(v_0)}(\PP^{val(v_0)-3})$ is the trivial representation of $\Stab(T)$.
\end{proof}

\begin{remark}
The isomorphism
\[\bigoplus_{T\in \sT_{n,k}/S_n}U_T \simeq \bigoplus_{T\in \sT_{n,n-2-k}/S_n}U_T\]
induced by the Poincar\'e duality actually holds on the level of weighted rooted trees. Consider a map
\[\sT_{n,k}\longrightarrow \sT_{n,n-2-k}\]
which associates to a weighted rooted tree $T\in \sT_{n,k}$ a weighted rooted tree obtained from $T$ by changing the weight $w(v)$ of $v\in Ver(T)$ to $val(v)-3-w(v)$ if $v$ is the root, and to $val(v)-2-w(v)$ otherwise. This is a well-defined bijection due to the valency condition
\[0 \leq w(v_0) \leq val(v_0)-3 \quad \text{ and } \quad 1\leq w(v)\leq val(v)-3 \text{ for } v\neq v_0\]
and since by counting the number of the inputs and output we have
\[n+1=\sum_{v\in Ver(T)}val(v)-2 \left|Ver(T)\setminus \{v_0\}\right|.\]
\end{remark}

\begin{example} \label{ex:lowp}
Let us calculate \eqref{eq:Mn+1} for $k=1$ and $k=2$.

Let $k=1$. If we forget all the inputs, there are two types of weighted rooted trees as follows due to the conditions on weights:
\[\begin{tikzpicture} [scale=.5,auto=left,every node/.style={scale=1}]
      \tikzset{Bullet/.style={circle,draw,fill=black,scale=0.5}}
      \node[] (n) at (0,0) {};
      \node[] at (0,-1) {};
      \draw[black] (n) node[] {$\sT_{n,1}/S_n:$};
    \end{tikzpicture}
\begin{tikzpicture} [scale=.5,auto=left,every node/.style={scale=0.8}]
      \tikzset{Bullet/.style={circle,draw,fill=black,scale=0.5}}
      \node[Bullet] (n1) at (0,0) {};
      \node[] at (0,-1) {};
      \node[] at (.5,-1) {};
      \draw[black] (n1) -- (0,1);
      \draw[black] (n1) node[left] {$(1)$};
\end{tikzpicture}
\begin{tikzpicture} [scale=.5,auto=left,every node/.style={scale=1}]
      \tikzset{Bullet/.style={circle,draw,fill=black,scale=0.5}}
       \node[] (n) at (0,0) {};
      \node[] at (0,-0.5) {};
      \draw[black] (-.3,0) node[] {$,~$};
    \end{tikzpicture}
\begin{tikzpicture} [scale=.5,auto=left,every node/.style={scale=0.8}]
      \tikzset{Bullet/.style={circle,draw,fill=black,scale=0.5}}
      \node[Bullet] (n1) at (1,-1) {};
      \node[Bullet] (n2) at (1,-2) {};

      \draw[black] (n1) -- (1,0);
      \draw[black] (n1) -- (n2);
      \draw[black] (n1) node[right] {$ $};
      \draw[black] (n1) node[left] {$(0)$};
      \draw[black] (n2) node[right] {$ $};
      \draw[black] (n2) node[left] {$(1)$};
    \end{tikzpicture}
    \]
The numbers in parentheses denote the weights of the corresponding vertices. There is one weighted rooted tree of the first type, which generates the trivial representation. Those of the second type generate $U_{a,n-a}$ where $a$ denotes the number of inputs attached to the vertex of weight 1 and hence ranges over $3\leq a\leq n-1$ by the valency conditions. By Proposition~\ref{Prop.RepKap}, we have
    \[\Res^{S_{n+1}}_{S_n}A^1(\overline{\mathcal{M}}_{0,n+1})=U_n\oplus\bigoplus_{a=3}^{n-1} U_{a,n-a}=\bigoplus^n_{a=3} U_{a,n-a}.\]

Now let $k=2$. There are five types of weighted rooted trees. All five cases with the corresponding representations are shown in the following table.

\begin{equation} \label{List.Baretrees.Deg2}
\begin{tabular}{|c|c|c|}
\hline
$\sT_{n,2}/S_n$ & tree & representation\\
\hline
(1) &
\parbox[c]{1cm}{
\begin{tikzpicture} [scale=.5,auto=left,every node/.style={scale=0.8}]
      \tikzset{Bullet/.style={circle,draw,fill=black,scale=0.5}}
      \node[Bullet] (n1) at (0,0) {};
      \draw[black] (n1) -- (0,1);
      \draw[black] (n1) node[left] {$(2)$};
      \draw (0,1) node { };
    \end{tikzpicture}
    }
    & $U_n$ \\
\hline
(2) &
\parbox[c]{1cm}{
\begin{tikzpicture} [scale=.5,auto=left,every node/.style={scale=0.8}]
      \tikzset{Bullet/.style={circle,draw,fill=black,scale=0.5}}
      \node[Bullet] (n1) at (1,-1) {};
      \node[Bullet] (n2) at (1,-2) {};
      \draw[black] (n1) -- (1,0);
      \draw[black] (n1) -- (n2);
      \draw[black] (n1) node[left] {$(1)$};
      \draw[black] (n2) node[left] {$(1)$};
      \draw[] (n1) -- (1.7, -1) node[near end, right] {$b$};
      \draw[] (n2) -- (1.7, -2) node[near end, right] {$a$};
    \draw (1,0) node { };
    \end{tikzpicture}
    }
    & $\displaystyle \bigoplus^{n-2}_{a=3} U_{a,n-a}$ \\
\hline
(3) &
\parbox[c]{1cm}{
\begin{tikzpicture} [scale=.5,auto=left,every node/.style={scale=0.8}]
      \tikzset{Bullet/.style={circle,draw,fill=black,scale=0.5}}
      \node[Bullet] (n1) at (1,-1) {};
      \node[Bullet] (n2) at (1,-2) {};
      \draw[black] (n1) -- (1,0);
      \draw[black] (n1) -- (n2);
      \draw[black] (n1) node[left] {$(0)$};
      \draw[black] (n2) node[left] {$(2)$};
            \draw (1,0) node { };
      \draw[] (n1) -- (1.7, -1) node[near end, right] {$b$};
      \draw[] (n2) -- (1.7, -2) node[near end, right] {$a$};
    \end{tikzpicture}
    }
    & $\displaystyle\bigoplus^{n-1}_{a=4} U_{a,n-a}$ \\
\hline
(4) &
\parbox[c]{1cm}{
\begin{tikzpicture} [scale=.5,auto=left,every node/.style={scale=0.8}]
      \tikzset{Bullet/.style={circle,draw,fill=black,scale=0.5}}
      \node[Bullet] (n1) at (1,-1) {};
      \node[Bullet] (n2) at (1,-2) {};
      \node[Bullet] (n3) at (1,-3) {};

      \draw[black] (n1) -- (1,0);
      \draw[black] (n1) -- (n2);
      \draw[black] (n2) -- (n3);
      \draw[black] (n1) node[left] {$(0)$};
      \draw[black] (n2) node[left] {$(1)$};
      \draw[black] (n3) node[left] {$(1)$};
            \draw (1,0) node { };    \draw (1,-3.5) node { };
      \draw[] (n1) -- (1.7, -1) node[near end, right] {$c$};
      \draw[] (n2) -- (1.7, -2) node[near end, right] {$b$};
      \draw[] (n3) -- (1.7, -3) node[near end, right] {$a$};
    \end{tikzpicture}
    }
    & $\displaystyle\bigoplus_{\substack{a\geq 3, b\geq 2 \\ a+b \leq n-1}} U_{a,b,n-a-b}$\\
\hline
(5) &
\parbox[c]{2.1cm}{
\begin{tikzpicture} [scale=.5,auto=left,every node/.style={scale=0.8}]
      \tikzset{Bullet/.style={circle,draw,fill=black,scale=0.5}}
      \node[Bullet] (n1) at (1,-1) {};
      \node[Bullet] (n2) at (0,-2) {};
      \node[Bullet] (n3) at (2,-2) {};

      \draw[black] (n1) -- (1,0);
      \draw[black] (n1) -- (n2);
      \draw[black] (n1) -- (n3);
      \draw[black] (n1) node[right] {$ $};
      \draw[black] (n1) node[left] {$(0)$};
      \draw[black] (n2) node[right] {$ $};
      \draw[black] (n2) node[left] {$(1)$};
      \draw[black] (n3) node[right] {$ $};
      \draw[black] (n3) node[right] {$(1)$};
            \draw (1,0) node { };
      \draw[] (n2) -- (0, -3) node[near end, left] {$a$};
      \draw[] (n3) -- (2, -3) node[near end, right] {$b$};
    \end{tikzpicture}
    }
    & $\displaystyle\left(\bigoplus_{a \geq 3} U_{a,a,n-2a}^{S_2}\right) \oplus \left(\bigoplus_{3\leq a < b} U_{a,b,n-a-b}\right)$ \\
\hline
\end{tabular}
\end{equation}

The letters $a$, $b$ and $c$ in the trees indicate the number of inputs attached to the vertex, whose ranges are determined by the valency conditions on vertices. Note that if $a=b$ in case (5), switching the two vertices with weight 1 gives an isomorphic tree, and hence we take the symmetric part of the corresponding representation under this action.

Therefore by Proposition~\ref{Prop.RepKap}, we have
\begin{equation*}
    \begin{split}
    \Res^{S_{n+1}}_{S_n}A^2(\overline{\mathcal{M}}_{0,n+1})=&~U_n\oplus \bigoplus^{n-2}_{a=3} U_{a,n-a}\oplus \bigoplus^{n-1}_{a=4} U_{a,n-a}\oplus \bigoplus_{\substack{a\geq 3, b\geq 2 \\ a+b \leq n-1}} U_{a,b,n-a-b} \\
    &\oplus \bigoplus_{a \geq 3} U_{a,a,n-2a}^{S_2}\oplus \bigoplus_{3\leq a < b} U_{a,b,n-a-b}.
    \end{split}
\end{equation*}
\end{example}

\bigskip

\section{Representations on the cohomology of $\overline{\mathcal{M}}_{0,n}$} \label{Sec.CF}
As a consequence of Proposition \ref{Prop.RepKap} and Theorem \ref{46}, we obtain a closed formula for the $S_n$-representation on $A^k(\overline{\mathcal{M}}_{0,n})$.

\begin{theorem} \label{Thm.ClosedForm} For $k\ge 0$,  $A^k(\overline{\mathcal{M}}_{0,n})$ as a virtual permutation representation of $S_n$ is
\begin{equation*}
    \begin{split}
        \sum_{\substack{i\geq 0 \\T\in \sT_{n,i}/S_n}}(-1)^{k-i} &U_T+\frac{1}{2}\sum_{\substack{2\leq h \leq n-2\\ i+j\leq k-1}}\sum_{\substack{T_1\in \sT_{h,i}/S_h\\ T_2\in \sT_{n-h,j}/S_{n-h}\\T_1\neq T_2}}(-1)^{k-i-j}U_{T_1}.U_{T_2}\\
            &+(-1)^k\sum_{\substack{2i\leq k-1\\T\in \sT_{\frac{n}{2},i}/S_{\frac{n}{2}}}}s_{(1,1)}\circ U_T,
    \end{split}
\end{equation*}
where $\sT_{\frac{n}{2},i}$ is set to be empty for $n$ odd.
Equivalently, we have
\begin{equation}\label{50}
    \begin{split}
        P_n= & \frac{1}{1+t}\sum_{\substack{0\leq k\leq n-2\\T\in \sT_{n,k}/S_n}}U_T t^k-\frac{1}{2(1+t)}\sum_{\substack{2\leq h \leq n-2 \\ i,j\geq 0\\i+j\leq n-3}}\sum_{\substack{T_1\in \sT_{h,i}/S_h\\ T_2\in \sT_{n-h,j}/S_{n-h}\\T_1\neq T_2}}U_{T_1}.U_{T_2}t^{i+j+1}\\
        &-\frac{1}{1+t}\sum_{\substack{0\leq i\leq \frac{n}{2}-2\\T\in \sT_{\frac{n}{2},i}/S_{\frac{n}{2}}
        }}\left(s_{(1,1)}\circ U_T\right)t^{2i+1}.
    \end{split}
\end{equation}
\end{theorem}

\begin{proof}
    By Theorem \ref{46} and Proposition~\ref{Prop.RepKap}, we have
    \begin{equation*}
        P_n=\frac{1}{1+t}Q_n-\frac{t}{2(1+t)}\left(\sum_{2\leq h\leq n-2,~h\neq\frac{n}{2}}Q_hQ_{n-h}+2s_{(1,1)}\circ Q_{\frac{n}{2}}\right),
    \end{equation*}
    \[  Q_h=\sum_{k\geq 0}\sum_{T\in \sT_{h,k}/S_h}U_{T}t^k. \]
    By the latter equality, we have
    \begin{equation}
        \begin{split}
            Q_hQ_{n-h} & = \sum_{i,j\geq0}\sum_{\substack{T_1\in \sT_{h,i}/S_h\\T_2 \in \sT_{n-h,j}/S_{n-h}}}U_{T_1}.U_{T_2}t^{i+j} \quad \text{ and}\\
            s_{(1,1)}\circ Q_{\frac{n}{2}} & =\frac{1}{2}\sum_{i,j\geq 0}\sum_{\substack{T_1 \in \sT_{\frac{n}{2},i}/S_\frac{n}{2}\\T_2 \in \sT_{\frac{n}{2},j}/S_\frac{n}{2}\\T_1\neq T_2}}U_{T_1}.U_{T_2}t^{i+j}+\sum_{\substack{T\in \sT_{\frac{n}{2},i}/S_{\frac{n}{2}}\\ i\geq0}}\left(s_{(1,1)}\circ U_T\right)t^{2i},
        \end{split}
    \end{equation}
    where the second equality holds by \eqref{Eq.AntiSymm} in Example~\ref{Ex.Plethysm}. Now the assertion follows from $\frac{1}{1+t}=\sum_{i\geq 0} (-1)^it^i$.
\end{proof}

Our results explicitly compute the representations on $A^k(\overline{\mathcal{M}}_{0,n})$ for $k\le 3$ and for all $n$.

\begin{corollary} \label{Cor.SingleDeg}
Let $n\geq 4$ be an integer.
\begin{enumerate}
    \item $\displaystyle A^{1}(\overline{\mathcal{M}}_{0,n})=\bigoplus_{a\geq 4 \text{ even}}U_{a,n-a}.$
    \item $\displaystyle A^{2}(\overline{\mathcal{M}}_{0,n})=\bigoplus_{a\geq 5 \text{ odd}}U_{a,n-a}\oplus \bigoplus_{3\leq a < b}U_{a,b,n-a-b}\oplus \bigoplus_{a\geq 3}U_{a,a,n-2a}^{S_2}.$
    \item $A^3(\overline{\mathcal{M}}_{0,n})$ is the signed sum of representations over weighted rooted trees as in Table \ref{tab:p3}.
\end{enumerate}
Moreover, $A^k(\overline{\mathcal{M}}_{0,n})$ is a permutation representation for $k\leq 3$ and for $k\ge n-6$.
\end{corollary}
\begin{proof}
The formulas (1), (2), (3) are direct consequences of Theorem \ref{Thm.ClosedForm}.
See \S\ref{Subsec.Proof.SingleDeg} for the detailed computation.
It is obvious from (1) and (2) that $A^k(\mzn)$ is a permutation representation for $k\le 2$ or $k\ge n-5$. From Table \ref{tab:p3}, we see that when $n$ is not a multiple of 3, there is no negative term in $A^3(\mzn)$ and hence $A^3(\mzn)$ is a permutation representation.
The case of $A^3(\mzn)$ when $n$ is a multiple of 3 will be treated in \S\ref{Subsec.Proof.SingleDeg} below.
\end{proof}

\begin{table}[hbt!]
    \centering
\begin{tabular}[c]{|c|c|c|c|   }
\hline
sign & tree & range & representation\\
\hline
$+$ &
    &
    & $\displaystyle\bigoplus_{a\geq 6 \text{ even}}U_{a,n-a}$ \\
\hline
$+$ &
\parbox[c][5em][c]{1cm}{
\begin{tikzpicture} [scale=.4,auto=left,every node/.style={scale=0.7}]
      \tikzset{Bullet/.style={circle,draw,fill=black,scale=0.5}}
      \node[Bullet] (n1) at (1,0) {};
      \node[Bullet] (n2) at (1,-1) {};
      \node[Bullet] (n3) at (1,-2) {};

      \draw[black] (n1) -- (1,1);
      \draw[black] (n1) -- (n2);
      \draw[black] (n2) -- (n3);
      \draw[black] (n1) node[right] {$~c$};
      \draw[black] (n1) node[left] {$(0)$};
      \draw[black] (n2) node[right] {$~b$};
      \draw[black] (n2) node[left] {$(2)$};
      \draw[black] (n3) node[right] {$~a$};
      \draw[black] (n3) node[left] {$(1)$};
    \end{tikzpicture}
    }
    &
    \parbox[c][5em][c]{110pt}{
    {\small $3\leq a\leq b$, $c\geq 2$ and \\ $(a,b,c)\neq (\frac{n-2}{2},\frac{n-2}{2},2)$ }}

    &
    \parbox[c][6.5em][c]{90pt}{
        $\displaystyle \bigoplus_{3\leq a <\frac{n-2}{2}} U_{a,n-a-2,2}$
$\displaystyle \oplus \bigoplus_{\substack{3\leq a\leq b\\ a+b \leq n-3}}U_{a,b,c}$
}
    \\
\hline
$+$ &
\parbox[c][][c]{1.8cm}{
\begin{tikzpicture} [scale=.4,auto=left,every node/.style={scale=0.7}]
      \tikzset{Bullet/.style={circle,draw,fill=black,scale=0.5}}
      \node[Bullet] (n1) at (1,-1) {};
      \node[Bullet] (n2) at (0,-2) {};
      \node[Bullet] (n3) at (2,-2) {};

      \draw[black] (n1) -- (1,0);
      \draw[black] (n1) -- (n2);
      \draw[black] (n1) -- (n3);
      \draw[black] (n1) node[right] {$~c$};
      \draw[black] (n1) node[left] {$(0)$};
      \draw[black] (n2) node[left] {$(2)$};
      \draw[black] (0,-2.2) node[below] {$a$};
      \draw[black] (n3) node[right] {$(1)$};
      \draw[black] (2,-2.2) node[below] {$b$};
    \end{tikzpicture}
    }
    &
    {\small $a\geq 4$, $b\geq 3$ and $c\geq 0$ }
    &
\parbox[c][3.3em][c]{80pt}{$\displaystyle\bigoplus_{a\geq 4, b\geq 3} U_{a,b,n-a-b}$}
    \\
\hline
$+$ &
\parbox[c][][c]{1.8cm}{
    \begin{tikzpicture} [scale=.4,auto=left,every node/.style={scale=0.7}]
      \tikzset{Bullet/.style={circle,draw,fill=black,scale=0.5}}
      \node[Bullet] (n0) at (1,0) {};
      \node[Bullet] (n1) at (1,-1) {};
      \node[Bullet] (n2) at (0,-2) {};
      \node[Bullet] (n3) at (2,-2) {};

      \draw[black] (n0) -- (1,1);
      \draw[black] (n0) -- (n1);
      \draw[black] (n1) -- (n2);
      \draw[black] (n1) -- (n3);
      \draw[black] (n0) node[right] {$~d$};
      \draw[black] (n0) node[left] {$(0)$};
      \draw[black] (n1) node[right] {$~c$};
      \draw[black] (n1) node[left] {$(1)$};
      \draw[black] (0,-2.2) node[below] {$a$};
      \draw[black] (n2) node[left] {$(1)$};
      \draw[black] (2,-2.2) node[below] {$b$};
      \draw[black] (n3) node[right] {$(1)$};
    \end{tikzpicture}
    }
    &
\parbox[c][5em][c]{120pt}{\centering\small $3\leq a\leq b\leq \frac{n-2}{2}$, $c\geq0$,\\  $d=1$ and \\ $(a,b,c,d)\neq (\frac{n-2}{2},\frac{n-2}{2},1,1)$}
    &
\parbox[c][6.5em][c]{120pt}{$\displaystyle\bigoplus_{3\leq a<b\leq \frac{n-2}{2}} U_{a,b,n-a-b-1,1}$
    $\displaystyle\oplus \bigoplus_{3\leq a<\frac{n-2}{2}} U_{a,a,n-2a-1,1}^{S_2}$
    }\\
\hline
$+$ &
\parbox[c][][c]{2cm}{
\begin{tikzpicture} [scale=.4,auto=left,every node/.style={scale=0.7}]
      \tikzset{Bullet/.style={circle,draw,fill=black,scale=0.5}}
      \node[Bullet] (n1) at (1,-1) {};
      \node[Bullet] (n2) at (0,-2) {};
      \node[Bullet] (n3) at (2,-2) {};
      \node[Bullet] (n4) at (2,-3) {};

      \draw[black] (n1) -- (1,0);
      \draw[black] (n1) -- (n2);
      \draw[black] (n1) -- (n3);
      \draw[black] (n3) -- (n4);
      \draw[black] (n1) node[right] {$~d$};
      \draw[black] (n1) node[left] {$(0)$};
      \draw[black] (0,-2.2) node[below] {$c$};
      \draw[black] (n2) node[left] {$(1)$};
      \draw[black] (n3) node[right] {$\quad \,\,\, b$};
      \draw[black] (n3) node[right] {$(1)$};
      \draw[black] (n4) node[right] {$\quad \,\,\, a$};
      \draw[black] (n4) node[right] {$(1)$};
    \end{tikzpicture}
    }
    &
    \parbox[c][5em][c]{120pt}{\small $a,c\geq 3$, $b\geq 2$, $d\geq 1$ and\\
    either \\
    \hspace*{1ex}(i) $a+b<\frac{n}{2}$ or
    \\
    \hspace*{1ex}(ii) $a+b=\frac{n}{2}$ and $b<d$}
    &
    \parbox[c][8em][c]{110pt}{$\displaystyle\bigoplus_{\substack{a,c \geq 3, b\geq2 \\ a+b<\frac{n}{2}\\ a+b+c\leq n-1}} U_{a,b,c,n-a-b-c}$
    $\displaystyle \oplus \bigoplus_{\substack{2\leq b \leq \frac{n}{2}-3\\b<d\leq \frac{n}{2}-3} }U_{\frac{n}{2}-b,b,\frac{n}{2}-d,d}$
    }\\
\hline
$+$ &
\parbox[c][][c]{1.8cm}{
\begin{tikzpicture} [scale=.4,auto=left,every node/.style={scale=0.7}]
      \tikzset{Bullet/.style={circle,draw,fill=black,scale=0.5}}
      \node[Bullet] (n1) at (1,-1) {};
      \node[Bullet] (n2) at (0,-2) {};
      \node[Bullet] (n3) at (1,-2) {};
      \node[Bullet] (n4) at (2,-2) {};

      \draw[black] (n1) -- (1,0);
      \draw[black] (n1) -- (n2);
      \draw[black] (n1) -- (n3);
      \draw[black] (n1) -- (n4);
      \draw[black] (n1) node[right] {$~d$};
      \draw[black] (n1) node[left] {$(0)$};
      \draw[black] (n2) node[left] {$a~$};
      \draw[black] (0,-2.2) node[below] {$(1)$};
      \draw[black] (n3) node[right] {$b$};
      \draw[black] (1,-2.2) node[below] {$(1)$};
      \draw[black] (n4) node[right] {$~c$};
      \draw[black] (2,-2.2) node[below] {$(1)$};
    \end{tikzpicture}
    }
    &
    \parbox[c][5em][c]{125pt}{\small(i) $3\leq a \leq b\leq c$, $d\geq1$,\\
    (ii) $3\leq a<b<c$, $d=0$, or \\
    (iii) $(a,b,c,d)=(\frac{n}{3},\frac{n}{3},\frac{n}{3},0)$\\
    }
    &
     \parbox[c][9.2em][c]{112pt}{$\displaystyle\bigoplus_{3\leq a<b<c}U_{a,b,c,n-a-b-c}$ $\displaystyle\oplus \bigoplus_{\substack{a,b\geq 3,~a\neq b \\ 2a+b\leq n-1}}U_{a,a,b,n-2a-b}^{S_2}$
     $\displaystyle\oplus \bigoplus_{3\leq a \leq \frac{n}{3}}U_{a,a,a,n-3a}^{S_3}$ } \\
\hline
$-$ &
\parbox[c][3.7em]{1.9cm}{
    \begin{tikzpicture} [scale=.4,auto=left,every node/.style={scale=0.7}]
      \tikzset{Bullet/.style={circle,draw,fill=black,scale=0.5}}
      \node[Bullet] (n1) at (1,-0.5) {};
      \node[Bullet] (n2) at (0,-1.5) {};
      \node[Bullet] (n3) at (2,-1.5) {};

      \draw[black] (n1) -- (1,0.5);
      \draw[black] (n1) -- (n2);
      \draw[black] (n1) -- (n3);
      \draw[black] (n1) node[left] {$(1)~$};
      \draw[black] (n1) node[right] {$~\frac{n}{3}$};
      \draw[black] (n2) node[left] {$(1)$};
      \draw[black] (0,-1.6) node[below] {$\frac{n}{3}$};
      \draw[black] (n3) node[right] {$(1)$};
      \draw[black] (2,-1.6) node[below] {$\frac{n}{3}$};
    \end{tikzpicture}
    }
    &
    &
\parbox[c][1em][c]{30pt}{{$U_{\frac{n}{3},\frac{n}{3},\frac{n}{3}}^{S_2}$}}
    \\
\hline
$+$ &
\parbox[c][3.7em]{2.2cm}{\begin{tikzpicture} [scale=.4,auto=left,every node/.style={scale=0.7}]
      \tikzset{Bullet/.style={circle,draw,fill=black,scale=0.5}}
      \node[Bullet] (n1) at (1,-1) {};
      \node[Bullet] (n2) at (0,-2) {};
      \node[Bullet] (n3) at (2,-2) {};

      \draw[black] (n1) -- (1,0);
      \draw[black] (n1) -- (n2);
      \draw[black] (n1) -- (n3);
      \draw[black] (n1) node[right] {$~2$};
      \draw[black] (n1) node[left] {$(1)$};
      \draw[black] (n2) node[left] {$(1)$};
      \draw[black] (0,-2.3) node[below] {\small $\frac{n-2}{2}$};
      \draw[black] (n3) node[right] {$(1)$};
      \draw[black] (2,-2.3) node[below] {\small $\frac{n-2}{2}$};
      \draw[] (1,-3) node[below] {$ $};
    \end{tikzpicture}
    }
    &
    &
    \parbox[c][1em][c]{50pt}{$U_{\frac{n-2}{2},\frac{n-2}{2},2}^{S_2}$}\\
\hline
$+$ &
    &
    &
\parbox[c][3.5em][c]{100pt}{ $\displaystyle\bigoplus_{3\leq a \leq \frac{n}{2}-2}s_{(2)}\circ U_{a,\frac{n}{2}-a}$
}
\\
\hline
\noalign{\smallskip}\noalign{\smallskip}
\end{tabular}
    \caption{$A^3(\overline{\mathcal{M}}_{0,n})$}
    \label{tab:p3}
\end{table}


In degrees 1 and 2, Corollary \ref{Cor.SingleDeg} (1) and (2) agree with the known results in \cite{FG} (see also \cite{Fak}) and \cite{RS}.

\begin{example} \label{Ex.Low.n} For $n\leq 11$ and $k=3,4$, a direct calculation shows that $A^k(\overline{\mathcal{M}}_{0,n})$ are permutation representations of $S_n$ as follows. Note that those not listed below are obtained from the lower degree cases by Poincar\'e duality.
\begin{equation*}
    \begin{split}
            A^3(\overline{\mathcal{M}}_{0,9})=&~U_{1,8}\oplus U_{3,6} \oplus U_{4,5}^{\oplus 2}\oplus U_{1,3,5}^{\oplus 2}\oplus U_{1,4,4} \oplus U_{2,3,4} \oplus U_{3,3,3},  \\
            &~\oplus U_{3,3,1,2}^{S_2}\oplus U_{3,3,3}^{S_3},\\
            A^3(\overline{\mathcal{M}}_{0,10})=&~U_{10}\oplus U_{2,8}\oplus U_{3,7}\oplus U_{4,6}^{\oplus 3}\oplus U_{5,5}\oplus U_{1,3,6}\oplus U_{1,4,5}^{\oplus 2}\oplus U_{2,3,5}^{\oplus 2}\\
            &~\oplus U_{2,4,4}\oplus U_{3,3,4}^{\oplus 3}\oplus U_{1,2,3,4}
            \oplus U_{4,4,2}^{S_2}\oplus U_{3,3,3,1}^{S_2}\oplus s_{(2)}\circ U_{3,2}\\
            &~ \oplus U_{3,3,3,1}^{S_3},\\
            A^3(\overline{\mathcal{M}}_{0,11})=&~U_{1,10}\oplus U_{3,8}^{\oplus 2}\oplus U_{4,7}^{\oplus 2}\oplus U_{5,6}^{\oplus 3}\oplus U_{1,3,7} \oplus U_{1,4,6}^{\oplus 2}\oplus U_{1,5,5}\oplus U_{2,3,6}^{\oplus 2}\\
            &~\oplus U_{2,4,5}^{\oplus 3}\oplus U_{3,3,5}^{\oplus 3}\oplus U_{3,4,4}^{\oplus 4} \oplus U_{1,2,3,5} \oplus U_{1,3,3,4} \oplus U_{2,2,3,4}\\
            &~\oplus U_{2,3,3,3} \oplus ({U_{3,3,1,4}^{S_2}})
            ^{\oplus 2}\oplus U_{4,4,1,2}^{S_2}\oplus U_{3,3,3,2}^{S_3},\\
          A^4(\overline{\mathcal{M}}_{0,11})=&~
            U_{11}
            \oplus U_{2,9}
            \oplus U_{3,8}
            \oplus U_{4,7}^{\oplus 3}
            \oplus U_{5,6}^{\oplus 3}
            \oplus U_{1,3,7}
            \oplus U_{1,4,6}^{\oplus 2}
            \oplus U_{1,5,5}\\
            &~\oplus U_{2,3,6}^{\oplus 2}
            \oplus U_{2,4,5}^{\oplus 4}
            \oplus U_{3,3,5}^{\oplus 3}
            \oplus U_{3,4,4}^{\oplus 4}
            \oplus (U_{3,3,5}^{S_2})^{\oplus 2}
            \oplus (U_{4,4,3}^{S_2})^{\oplus 2}\\
            &~\oplus U_{5,5,1}^{S_2}
            \oplus U_{1,2,3,5}
            \oplus U_{1,2,4,4}
            \oplus U_{1,3,3,4}^{\oplus 3}
            \oplus U_{2,2,3,4}^{\oplus 2}
            \oplus U_{2,3,3,3}\\
        &~\oplus U_{3,3,1,4}^{S_2}
            \oplus (U_{3,3,2,3}^{S_2})^{\oplus 2}
            \oplus U_{H'}
            \oplus U_{3,3,3,1,1}^{S_3}.
    \end{split}
\end{equation*}
where $H'$ denotes the subgroup $((S_2\times S_3)\times (S_2 \times S_3))\rtimes S_2 \subset S_{10}\times e \subset S_{11}.$

We also list the $S_n$-characters of $A^k(\overline{\mathcal{M}}_{0,n})$ for $n\leq 11$ in Appendix~\ref{App.A}.
\end{example}

These computations confirm that the $S_n$-representation on $A^k(\overline{\mathcal{M}}_{0,n})$ is a permutation representation for all degrees $k$ when $n\leq 11$. In the next section, we will prove that $A^k(\overline{\mathcal{M}}_{0,n})\oplus A^{k-1}(\overline{\mathcal{M}}_{0,n})$ is a permutation representation for each $k\geq0$.

\begin{remark}
One may want to refine Question \ref{q1} (2) by asking whether the primitive cohomology $H^{2k}_{\mathrm{prim}}(\mzn)$ is a permutation representation or not.
The answer is no. The primitive cohomology $H^{2k}_{\mathrm{prim}}(\mzn)$ as a virtual representation is equal to $A^k(\overline{\mathcal{M}}_{0,n})-A^{k-1}(\overline{\mathcal{M}}_{0,n})$ which is not a permutation representation in general. Indeed, the character of $A^2(\overline{\mathcal{M}}_{0,7})-A^{1}(\overline{\mathcal{M}}_{0,7})$ is (cf. Appendix)
\[ 2 s_{(7)} + s_{(4, 3)} + 2 s_{(5, 2)} + s_{(6, 1)} + s_{(4, 2, 1)}. \]
If this is a permutation representation, it should be a direct sum of two transitive permutation representations. By examining all possible combinations and using the fact that a transitive representation $V$ of a group $G$ satisfies the properties that
\begin{enumerate}
    \item $\mathrm{Tr}_V(1)$ divides $|G|$, and
    \item $\mathrm{Tr}_V(\sigma)\le\mathrm{Tr}_V(\sigma^2)$ for $\sigma\in G$,
\end{enumerate}
one can check that $A^2(\overline{\mathcal{M}}_{0,7})-A^{1}(\overline{\mathcal{M}}_{0,7})$ is not a permutation representation.
\end{remark}

Although a primitive cohomology is not a permutation representation in general, we
do not know whether the sum of two consecutive primitive cohomology groups is a permutation representation or not.

\begin{question}\label{49}
Is $A^k(\overline{\mathcal{M}}_{0,n})-A^{k-2}(\overline{\mathcal{M}}_{0,n})$ a permutation representation of $S_n$ for every $n$, $k$?
\end{question}

If the answer to Question \ref{49} is yes, then the answer to Question \ref{q1} (2) is also yes by a simple induction on $k$ together with Corollary \ref{Cor.SingleDeg}.
Indeed, since $A^k=(A^k-A^{k-2})+A^{k-2}$, if $A^{k-2}$ and $A^{k}-A^{k-2}$ are permutation representations, so is $A^k$. The answer to Question \ref{49} is yes for all $k\leq 3$ and all $n$, and for $(n,k)=(11,4)$ by Corollary~\ref{Cor.SingleDeg} and Example~\ref{Ex.Low.n}.

\subsection{Poincare polynomial of $\overline{\mathcal{M}}_{0,n}/S_n$}
As an application of Theorem \ref{Thm.ClosedForm}, we can calculate the Poincar\'e polynomial of the moduli space $\mzn/S_n$ of $n$ unordered points on rational curves.
Let $$p_n=\sum_k \dim A^k(\mzn/S_n) t^k,\quad  q_n=\sum _k \dim A^k(\mzno/S_n) t^k\quad  \in \mathbb{Z}_{\geq0}[t]$$
be the Poincar\'e polynomials of $\overline{\mathcal{M}}_{0,n}/S_n$ and $\overline{\mathcal{M}}_{0,n+1}/S_n$ respectively. By taking the $S_n$-invariant part of \eqref{50}, we immediately obtain the following.

\begin{corollary} \label{63} Let $n\geq 3$.
    \begin{enumerate}
        \item  $p_n$ and $q_n$ are related by
    \[(1+t)p_n+\frac{1}{2}t\left(\sum_{h=2}^{n-2}q_hq_{n-h}-q_{n/2}\right)=q_n,\]
    where we set $q_{n/2}=0$ if $n$ is odd.
        \item $q_n$ satisfies
    \[q_n(t)=\sum_{k\geq0} b_{n,k}t^k, \quad  b_{n,k}=|\sT_{n,k}/S_n|.\]
\end{enumerate}
In particular, we have a closed formula
\begin{equation*}
    \begin{split}
        \dim A^k(\overline{\mathcal{M}}_{0,n}/S_n)=&\sum_{i= 0}^k(-1)^{k-i} b_{n,i}+\frac{1}{2}\sum_{\substack{2\leq h \leq n-2\\ i+j\leq k-1}}(-1)^{k-i-j}b_{h,i}b_{n-h,j}\\
        &-\frac{1}{2}(-1)^k\sum_{2i\leq k-1}b_{n/2,i},
    \end{split}
\end{equation*}
where $b_{n/2,i}$ is set to be zero for $n$ odd.
\end{corollary}
This reduces the computation of $p_n$ to a purely combinatorial problem, namely the problem of counting the weighted rooted trees up to the ordering of inputs.

For $n\leq 11$, $p_n$ is computed in Table \ref{Table.Poincare}.
\begin{table}[htb!]
\centering
\begin{tabular}{cl}
\toprule
$n$ & \qquad\qquad\qquad\qquad\qquad$p_n$ \\
\midrule
3 &  \quad 1\\
4 &  \quad $1+t$\\
5  & \quad $1+t+t^2$\\
6  & \quad $1+2t+2t^2+t^3$\\
7  & \quad $1+2t+4t^2+2t^3+t^4$\\
8  & \quad $1+3t+6t^2+6t^3+3t^4+t^5$\\
9  & \quad $1+3t+9t^2+11t^3+9t^4+3t^5+t^6$\\
10  & \quad $1+4t+12t^2+21t^3+21t^4+12t^5+4t^6+t^7$\\
11  & \quad $1+4t+16t^2+32t^3+44t^4+32t^5+16t^6+4t^7+t^8$\\
\bottomrule
\noalign{\smallskip}\noalign{\smallskip}
\end{tabular}
\caption{Poincar\'{e} polynomials of $\overline{\mathcal{M}}_{0,n}/S_n$ for $n\leq 11$}
\label{Table.Poincare}
\end{table}

\subsection{Representations on the cohomology of stable map spaces}
Our method works well for the moduli space $\mzn(\PP^{m-1},1)$ of stable maps to $\PP^{m-1}$ of degree 1.
\begin{theorem} \label{FMmain} For $m\ge 2$, we have
    \begin{equation}\label{53}
    \begin{split}
        &P_{\overline{\mathcal{M}}_{0,n}(\PP^{m-1},1)}= \frac{1-t^{2m}}{1-t}P_{\overline{\mathcal{M}}_{0,n}}+ts_{(1,1)}\circ \frac{1-t^m}{1-t}Q_{\frac{n}{2}}\\
                &\quad +\frac{1-t^m}{1-t}\left(\frac{t-t^{m-1}}{1-t} Q_n+\frac{t-t^m}{1-t} s_{(1)}.Q_{n-1}+\frac{t-t^{m+1}}{1-t} \sum_{h=2}^{\lfloor\frac{n-1}{2}\rfloor}Q_hQ_{n-h}\right).
    \end{split}
    \end{equation}
\end{theorem}
\begin{proof}
The theorem follows from \eqref{51} and \eqref{43}.
\end{proof}

If we plug \eqref{52} into  \eqref{53}, we obtain a closed formula for the $S_n$-equivariant Poincar\'e polynomial
$P_{\overline{\mathcal{M}}_{0,n}(\PP^{m-1},1)}$ of the moduli space $\overline{\mathcal{M}}_{0,n}(\PP^{m-1},1)$
of stable maps to $\PP^{m-1}$ of degree 1 for $m\ge 2$.
\bigskip


The $S_n$-equivariant Poincar\'{e} polynomial $P_{\mzn(\PP^{m-1},1)}$ can be computed by torus localization as well. Consider the action of $\CC^*$ on $\PP^{m-1}$ given by
\[ t\cdot [ x_0, x_1, \cdots, x_{m-1}] =[ x_0, tx_1, \cdots, tx_{m-1}]  \hspace{2em} \text{for } t \in \CC^*.\]
There are two fixed sets: the isolated point $p=[1:0:\cdots :0] $ and the hyperplane $H=\{x_0=0\}$. In \cite{Oprea}, Oprea showed that the moduli space $\mzn(\PP^{m-1},1)$ admits a Bia\l{}ynicki-Birula decomposition for the torus action on $\mzn(\PP^{m-1},1)$ induced by this $\CC^*$-action on $\PP^{m-1}$. The decomposition is compatible with $S_n$-action and this leads to the computation of $S_n$-equivariant Poincar\'{e} polynomial $P_{\mzn(\PP^{m-1},1)}$.

For a $\CC^*$-fixed stable map $f$ in $\mzn(\PP^{m-1},1)$, the domain curve has a unique non-contracted component which is mapped isomorphically by $f$ to a curve in $H$ or an invariant curve joining $p$ and a point in $H$. The contracted components and the special points are mapped by $f$ to $p$ or to $H$. For a general classification of fixed stable maps indexed by decorated graphs, see \cite{Oprea}.

Let $\mathcal{F}_i$ be a connected component of the fixed point locus. Then $\mzn(\PP^{m-1},1)$ is covered by locally closed subvarieties $\mathcal{F}_i^+$, called the \emph{Bia\l{}ynicki-Birula (plus) cells}, which are affine fibrations over $\mathcal{F}_i$. The Poincar\'{e} polynomial of $\mzn(\PP^{m-1},1)$ is then given by (cf. \cite[Lemma 6]{Oprea})
\[ P_{\mzn(\PP^{m-1},1)}(t) = \sum P_{\mathcal{F}_i}(t) t^{n_i^-}, \]
where $n_i^-$ is the codimension of $\mathcal{F}_i^+$, which is equal to the number of negative weights on the tangent bundle of $\mzn(\PP^{m-1},1)$ at a fixed point in $\mathcal{F}_i$.

\begin{theorem} \label{thm:Tloc} For $m\ge 2$,
\begin{equation} \label{eq:Tloc}
    P_{\mzn(\PP^{m-1},1)}=  t^2P_{\mzn(\PP^{m-2},1)}+ \sum_{a=0}^n t^{\min\{n-a, 2\}} \frac{1-t^{m-2}}{1-t}Q_{a}Q_{n-a},
\end{equation}
where we set $P_{\mzn(\PP^{m-2},1)}=0$ for $m=2$.
\end{theorem}
\begin{proof}
The locus of $\CC^*$-fixed stable maps whose image lies in $H$ is nonempty if and only if $m>2$ and   is isomorphic to $\mzn(\PP^{m-2},1)$. By \cite[Theorem 2]{Oprea}, the corresponding Bia\l{}ynicki-Birula cell has codimension 2. Hence we have the first term on the right side.

The locus of $\CC^*$-fixed stable maps whose image is an invariant curve joining $p$ and a point in $H$ is isomorphic to the disjoint union of $\overline{\mathcal{M}}_{0,a+1}\times\overline{\mathcal{M}}_{0,n-a+1}\times \PP^{m-2}$ for $0\le a\le n$, where $a$ is the number of markings mapped to $p$. By \cite[Theorem 2]{Oprea} again, the corresponding Bia\l{}ynicki-Birula cell has codimension $\min\{n-a, 2\}$. By regrouping $S_n$-orbits, we have \eqref{eq:Tloc}
\end{proof}
\begin{remark}
After tedious algebraic computation, one can see that \eqref{47} combined with \eqref{FMmain} is equivalent to \eqref{eq:Tloc}. The advantage of using quasimap wall crossing is that to get the cohomology of $\mzn(\PP^{m-1},1)$, we only need the cohomology of $\overline{\mathcal{M}}_{0,k}$ for $k\le n$ in \eqref{FMmain}, while \eqref{eq:Tloc} requires the cohomology of $\overline{\mathcal{M}}_{0,n+1}$  (when $a=0$ or $a=n$). We list the $S_n$-characters of  $A^k(\PP^{1}[n])=A^k(\mzn(\PP^{1},1))$ for $n\le 11$ in Appendix \ref{App.B}.
\end{remark}

\bigskip

\section{$A^k(\overline{\mathcal{M}}_{0,n})\oplus A^{k-1}(\overline{\mathcal{M}}_{0,n})$ is a permutation representation} \label{Sec.PermRep}

In this section, we prove \eqref{16}.
\begin{theorem} \label{Cor.SubseqDeg}
    For each $n\geq 3$ and $k\geq0$, $A^{k}(\overline{\mathcal{M}}_{0,n})\oplus A^{k-1}(\overline{\mathcal{M}}_{0,n})$ is a permutation representation of $S_n$.
\end{theorem}

The strategy of the proof is as follows.
From \eqref{44}, \eqref{45} and \eqref{52}, by taking the coefficient of $t^k$ of $(1+t)P_n$ and by applying Example \ref{Ex.Plethysm} (2),
we have the equality
\begin{equation} \label{Eq.Degreewise}
    \begin{split}
        A^k(\overline{\mathcal{M}}_{0,n})\oplus & A^{k-1}(\overline{\mathcal{M}}_{0,n})\oplus W_{\neq} \oplus W_{=}\\
        & =\bigoplus_{T\in \sT_{n,k}/S_n}U_T\oplus \bigoplus_{T\in \sT_{\frac{n}{2},\frac{k-1}{2}}/S_{\frac{n}{2}}}s_{(2)}\circ U_T
    \end{split}
\end{equation}
of $S_n$-representations where $W_{\neq},W_{=}$ are the $S_n$-representations
\begin{equation*}
    \begin{split}
         W_{\neq}:=\frac{1}{2}\sum_{\substack{2\leq h \leq n-2 \\ i,j\geq 0 \\ i+j=k-1}}\sum_{\substack{T_1\in \sT_{h,i}/S_h \\ T_2 \in \sT_{n-h,j}/S_{n-h}\\ T_1\neq T_2}}U_{T_1}.U_{T_2},\quad  W_{=}:= \sum_{T\in \sT_{\frac{n}{2},\frac{k-1}{2}}/S_{\frac{n}{2}}}U_T.U_T
    \end{split}
\end{equation*}
respectively. Note that $W_{\neq}$ is a permutation representation since for each pair $(T_1,T_2)$ of weighted rooted trees, $U_{T_1}.U_{T_2}=U_{T_2}.U_{T_1}$ appears twice. We prove Theorem \ref{Cor.SubseqDeg} by cancelling out $W_{\neq}\oplus W_{=}$ with permutation representations on the right hand side of \eqref{Eq.Degreewise}. In \S\ref{Subsec.Proof.SingleDeg}, we complete our proof of Corollary~\ref{Cor.SingleDeg}.

\subsection{Operations on weighted rooted trees}
For all integers $n\geq 2$ and $w,p \geq 0$, let $\sT^{w}_{n,p}$ (resp. $\sT^{>w}_{n,p}$, $\sT^{\mathrm{nm}}_{n,p}$) denote the subsets of $\sT_{n,p}$ consisting of trees with \[w(v_0)=w ~~\text{ (resp. }~w(v_0)>w,~w(v_0)<val(v_0)-3).\]
For each integer $k\geq 3$, we denote by $\sT^{val\geq k}_{n,p}$ (resp. $\sT^{val=k}_{n,p}$) the subsets consisting of trees with $val(v_0)\geq k$ (resp. $val(v_0)=k$). We write \[\sT^{w,val\geq k}_{n,p}:=\sT^{w}_{n,p} \cap \sT^{val\geq k}_{n,p},\quad \sT^{>w,val\geq k}_{n,p}:=\sT^{>w}_{n,p} \cap \sT^{val\geq k}_{n,p}.\]
In particular, $\sT_{n,p}=\sT_{n,p}^{>0}\sqcup \sT_{n,p}^{0,val\geq4}\sqcup \sT_{n,p}^{val=3}$. We will sometimes omit the subscript when there is no risk of confusion. We introduce the following three operations on weighted rooted trees.
\begin{itemize}
    \item[(1)] $[1]:~\sT^{\mathrm{nm}}_{n,p}\longrightarrow \sT_{n,p+1}^{>0}$
    \[
    \begin{tikzpicture} [scale=.5,auto=left,every node/.style={scale=0.8}]
      \tikzset{Bullet/.style={circle,draw,fill=black,scale=0.5}}
      \node[Bullet] (T) at (-4,0) {};
      \node[Bullet] (n) at (0,0) {};

      \draw[black] (T) -- (-4,1);
      \draw[black] (T) node[below] {$T$};
      \draw[black] (-2.5,0.5) node[right] {$\longmapsto$};
      \draw[black] (n) -- (0,1);
      \draw[black] (n) node[below] {$T$};
      \draw[black] (n) node[right] {$(+1)$};
    \end{tikzpicture}
    \]
     assigns to each $T$ a weighted rooted tree $T[1]$. Here $T[1]$ denotes the weighted rooted tree having the same underlying rooted tree as $T$ with the same weight except $w(v_0(T[1]))=w(v_0)+1$.

    Here we draw only the root and the output for each rooted tree and $(+1)$ means that the weight of the root is increased by 1.
    \item[(2)] $*:~\sT_{n_1,p_1}\times \sT^{>0}_{n_2,p_2}\longrightarrow \sT_{n_1+n_2,p_1+p_2}^{\mathrm{nm}}$
    \[\begin{tikzpicture} [scale=.5,auto=left,every node/.style={scale=0.8}]
      \tikzset{Bullet/.style={circle,draw,fill=black,scale=0.5}}
      \node[Bullet] (T1) at (-6,-1) {};
      \node[Bullet] (T2) at (-4,-1) {};
      \node[Bullet] (n1) at (1,-0.5) {};
      \node[Bullet] (n2) at (1,-1.5) {};

      \draw[black] (T1) -- (-6,0);
      \draw[black] (T2) -- (-4,0);
      \draw[black] (T1) node[below] {$T_1$};
      \draw[black] (T2) node[below] {$T_2$};
      \draw[black] (-2.5,-1) node[right] {$\longmapsto$};
      \draw[black] (n1) -- (1,0.5);
      \draw[black] (n1) -- (n2);
      \draw[black] (n1) node[left] {$T_1$};
      \draw[black] (n2) node[left] {$T_2$};
    \end{tikzpicture}
    \]
    assigns to a pair $(T_1,T_2)$ a weighted rooted tree $T_1*T_2$ obtained by attaching the output of $T_2$ to the root of $T_1$.

    \item[(3)] $ \cdot_r:~ \prod_{i=1}^N\sT^{>0}_{n_i,p_i}\longrightarrow \sT_{\sum n_i+r,\sum p_i}^{0,val=N+r}$
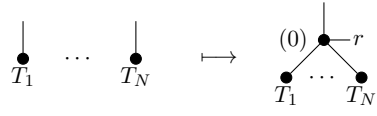
    \[ \qquad
    \begin{tikzpicture} [scale=.5,auto=left,every node/.style={scale=0.8}]
      \tikzset{Bullet/.style={circle,draw,fill=black,scale=0.5}}
      \node[Bullet] (T1) at (-7,-1) {};
      \node[Bullet] (T2) at (-4,-1) {};
      \node[Bullet] (n0) at (1,-0.5) {};
      \node[Bullet] (n1) at (0,-1.5) {};
      \node[Bullet] (n2) at (2,-1.5) {};

      \draw[black] (-5.5,-1) node[] {$\cdots$};
      \draw[black] (T1) -- (-7,0);
      \draw[black] (T2) -- (-4,0);
      \draw[black] (T1) node[below] {$T_1$};
      \draw[black] (T2) node[below] {$T_N$};
      \draw[black] (-2.5,-1) node[right] {$\longmapsto$};
      \draw[black] (n0) -- (1,0.5);
      \draw[black] (n0) -- (n1);
      \draw[black] (n0) -- (n2);
      \draw[black] (1,-1.5) node[] {$\cdots$};
      \draw[black] (n1) node[below] {$T_1$};
      \draw[black] (n0) -- (1.7,-0.5) node[near end, right] {$r$};
      \draw[black] (n0) node[left] {$(0)~$};
      \draw[black] (n2) node[below] {$T_N$};
    \end{tikzpicture}
    \]
    assigns to an $N$-tuple $(T_1,\cdots,T_N)$ a new weighted rooted tree obtained by attaching the outputs of $T_1,\cdots,T_N$ and extra $r$ inputs to a new vertex with weight 0, which is to be the root. This is well-defined when $N+r\geq 2$ while $N$ or $r$ is possibly zero. We denote this new weighted rooted tree by $T_1\cdot \cdots \cdot T_N\cdot 1^r$.

    When we write $T=T_1\cdot T_2$ for $T\in \sT^{val=3}$, we allow $T_i$ to denote an input and in this case we write $T_i=1$ for convenience.
    
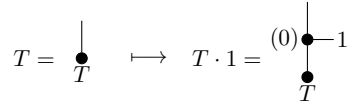
\[\begin{tikzpicture} [scale=.5,auto=left,every node/.style={scale=0.8}]
      \tikzset{Bullet/.style={circle,draw,fill=black,scale=0.5}}
      \node[Bullet] (T1) at (-6,-1) {};
      \node[Bullet] (T2) at (0,-0.5) {};
      \node[Bullet] (T2') at (0,-1.5) {};

      \draw[black] (T1) -- (-6,0);
      \draw[black] (T2) -- (0,0.5);
      \draw[black] (T2) -- (T2');
      \draw[black] (T1) node[below] {$T$};
      \draw[] (-6.5,-1) node[left] {$T=$};
      \draw[] (-3.5,-1) node[left] {$\longmapsto$};
      \draw[black] (T2') node[below] {$T$};
      \draw[] (-1,-1) node[left] {$T\cdot1=$};
      \draw[black] (T2) node[left] {$(0)$};
      \draw[] (T2) -- (0.7, -0.5) node[near end, right] {$1$};
    \end{tikzpicture}\]
\end{itemize}

Note that every weighted rooted tree $T$ can be uniquely written as $T=(T_1\cdot\cdots\cdot T_N \cdot 1^r)[a]$ for some $N,r,a\geq 0$ where $[a]$ denotes the composition of $a$ times of $[1]$.

\subsection{An order on weighted rooted trees} \label{Subsec.Order}
In this subsection, we consider a total order on weighted rooted trees, so that we can write the representation $W_{\neq}$ in \eqref{Eq.Degreewise}  as a sum over (pairs of) weighted rooted trees with positive integral coefficients. Then we apply the operations on pairs of weighted rooted trees to identify $W_{\neq}$ with a permutation representation which is a sum over weighted rooted trees.

A \emph{weighted rooted bare tree} is a weighted rooted tree after forgetting all the inputs.
By definition, weighted rooted bare trees of weight $k$ are pairs $(T,w)$ of trees $T$ with one leg called the output and a weight function $w:Ver(T)\xrightarrow{}\Z$ which satisfies $w(v_0)\geq 0, w(v)\geq 1$ for $v\neq v_0$ and $\sum_{v\in Ver(T)} w(v)=k$ where $v_0$ is the vertex to which the output is attached, called the root of $T$.
Consider the set $\sK_k$ of weighted rooted bare trees of weight $k$.
For each $n$ and $k$, there is a natural forgetful map
\[F:\sT_{n,k}/S_n\lra\sK_k\]
which forgets the inputs of weighted rooted trees.
For $T\in \sT_{n,k}$, we call $F(T)$ the weighted rooted bare tree \emph{associated to} $T$.

We fix any total order on the set $\sK=\sqcup_{k\ge 0}\sK_k$ of weighted rooted bare trees, which respects the weights, that is, for $K_1\in \sK_{k_1}$ and $K_2\in \sK_{k_2}$, $K_1< K_2$ if $k_1< k_2$. Consider any total order on the set \[\bigsqcup_{n\geq 2,k\geq 0}\sT_{n,k}/S_n\]
of all the equivalence classes of weighted rooted trees under the permutations of inputs such that for $T_1\in \sT_{n_1,k_1}$, $T_2\in \sT_{n_2,k_2}$,
\begin{itemize}
    \item[(i)]
    $T_1 < T_2$ if $F(T_1) < F(T_2)$,
    \item[(ii)]
    $T_1<T_2$ if $F(T_1)=F(T_2)$ and $n_1 < n_2$.
\end{itemize}

By definition, $T_1 = T_2$ if $T_1, T_2$ are identical as weighted rooted trees up to permutations of inputs. Furthermore, we assume $1<T$ for all weighted rooted trees $T$, for convenience.

\begin{definition} Let $T=(T_1\cdot\cdots\cdot T_N\cdot 1^r)[a]$ be a weighted rooted tree. We call each $T_i$ \emph{a component} of $T$. If there exists $i$ such that $T_j<T_i$ for all $j\neq i$, then we say $T_i$ is \emph{the dominating component} of $T$.
\end{definition}

Note that the operations on weighted rooted trees naturally descend to the operations on weighted rooted trees modulo permutations of inputs.

\subsection{Proof of Theorem \ref{Cor.SubseqDeg}}
Consider the set $\sP$
\[\sP=\Big\{(T_1,T_2)\in \bigsqcup_{\substack{2\leq h\leq n-2\\i+j=k-1}}\sT_{h,i}/S_h\times \sT_{n-h,j}/S_{n-h} ~\Big|~ T_1\leq T_2\Big\}.\]
of pairs of weighted rooted trees. Then we have
\[W_{\sP}:=\bigoplus_{(T_1,T_2)\in \sP}U_{T_1}.U_{T_2}=W_{\neq}\oplus W_{=}.\]
Therefore the desired cancellation of $W_{\neq}\oplus W_{=}$ in \eqref{Eq.Degreewise} happens if
we have the following.
\begin{proposition}\label{58}
There exists an $S_n$-equivariant embedding
\begin{equation} \label{Embedding}
    W_{\sP} \longrightarrow \bigoplus_{T\in \sT_{n,k}/S_n}U_T\oplus \bigoplus_{T\in \sT_{\frac{n}{2}, \frac{k-1}{2}}/S_{\frac{n}{2}}} s_{(2)}\circ U_T
\end{equation}
whose cokernel is a permutation representation.
\end{proposition}
This proposition proves Theorem \ref{Cor.SubseqDeg}.
The remainder of this subsection is devoted to a proof of Proposition \ref{58}.

To construct an embedding \eqref{Embedding}, we break $\sP$ into thirteen pieces as in Table \ref{tab:11}. The complexity is essentially due to the conditions on the valencies and weights on the weighted rooted tree.

\begin{table}[h!]
    \centering
       \begin{tabular}{|l|c|l|}
    \hline
    \multicolumn{3}{|l|}{When $T_2\in \sT^{val\geq4}$}\\
    \hline
    & $\sP_1$ & $w(v_0(T_2))>0$  \\
    & $\sP_2$  & $w(v_0(T_1))>0$ and $w(v_0(T_2))=0$ \\
&     $\sP_3$  & $w(v_0(T_1))=0$ and $w(v_0(T_2))=0$  \\
    \hline
    \multicolumn{3}{|l|}{When $T_2=T_2'\cdot T_2''$ and $T_2'< T_2''$}\\
    \hline
 &    $\sP_4$  & $T_1\ne T_2$ and $T_2'=1$ \\
   &  $\sR$  &  $T_1=T_2=T\cdot1$\\
   &  $\sP_5$  & $w(v_0(T_1))>0$ and $T_2'\neq1$ \\
    & $\sP_6$  & $w(v_0(T_1))=0$ and $T_2'\neq1$ \\
    \hline
    \multicolumn{3}{|l|}{When $T_1\in \sT^{val\geq4}$ and $T_2=T_2'\cdot T_2'$}\\
    \hline
&     $\sP_7$  & $w(v_0(T_1))>0$ and $T_1 \neq T_2'$ \\
  &   $\sP_8$  & $w(v_0(T_1))=0$ and $T_1[1]\neq T_2'$ \\
    & $\sP_9$  & Either $T_1=T_2'$ or ($w(v_0(T_1))=0$ and $T_1[1]=T_2'$) \\
     \hline
    \multicolumn{3}{|l|}{When $T_1=T_1'\cdot T_1''$ and $T_2=T_2'\cdot T_2'$.}\\
    \hline
&     $\sP_{10}$  & $T_1'\leq T_1''\neq T_2'$ \\
  &   $\sP_{11}$  & $T_1'<T_1''=T_2'$ \\
  &  $\sQ$  & $T_1'=T_1''=T_2'$ \\
\hline
\noalign{\smallskip}\noalign{\smallskip}
    \end{tabular}
    \caption{Partition of $\sP$}
    \label{tab:11}
\end{table}

We first examine $\sQ$ and $\sR$. These require us to use certain relations in the $S_n$-representations (Lemma \ref{lem:reps}), while on other pieces $\sP_i$, \eqref{Embedding} is given by operations defined purely on trees. Consider two subsets
\begin{equation*}
    \begin{split}
        \sQ:=&~\{(T_1,T_2)\in \sP~|~T_1=T_2=T\cdot T, \text{ for some }T\},\\
        \sR:=&~\{(T_1,T_2)\in \sP~|~T_1=T_2=T\cdot1, \text{ for some }T\},
    \end{split}
\end{equation*}
of $\sP$, which consist of the pairs of weighted rooted trees of the forms
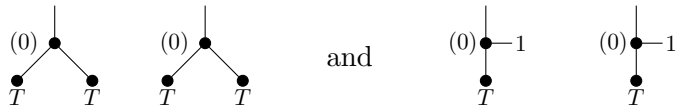
\[
    \begin{tikzpicture} [scale=.5,auto=left,every node/.style={scale=0.8}]
      \tikzset{Bullet/.style={circle,draw,fill=black,scale=0.5}}
      \node[Bullet] (T0) at (-3,-0.5) {};
      \node[Bullet] (T1) at (-4,-1.5) {};
      \node[Bullet] (T2) at (-2,-1.5) {};
      \node[Bullet] (n0) at (1,-0.5) {};
      \node[Bullet] (n1) at (0,-1.5) {};
      \node[Bullet] (n2) at (2,-1.5) {};

      \draw[black] (T0) -- (-3,0.5);
      \draw[black] (T1) -- (T0);
      \draw[black] (T2) -- (T0);
      \draw[black] (T1) node[below] {$T$};
      \draw[black] (T2) node[below] {$T$};
      \draw[black] (T0) node[left] {$(0)~$};
      \draw[black] (n0) -- (1,0.5);
      \draw[black] (n0) -- (n1);
      \draw[black] (n0) -- (n2);
      \draw[black] (n1) node[below] {$T$};
      \draw[black] (n0) node[left] {$(0)~$};
      \draw[black] (n2) node[below] {$T$};
    \end{tikzpicture}\qquad
    \begin{tikzpicture} [scale=0.5,auto=left,every node/.style={scale=1}]
      \tikzset{Bullet/.style={circle,draw,fill=black,scale=0.5}}
      \draw[black] (0,0) node[] {and};
      \draw[black] (0,-1.5);
    \end{tikzpicture}\qquad
    \begin{tikzpicture} [scale=.5,auto=left,every node/.style={scale=0.8}]
      \tikzset{Bullet/.style={circle,draw,fill=black,scale=0.5}}
      \node[Bullet] (T1) at (-4,-0.5) {};
      \node[Bullet] (T1') at (-4,-1.5) {};
      \node[Bullet] (T2) at (0,-0.5) {};
      \node[Bullet] (T2') at (0,-1.5) {};

      \draw[black] (T1) -- (-4,0.5);
      \draw[black] (T1) -- (T1');
      \draw[black] (T2) -- (0,0.5);
      \draw[black] (T2) -- (T2');
      \draw[black] (T1) node[left] {$(0)$};
      \draw[black] (T1') node[below] {$T$};
      \draw[black] (T2') node[below] {$T$};
      \draw[black] (T2) node[left] {$(0)$};
      \draw[] (T1) -- (-3.3, -0.5) node[near end, right] {$1$};
      \draw[] (T2) -- (0.7, -0.5) node[near end, right] {$1$};
    \end{tikzpicture}
\]
respectively. Let
$$\sP':=\sP\setminus(\sQ\sqcup \sR).$$
We write $W_{\star}:=\bigoplus_{(T_1,T_2)\in \star}U_{T_1}.U_{T_2}$ for $\star=\sP',\sQ,\sR$ so that $W_{\sP}=W_{\sP'}\oplus W_{\sQ}\oplus W_{\sR}$. Let $U_{\star}:=\bigoplus_{T\in\star}U_T$ for each subset $\star \subset \sT_{n,p}/S_n$.

Let $$\Bar{\sT}_0\subset \sT_{n,k}/S_n$$ be the subset consisting of the weighted rooted trees of one of the three forms
\[
    \begin{tikzpicture} [scale=.5,auto=left,every node/.style={scale=0.8}]
      \tikzset{Bullet/.style={circle,draw,fill=black,scale=0.5}}
      \node[Bullet] (T0) at (-5,-0.5) {};
      \node[Bullet] (T1) at (-6,-1.5) {};
      \node[Bullet] (T2) at (-5,-1.5) {};
      \node[Bullet] (T3) at (-4,-1.5) {};
      \node[Bullet] (n0) at (1,0) {};
      \node[Bullet] (n1) at (1,-1) {};
      \node[Bullet] (n2) at (0,-2) {};
      \node[Bullet] (n3) at (2,-2) {};
      \node[Bullet] (m0) at (7,0) {};
      \node[Bullet] (m1) at (7,-1) {};
      \node[Bullet] (m2) at (7,-2) {};

      \draw[black] (T0) -- (-5,0.5);
      \draw[black] (T1) -- (T0);
      \draw[black] (T2) -- (T0);
      \draw[black] (T3) -- (T0);
      \draw[black] (T1) node[below] {$T$};
      \draw[black] (T2) node[below] {$T$};
      \draw[black] (T3) node[below] {$T$};
      \draw[black] (T0) node[right] {$~T$};
      \draw[black] (T0) node[left] {$(+1)~$};
      \draw[black] (n0) -- (1,1);
      \draw[black] (n0) -- (n1);
      \draw[black] (n1) -- (n2);
      \draw[black] (n1) -- (n3);
      \draw[black] (n0) node[left] {$(0)~$};
      \draw[black] (n1) node[left] {$(1)~$};
      \draw[black] (n2) node[below] {$T$};
      \draw[black] (n3) node[below] {$T$};
      \draw[] (n0) -- (1.7, 0) node[near end, right] {$1$};
      \draw[] (n1) -- (1.7, -1) node[near end, right] {$1$};
      \draw[black] (m0) -- (7,1);
      \draw[black] (m0) -- (m1);
      \draw[black] (m1) -- (m2);
      \draw[black] (m0) node[left] {$(0)$};
      \draw[black] (m1) node[left] {$(+1)$};
      \draw[black] (m1) node[right] {$~T$};
      \draw[black] (m2) node[right] {$~T$};
      \draw[] (m0) -- (7.7, 0) node[near end, right] {$2$};
      \draw[] (-5,-3.25) node[] {$(((T*T)*T)*T)[1]$};
      \draw[] (1,-3.25) node[] {$(T\cdot T\cdot 1)[1]\cdot1$};
      \draw[] (7,-3.25) node[] {$(T*T)[1]\cdot 1^2$};
    \end{tikzpicture}
\]
for some $T\in \sT^{>0}$, and let $\Bar{\sT}'\subset \sT_{\frac{n}{2},\frac{k-1}{2}}$ be the subset
\[\Bar{\sT}':=\left\{T'\in \sT_{\frac{n}{2},\frac{k-1}{2}}/S_{\frac{n}{2}}~|~T'=T\cdot T \text{ or } T\cdot 1, \text{ for some }T\in \sT^{>0}\right\}.
\]

We will show that there exist two $S_n$-equivariant embeddings
\beq\label{59} W_{\sQ}\oplus W_{\sR} \longrightarrow U_{\bar{\sT}_0}\oplus \bigoplus_{T\in \bar{\sT}'}s_{(2)}\circ U_T, \qquad W_{\sP'}\longrightarrow U_{(\sT_{n,k}/S_n)\setminus \bar{\sT}_0},\eeq
which induce the desired embedding \eqref{Embedding}.
It is obvious that Proposition \ref{58} follows from a construction of equivariant embeddings \eqref{59}
whose cokernels are permutation representations.
In \S\ref{60}, we will construct the first embedding in \eqref{59} whose cokernel is a permutation representation. In \S\ref{61}, we will construct the second whose cokernel is a permutation representation. These complete our proof of Theorem \ref{Cor.SubseqDeg}.

\subsubsection{An embedding of $W_{\sQ}\oplus W_{\sR}$}\label{60}
Consider the injections
\begin{equation*}
    \begin{split}
        &\varphi_1:\sQ \xrightarrow{} \sT_{n,k}/S_n, \quad (T\cdot T,T \cdot T) \mapsto (((T*T)*T)*T)[1]\\
        &\varphi_2:\sR \xrightarrow{} \sT_{n,k}/S_n, \quad (T\cdot1, T\cdot1) \mapsto (T\cdot T\cdot 1)[1]\cdot1\\
        &\varphi_3:\sR \xrightarrow{} \sT_{n,k}/S_n, \quad (T\cdot1, T\cdot1) \mapsto (T*T)[1]\cdot 1^2
    \end{split}
\end{equation*}
whose images are disjoint in $\sT_{n,k}/S_n$, so that
\[\Bar{\sT}_0=\mathrm{Im}(\varphi_1)\sqcup \mathrm{Im}(\varphi_2)\sqcup \mathrm{Im}(\varphi_3).\]

We introduce two identities of $S_n$-representations, which relate the weighted rooted trees in $\Bar{\sT}_0, \Bar{\sT}', \sQ$ and $\sR$.
\begin{lemma}\label{lem:reps}
\begin{enumerate}
    \item $U_{T\cdot T}.U_{T\cdot T}\oplus U_{(T\cdot T \cdot T\cdot T)[1]}=U_{\varphi_1(T\cdot T,T\cdot T)}\oplus (s_{(2)}\circ U_{T\cdot T})$.
    \item $U_{T\cdot1}.U_{T\cdot1}\oplus U_{T\cdot T\cdot1^2[1]}^
    {\oplus 2}= U_{\varphi_2(T\cdot1,T\cdot1)}\oplus U_{\varphi_3(T\cdot1,T\cdot1)}\oplus (s_{(2)}\circ U_{T\cdot1})$.
\end{enumerate}
\end{lemma}
\begin{proof}
(1) Since $U_{T\cdot T}=s_{(2)}\circ U_T$, by Example~\ref{Ex.Plethysm} (2) and the associativity of the plethysm, we have \[U_{T\cdot T}.U_{T\cdot T}=(s_{(2)}+s_{(1,1)})\circ U_{T\cdot T}=((s_{(2)}+s_{(1,1)})\circ s_{(2)}) \circ U_T.\]
On the other hand, we have
\[U_{\varphi_1(T\cdot T,T\cdot T)}=U_{T\cdot T\cdot T}.U_T=(s_{(3)}\circ U_T).U_T=(s_{(3)}.s_{(1)})\circ U_T.\]
Since we have $U_{(T\cdot T\cdot T \cdot T)[1]}=s_{(4)}\circ U_T$ and $s_{(2)}\circ U_{T\cdot T}=(s_{(2)}\circ s_{(2)})\circ U_T$, the assertion follows from the linearity and $s_{(1,1)}\circ s_{(2)}=s_{(3,1)}$.\\
(2) Similarly, since we have $U_{T\cdot1}=U_T.s_{(1)}$, it follows that
\begin{equation*}
    \begin{split}
        &U_{T\cdot1}.U_{T\cdot1}= U_T.U_T.s_{(1)}.s_{(1)}=((s_{(2)}+s_{(1,1)})\circ U_T).(s_{(2)}+s_{(1,1)}),\\
        &U_{T\cdot T\cdot 1^2[1]}=U_{T\cdot T}.s_{(2)}=(s_{(2)}\circ U_T).s_{(2)},\\
        &U_{\varphi_2(T\cdot 1,T\cdot1)}=U_{T\cdot T}.s_{(1)}.s_{(1)}=(s_{(2)}\circ U_T).(s_{(2)}+s_{(1,1)}),\\ &U_{\varphi_3(T\cdot1,T\cdot1)}=U_T.U_T.s_{(2)}=((s_{(2)}+s_{(1,1)})\circ U_{T}).s_{(2)}.
    \end{split}
\end{equation*}
Now the assertion follows from the identity (cf. \cite[Example I.8.3]{Mac})
\[s_{(2)}\circ (V.s_{(1)})=(s_{(2)}\circ V).s_{(2)}+(s_{(1,1)}\circ V).s_{(1,1)},\]
for any $S_n$-representations $V$.
\end{proof}

By Lemma~\ref{lem:reps}, we have an $S_n$-equivariant embedding
\[W_{\sQ}\oplus W_{\sR}\longrightarrow U_{\bar{\sT}_0}\oplus \bigoplus_{T\in \bar{\sT}'}(s_{(2)}\circ U_T),\]
whose cokernel is a permutation representation.

\subsubsection{An embedding of $W_{\sP'}$} \label{61}
Now it remains to construct an injective map
\beq\label{55} \Phi:\sP' \longrightarrow \Big(\sT_{n,k}/S_n\Big)\setminus \bar{\sT}_0\eeq
which satisfies $U_{\Phi(T_1,T_2)}=U_{T_1}.U_{T_2}$. For this, we divide $\sP'$ into subsets $\sP_1,\cdots, \sP_{11}$ as in Table \ref{tab:11}.

We construct injective maps $\Phi_i:\sP_i\xrightarrow{}\sT_{n,k}/S_n$ for $1\leq i\leq 11$ whose images are mutually disjoint and are disjoint from $\bar{\sT}_0$. For concrete examples of the maps $\Phi_i$, see the corresponding definitions for $k=3$ in \S\ref{Subsec.Proof.SingleDeg}.

\emph{Case 1: pairs with $T_2\in \sT^{val\geq4}$.} These pairs are of three different types. We define $\Phi_i$ for $i=1,2,3$ as follows.
\begin{enumerate}
    \item[(i)] Let $\sP_1$ be the set of such pairs with $T_2\in \sT^{>0}$. Define $\Phi_1$ by $\Phi_1(T_1,T_2)=(T_1*T_2)[1]$.
    
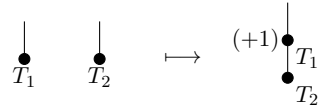
\[
    \begin{tikzpicture} [scale=.5,auto=left,every node/.style={scale=0.8}]
      \tikzset{Bullet/.style={circle,draw,fill=black,scale=0.5}}
      \node[Bullet] (T1) at (-6,-1) {};
      \node[Bullet] (T2) at (-4,-1) {};
      \node[Bullet] (n1) at (1,-0.5) {};
      \node[Bullet] (n2) at (1,-1.5) {};

      \draw[black] (T1) -- (-6,0);
      \draw[black] (T2) -- (-4,0);
      \draw[black] (T1) node[below] {$T_1$};
      \draw[black] (T2) node[below] {$T_2$};
      \draw[black] (-2.5,-1) node[right] {$\longmapsto$};
      \draw[black] (n1) -- (1,0.5);
      \draw[black] (n1) -- (n2);
      \draw[black] (n1) node[below right] {$T_1$};
      \draw[black] (n1) node[left] {$(+1)$};
      \draw[black] (n2) node[below right] {$T_2$};
    \end{tikzpicture}
    \]

    This is well-defined since $T_1*T_2\in \sT^{\mathrm{nm}}$ and $T_2\in \sT^{>0}$. The injectivity follows from the observation that $T_2$ is uniquely determined as the dominating component of $(T_1*T_2)[1]$. Indeed, for any components $T_1'$ of $T_1$, we have $T_1'<T_1\leq T_2$.

    Every tree in the image has the root of weight at least 1 and valency at least 4.

    \item[(ii)] Let $\sP_2$ be the set of such pairs with $T_1\in \sT^{>0}$ and $T_2\in \sT^{0}$. Define $\Phi_2$ by $\Phi_2(T_1,T_2)=T_1\cdot (T_2[1])$.
    \[\qquad \qquad
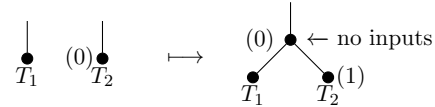
    \begin{tikzpicture} [scale=.5,auto=left,every node/.style={scale=0.8}]
      \tikzset{Bullet/.style={circle,draw,fill=black,scale=0.5}}
      \node[Bullet] (T1) at (-6,-1) {};
      \node[Bullet] (T2) at (-4,-1) {};
      \node[Bullet] (n0) at (1,-0.5) {};
      \node[Bullet] (n1) at (0,-1.5) {};
      \node[Bullet] (n2) at (2,-1.5) {};

      \draw[black] (T1) -- (-6,0);
      \draw[black] (T2) -- (-4,0);
      \draw[black] (T1) node[below] {$T_1$};
      \draw[black] (T2) node[below] {$T_2$};
      \draw[black] (T2) node[left] {$(0)$};
      \draw[black] (-2.5,-1) node[right] {$\longmapsto$};
      \draw[black] (n0) -- (1,0.5);
      \draw[black] (n0) -- (n1);
      \draw[black] (n0) -- (n2);
      \draw[black] (n1) node[below] {$T_1$};
      \draw[black] (n0) node[left] {$(0)~$};
      \draw[black] (n0) node[right] {$~\xleftarrow{} ~$no inputs};
      \draw[black] (n2) node[below] {$T_2$};
      \draw[black] (n2) node[right] {$(1)$};
    \end{tikzpicture}
    \]

    This is well-defined since $T_1, T_2[1]\in \sT^{>0}$ while $T_2[1]$ is well-defined since $T_2\in \sT^{\mathrm{nm}}$. The injectivity follows from that $T_2[1]$ is uniquely determined as it is the dominating component of $T_1\cdot(T_2[1])$. Indeed, we have $T_1\leq T_2 <T_2[1]$.

    Every tree in the image has the root of weight 0 and valency 3, having no inputs. Its dominating component has the root of weight 1.

    \item[(iii)] Let $\sP_3$ be the set of such pairs with $T_1, T_2\in \sT^0$. Define $\Phi_3$ by $\Phi_3(T_1,T_2)=T_1* (T_2[1])$.
    
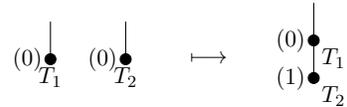
\[
    \begin{tikzpicture} [scale=.5,auto=left,every node/.style={scale=0.8}]
      \tikzset{Bullet/.style={circle,draw,fill=black,scale=0.5}}
      \node[Bullet] (T1) at (-6,-1) {};
      \node[Bullet] (T2) at (-4,-1) {};
      \node[Bullet] (n1) at (1,-0.5) {};
      \node[Bullet] (n2) at (1,-1.5) {};

      \draw[black] (T1) -- (-6,0);
      \draw[black] (T2) -- (-4,0);
      \draw[black] (T1) node[below] {$T_1$};
      \draw[black] (T1) node[left] {$(0)$};
      \draw[black] (T2) node[below] {$T_2$};
      \draw[black] (T2) node[left] {$(0)$};
      \draw[black] (-2.5,-1) node[right] {$\longmapsto$};
      \draw[black] (n1) -- (1,0.5);
      \draw[black] (n1) -- (n2);
      \draw[black] (n1) node[below right] {$T_1$};
      \draw[black] (n1) node[left] {$(0)$};
      \draw[black] (n2) node[below right] {$T_2$};
      \draw[black] (n2) node[left] {$(1)$};
    \end{tikzpicture}
    \]

    This is well-defined since $T_2[1]\in \sT^{>0}$ while $T_2[1]$ is well-defined since $T_2\in \sT^{\mathrm{nm}}$. The injectivity follows from the observation that $T_2[1]$ is uniquely determined as the dominating component of $T_1*(T_2[1])$. Indeed, for any components $T_1'$ of $T_1$ we have $T_1'<T_1\leq T_2$.

    Every tree in the image has the root of weight 0 and valency at least 4. Its dominating component has the root of weight 1.
\end{enumerate}

\emph{Case 2: pairs with $T_2=T_2'\cdot T_2''$ and $T_2'< T_2''$.} Recall that we allow $T_2'=1$, i.e. $T_2'$ may be an input. Other than pairs in $\sR$, there are three more different types. We define $\Phi_i$ for $i=4,5,6$ as follows.

\begin{enumerate}
    \item[(i)] Let $\sP_4$ be the set of such pairs with $T_2=T_2'\cdot 1$ and $T_1\neq T_2$. Define $\Phi_4$ by $\Phi_4(T_1,T_2'\cdot 1)=(T_1*T_2')[1]\cdot 1$.
        \[\qquad
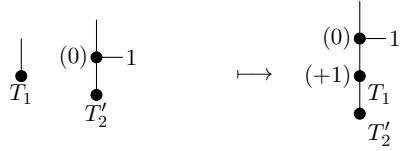
    \begin{tikzpicture} [scale=.5,auto=left,every node/.style={scale=0.8}]
      \tikzset{Bullet/.style={circle,draw,fill=black,scale=0.5}}
      \node[Bullet] (T1) at (-8,-1) {};
      \node[Bullet] (T2) at (-6,-0.5) {};
      \node[Bullet] (T2') at (-6,-1.5) {};
      \node[Bullet] (n1) at (1,0) {};
      \node[Bullet] (n2) at (1,-1) {};
      \node[Bullet] (n3) at (1,-2) {};

      \draw[black] (T1) -- (-8,0);
      \draw[black] (T2) -- (-6,0.5);
      \draw[black] (T2) -- (-5.3,-.5) node[near end, right] {$1$};
      \draw[black] (T2) -- (T2');
      \draw[black] (T1) node[below] {$T_1$};
      \draw[black] (T2') node[below] {$T_2'$};
      \draw[black] (T2) node[left] {$(0)$};
      \draw[black] (-2.5,-1) node[right] {$\longmapsto$};
      \draw[black] (n1) -- (1,1);
      \draw[black] (n1) -- (n2);
      \draw[black] (n2) -- (n3);
      \draw[black] (n1) -- (1.7,0) node[near end, right] {$1$};
      \draw[black] (n1) node[left] {$(0)$};
      \draw[black] (n2) node[below right] {$T_1$};
      \draw[black] (n3) node[below right] {$T_2'$};
      \draw[black] (n2) node[left] {$(+1)$};
    \end{tikzpicture}
    \]

    This is well-defined since $(T_1*T_2')[1]\in \sT^{>0}$ while $(T_1*T_2')[1]$ is well-defined since $T_1*T_2'\in \sT^{\mathrm{nm}}$ and $T_2'\in \sT^{>0}$. The injectivity can be proved as follows. Let $T_1=(T_{1,1}\cdot \cdots \cdot T_{1,N}\cdot 1^r)[a]$. If $N>1$ or $a>0$, then we have $w(T_{1,i})<w(T_1)\leq w(T_2)=w(T_2')$ for all $i$, so $T_2'$ is uniquely determined as it is the dominating component of $(T_1*T_2')[1]$.

    Suppose $N=1$ and $a=0$, so that $T_1=T\cdot 1^r$ with $r\ge1$. Then
\begin{equation}\label{eq:p4}
      \Phi_4(T\cdot 1^r,T_2'\cdot1) =
    \parbox[c][][c]{2cm}{ \hspace*{.5em}
    \begin{tikzpicture} [scale=.5,auto=left,every node/.style={scale=0.8}]
      \tikzset{Bullet/.style={circle,draw,fill=black,scale=0.5}}
      \node[Bullet] (n1) at (1,0) {};
      \node[Bullet] (n2) at (1,-1) {};
      \node[Bullet] (n3) at (0,-2) {};
      \node[Bullet] (n4) at (2,-2) {};

      \draw[black] (n1) -- (1,1);
      \draw[black] (n1) -- (n2);
      \draw[black] (n2) -- (n3);
      \draw[black] (n2) -- (n4);
      \draw[black] (n1) -- (1.7,0) node[near end, right] {$1$};
      \draw[black] (n1) node[left] {$(0)$};
      \draw[black] (n2) -- (1.7,-1) node[near end, right] {$r$};
      \draw[black] (n3) node[below] {$T$};
      \draw[black] (n4) node[below] {$T_2'$};
      \draw[black] (n2) node[left] {$(1)$};
    \end{tikzpicture}
    }
\end{equation}

Another possible pair $(\widetilde{T}_1,\widetilde{T}_2)$ mapped to \eqref{eq:p4} by $\Phi_4$ is
\[\widetilde{T}_1:=T_2'\cdot 1^r \text{ and }\widetilde{T}_2:=T\cdot1.\] Hence, it suffices to show $\widetilde{T}_1> \widetilde{T}_2$ so that $(\widetilde{T}_1,\widetilde{T}_2)\notin \sP'$. Since $F(\widetilde{T}_2)=F(T_1)\leq F(T_2)=F(\widetilde{T}_1)$, we may assume that $F(T_1)=F(T_2)$ for otherwise $F(\widetilde{T}_1)>F(\widetilde{T}_2)$. Hence we have $F(T)=F(T_2')$ and
\[\hspace{3em} n(T)-n(T_2')= n(T_1)-n(T_2)-(r-1)\leq n(T_1)-n(T_2) \leq0,\] where $n(T)$ denotes the number of the inputs of a weighted rooted tree $T$. If $r>1$, then $n(T)< n(T_2')$ and $n(\widetilde{T}_1)>n(\widetilde{T}_2)$, which implies $\widetilde{T}_1>\widetilde{T}_2$ because  $F(\widetilde{T}_1)=F(\widetilde{T}_2)$. Therefore, we may also assume that $r=1$ and then we have $\widetilde{T}_1=T_2$ and $\widetilde{T}_2=T_1$. Hence $\widetilde{T}_1> \widetilde{T}_2$ as required.

    Every tree in the image has the root of weight 0 and valency 3, having precisely 1 input. Moreover it is not of the form $(T\cdot T\cdot 1)[1]\cdot 1$ as $\sP_4 \cap \sR=\emptyset$. The dominating component of such a tree has the root of weight greater than 1.

    \item[(ii)] Let $\sP_5$ be the set of pairs with $T_1 \in  \sT^{>0}$,  $T_2=T_2'\cdot T_2''$, $T_2'< T_2''$ and $T_2'\neq 1$. Define $\Phi_5$ by $\Phi_5(T_1,T_2)=T_1\cdot ((T_2'*T_2'')[1])$.
    \[\quad \qquad
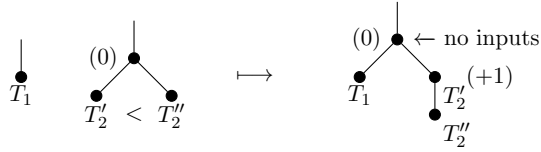
    \begin{tikzpicture} [scale=.5,auto=left,every node/.style={scale=0.8}]
      \tikzset{Bullet/.style={circle,draw,fill=black,scale=0.5}}
      \node[Bullet] (T1) at (-8,-1) {};
      \node[Bullet] (T2) at (-5,-0.5) {};
      \node[Bullet] (T2') at (-6,-1.5) {};
      \node[Bullet] (T2'') at (-4,-1.5) {};
      \node[Bullet] (n1) at (2,0) {};
      \node[Bullet] (n2) at (1,-1) {};
      \node[Bullet] (n3) at (3,-1) {};
      \node[Bullet] (n4) at (3,-2) {};

      \draw[black] (T1) -- (-8,0);
      \draw[black] (T2) -- (-5,0.5);
      \draw[black] (T2) -- (T2');
      \draw[black] (T2) -- (T2'');
      \draw[black] (T1) node[below] {$T_1$};
      \draw[black] (T2') node[below] {$T_2'$};
      \draw[black] (T2'') node[below] {$T_2''$};
      \draw[black] (T2) node[left] {$(0)~$};
      \draw[black] (-5,-2) node[] {$<$};
      \draw[black] (-2.5,-1) node[right] {$\longmapsto$};
      \draw[black] (n1) -- (2,1);
      \draw[black] (n1) -- (n2);
      \draw[black] (n1) -- (n3);
      \draw[black] (n3) -- (n4);
      \draw[black] (n1) node[right] {$~\xleftarrow{} ~$no inputs};
      \draw[black] (n1) node[left] {$(0)~$};
      \draw[black] (n2) node[below] {$T_1$};
      \draw[black] (n3) node[below right] {$T_2'$};
      \draw[black] (n3) node[right] {\quad$(+1)$};
      \draw[black] (n4) node[below right] {$T_2''$};
    \end{tikzpicture}
    \]

    This is well-defined since $T_1,T_2''\in \sT^{>0}$ and $T_2'*T_2''\in \sT^{\mathrm{nm}}$. The injectivity follows from the observation that $(T_2'*T_2'')[1]$ is uniquely determined as the dominating component of $T_1\cdot((T_2'*T_2)[1])$ and $T_2''$ is uniquely determined as the dominating component of $(T_2'*T_2'')[1]$. Indeed, we have $w(T_1)\leq w(T_2'\cdot T_2'') < w(T_2'*T_2''[1])$ and any components other than $T_2''$ in $(T_2'*T_2'')[1]$ are components of $T_2'$, so they are dominated by $T_2'$.

    Every tree in the image has the root of weight 0 and valency 3, having no inputs. Its dominating component has the root of weight greater than 1.

    \item[(iii)] Let $\sP_6$ be the set of pairs with $T_1 \in  \sT^{0}$, $T_2=T_2'\cdot T_2''$, $T_2'< T_2''$ and $T_2'\neq 1$. Define $\Phi_6$ by $\Phi_6(T_1,T_2)=T_1*((T_2'*T_2'')[1])$.
    
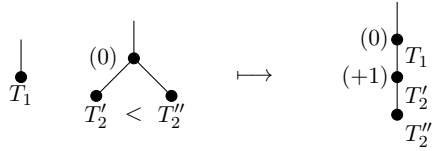
\[
    \begin{tikzpicture} [scale=.5,auto=left,every node/.style={scale=0.8}]
      \tikzset{Bullet/.style={circle,draw,fill=black,scale=0.5}}
      \node[Bullet] (T1) at (-8,-1) {};
      \node[Bullet] (T2) at (-5,-0.5) {};
      \node[Bullet] (T2') at (-6,-1.5) {};
      \node[Bullet] (T2'') at (-4,-1.5) {};
      \node[Bullet] (n1) at (2,0) {};
      \node[Bullet] (n2) at (2,-1) {};
      \node[Bullet] (n3) at (2,-2) {};

      \draw[black] (T1) -- (-8,0);
      \draw[black] (T2) -- (-5,0.5);
      \draw[black] (T2) -- (T2');
      \draw[black] (T2) -- (T2'');
      \draw[black] (T1) node[below] {$T_1$};
      \draw[black] (T2') node[below] {$T_2'$};
      \draw[black] (T2'') node[below] {$T_2''$};
      \draw[black] (T2) node[left] {$(0)~$};
      \draw[black] (-5,-2) node[] {$<$};
      \draw[black] (-2.5,-1) node[right] {$\longmapsto$};
      \draw[black] (n1) -- (2,1);
      \draw[black] (n1) -- (n2);
      \draw[black] (n2) -- (n3);
      \draw[black] (n1) node[left] {$(0)$};
      \draw[black] (n1) node[below right] {$T_1$};
      \draw[black] (n2) node[below right] {$T_2'$};
      \draw[black] (n3) node[below right] {$T_2''$};
      \draw[black] (n2) node[left] {$(+1)$};
    \end{tikzpicture}
    \]

    This is well-defined since $T_2''\in \sT^{>0}$ and $T_2'*T_2''\in \sT^{\mathrm{nm}}$. The injectivity follows from the observation that $(T_2'*T_2'')[1]$ is uniquely determined as the dominating component of $T_1*((T_2'*T_2)[1])$ and that $T_2''$ is uniquely determined as the dominating component of $(T_2'*T_2'')[1]$.

    Every tree in the image has the root of weight 0 and valency at least 4. Its dominating component has the root of weight greater than 1.
\end{enumerate}

\emph{Case 3: pairs with $T_1\in \sT^{val\geq4}$ and $T_2=T_2'\cdot T_2'$.} These pairs are of three different types. We define $\Phi_i$ for $i=7,8,9$ as follows.

\begin{enumerate}
    \item[(i)] Let $\sP_7$ be the set of such pairs with $T_1\in \sT^{>0}$ and $T_1\neq T_2'$. Define $\Phi_7$ by $\Phi_7(T_1,T_2)=(T_2*T_1)[1]$.
    \[\quad \qquad
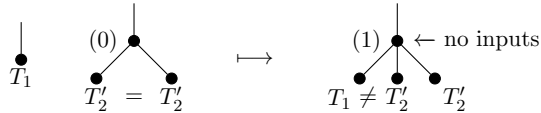
    \begin{tikzpicture} [scale=.5,auto=left,every node/.style={scale=0.8}]
      \tikzset{Bullet/.style={circle,draw,fill=black,scale=0.5}}
      \node[Bullet] (T1) at (-8,-1) {};
      \node[Bullet] (T2) at (-5,-0.5) {};
      \node[Bullet] (T2') at (-6,-1.5) {};
      \node[Bullet] (T2'') at (-4,-1.5) {};
      \node[Bullet] (n1) at (2,-0.5) {};
      \node[Bullet] (n2) at (1,-1.5) {};
      \node[Bullet] (n3) at (2,-1.5) {};
      \node[Bullet] (n4) at (3,-1.5) {};

      \draw[black] (T1) -- (-8,0);
      \draw[black] (T2) -- (-5,0.5);
      \draw[black] (T2) -- (T2');
      \draw[black] (T2) -- (T2'');
      \draw[black] (T1) node[below] {$T_1$};
      \draw[black] (T2') node[below] {$T_2'$};
      \draw[black] (T2'') node[below] {$T_2'$};
      \draw[black] (T2) node[left] {$(0)~$};
      \draw[black] (-5,-2) node[] {$=$};
      \draw[black] (-2.5,-1) node[right] {$\longmapsto$};
      \draw[black] (n1) -- (2,0.5);
      \draw[black] (n1) -- (n2);
      \draw[black] (n1) -- (n3);
      \draw[black] (n1) -- (n4);
      \draw[black] (n1) node[left] {$(1)~$};
      \draw[black] (n1) node[right] {$~\xleftarrow{}~$no inputs};
      \draw[black] (1,-1.57) node[below left] {$T_1$};
      \draw[black] (n3) node[below] {$T_2'$};
      \draw[black] (n4) node[below right] {$T_2'$};
      \draw[black] (1.23,-2) node[] {$\neq$};
    \end{tikzpicture}
    \]

    This is well-defined since $T_1,T_2'\in \sT^{>0}$ and $T_2*T_1\in \sT^{\mathrm{nm}}$. The injectivity immediately follows as $T_1$ is uniquely determined. Note that the condition $T_1\neq T_2'$ ensures that $U_{T_1}.U_{T_2}=U_{\Phi_7(T_1,T_2)}$.

    Every tree in the image has the root of weight 1 and valency 4, having no inputs.

    \item[(ii)] Let $\sP_{8}$ be the set of such pairs with $T_1\in \sT^{0}$ and $T_1[1]\neq T_2'$. Define $\Phi_8$ by $\Phi_{8}(T_1,T_2)=T_2*(T_1[1])$.
    \[\quad \qquad
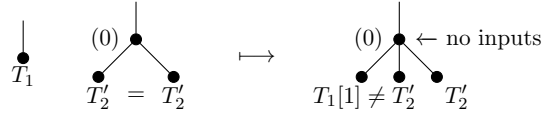
    \begin{tikzpicture} [scale=.5,auto=left,every node/.style={scale=0.8}]
      \tikzset{Bullet/.style={circle,draw,fill=black,scale=0.5}}
      \node[Bullet] (T1) at (-8,-1) {};
      \node[Bullet] (T2) at (-5,-0.5) {};
      \node[Bullet] (T2') at (-6,-1.5) {};
      \node[Bullet] (T2'') at (-4,-1.5) {};
      \node[Bullet] (n1) at (2,-0.5) {};
      \node[Bullet] (n2) at (1,-1.5) {};
      \node[Bullet] (n3) at (2,-1.5) {};
      \node[Bullet] (n4) at (3,-1.5) {};

      \draw[black] (T1) -- (-8,0);
      \draw[black] (T2) -- (-5,0.5);
      \draw[black] (T2) -- (T2');
      \draw[black] (T2) -- (T2'');
      \draw[black] (T1) node[below] {$T_1$};
      \draw[black] (T2') node[below] {$T_2'$};
      \draw[black] (T2'') node[below] {$T_2'$};
      \draw[black] (T2) node[left] {$(0)~$};
      \draw[black] (-5,-2) node[] {$=$};
      \draw[black] (-2.5,-1) node[right] {$\longmapsto$};
      \draw[black] (n1) -- (2,0.5);
      \draw[black] (n1) -- (n2);
      \draw[black] (n1) -- (n3);
      \draw[black] (n1) -- (n4);
      \draw[black] (n1) node[left] {$(0)~$};
      \draw[black] (n1) node[right] {$~\xleftarrow{}~$no inputs};
      \draw[black] (n3) node[below] {$T_1[1]\neq T_2'$\qquad\quad\quad};
      \draw[black] (n4) node[below right] {$T_2'$};
    \end{tikzpicture}
    \]

    This is well-defined since $T_1 \in \sT^{\mathrm{nm}}$ and $T_1[1]\in \sT^{>0}$. The injectivity immediately follows as $T_1[1]$ is uniquely determined. Note that the condition $T_1[1]\neq T_2'$ ensures that $U_{T_1}.U_{T_2}=U_{\Phi_8(T_1,T_2)}$.

    Every tree in the image has the root of weight 0 and valency 4, having no inputs. If it has the dominating component then its root has weight 1.

    \item[(iii)] Let $\sP_{9}$ be the set of such pairs with either $T_1=T_2'$ or $T_1\in \sT^{0}$ and $T_1[1]=T_2'$. Define $\Phi_9$ by $ \Phi_{9}(T_1,T_2)=((T_1*T_2')*T_2')[1]$.
    \[\quad
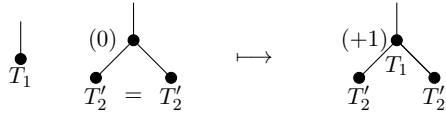
    \begin{tikzpicture} [scale=.5,auto=left,every node/.style={scale=0.8}]
      \tikzset{Bullet/.style={circle,draw,fill=black,scale=0.5}}
      \node[Bullet] (T1) at (-8,-1) {};
      \node[Bullet] (T2) at (-5,-0.5) {};
      \node[Bullet] (T2') at (-6,-1.5) {};
      \node[Bullet] (T2'') at (-4,-1.5) {};
      \node[Bullet] (n1) at (2,-0.5) {};
      \node[Bullet] (n2) at (1,-1.5) {};
      \node[Bullet] (n3) at (3,-1.5) {};

      \draw[black] (T1) -- (-8,0);
      \draw[black] (T2) -- (-5,0.5);
      \draw[black] (T2) -- (T2');
      \draw[black] (T2) -- (T2'');
      \draw[black] (T1) node[below] {$T_1$};
      \draw[black] (T2') node[below] {$T_2'$};
      \draw[black] (T2'') node[below] {$T_2'$};
      \draw[black] (T2) node[left] {$(0)~$};
      \draw[black] (-5,-2) node[] {$=$};
      \draw[black] (-2.5,-1) node[right] {$\longmapsto$};
      \draw[black] (n1) -- (2,0.5);
      \draw[black] (n1) -- (n2);
      \draw[black] (n1) -- (n3);
      \draw[black] (n1) -- (n4);
      \draw[black] (n1) node[left] {$~(+1)$};
      \draw[black] (2,-0.7) node[below] {$T_1$};
      \draw[black] (n2) node[below] {$T_2'$};
      \draw[black] (n3) node[below] {$T_2'$};
    \end{tikzpicture}
    \]

    This is well-defined since $(T_1*T_2')*T_2'\in \sT^{\mathrm{nm}}$. The injectivity follows from that $T_1=T_2'$ or $T_1[1]=T_2'$.

    Every tree in the image has the root of weight greater than 0 and valency at least 6.
\end{enumerate}

\emph{Case 4: pairs with $T_1=T_1'\cdot T_1''$ and $T_2=T_2'\cdot T_2'$.} Other than pairs in $\sQ$, there are two more different types. We define $\Phi_i$ for $i=10,11$ as follows.
\begin{enumerate}
    \item[(i)] Let $\sP_{10}$ be the set of such pairs with $T_1'\leq T_1''\neq T_2'$. Define $\Phi_{10}$ by $\Phi_{10}(T_1,T_2)=(T_1'\cdot T_1''\cdot T_2'\cdot T_2')[1]$.
    \[\quad \qquad
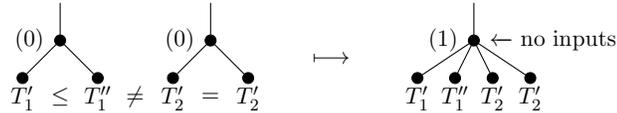
    \begin{tikzpicture} [scale=.5,auto=left,every node/.style={scale=0.8}]
      \tikzset{Bullet/.style={circle,draw,fill=black,scale=0.5}}
      \node[Bullet] (T1) at (-9,-0.5) {};
      \node[Bullet] (T1') at (-10,-1.5) {};
      \node[Bullet] (T1'') at (-8,-1.5) {};
      \node[Bullet] (T2) at (-5,-0.5) {};
      \node[Bullet] (T2') at (-6,-1.5) {};
      \node[Bullet] (T2'') at (-4,-1.5) {};
      \node[Bullet] (n1) at (2,-0.5) {};
      \node[Bullet] (n2) at (0.5,-1.5) {};
      \node[Bullet] (n3) at (1.5,-1.5) {};
      \node[Bullet] (n4) at (2.5,-1.5) {};
      \node[Bullet] (n5) at (3.5,-1.5) {};

      \draw[black] (T1) -- (-9,0.5);
      \draw[black] (T1) -- (T1');
      \draw[black] (T1) -- (T1'');
      \draw[black] (T2) -- (-5,0.5);
      \draw[black] (T2) -- (T2');
      \draw[black] (T2) -- (T2'');
      \draw[black] (T1) node[left] {$(0)~$};
      \draw[black] (T1') node[below] {$T_1'$};
      \draw[black] (T1'') node[below] {$T_1''$};
      \draw[black] (T2') node[below] {$T_2'$};
      \draw[black] (T2'') node[below] {$T_2'$};
      \draw[black] (T2) node[left] {$(0)~$};
      \draw[black] (-9,-2) node[] {$\leq$};
      \draw[black] (-7,-2) node[] {$\neq$};
      \draw[black] (-5,-2) node[] {$=$};
      \draw[black] (-2.5,-1) node[right] {$\longmapsto$};
      \draw[black] (n1) -- (2,0.5);
      \draw[black] (n1) -- (n2);
      \draw[black] (n1) -- (n3);
      \draw[black] (n1) -- (n4);
      \draw[black] (n1) -- (n5);
      \draw[black] (n1) node[left] {$(1)~$};
      \draw[black] (n1) node[right] {$~\xleftarrow{}~$no inputs};
      \draw[black] (n2) node[below] {$T_1'$};
      \draw[black] (n3) node[below] {$T_1''$};
      \draw[black] (n4) node[below] {$T_2'$};
      \draw[black] (n5) node[below] {$T_2'$};
   \end{tikzpicture}
    \]

    This is well-defined since $T_1',T_1'',T_2'\in \sT^{>0}$. The injectivity immediately follows as $T_2'$ is uniquely determined.

    Every tree in the image has the root of weight 1 and valency 5, having no inputs.

    \item[(ii)] Let $\sP_{11}$ be the set of such pairs with $T_1'< T_1''= T_2'$. Define $\Phi_{11}$ by $\Phi_{11}(T_1,T_2)=T_1'\cdot (((T_1''*T_2')*T_2')[1])$.
    \[\quad \qquad
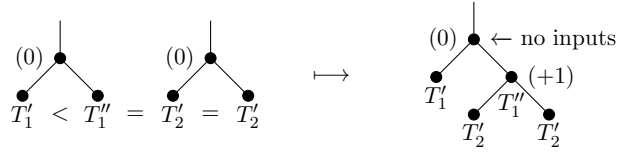
    \begin{tikzpicture} [scale=.5,auto=left,every node/.style={scale=0.8}]
      \tikzset{Bullet/.style={circle,draw,fill=black,scale=0.5}}
      \node[Bullet] (T1) at (-9,-0.5) {};
      \node[Bullet] (T1') at (-10,-1.5) {};
      \node[Bullet] (T1'') at (-8,-1.5) {};
      \node[Bullet] (T2) at (-5,-0.5) {};
      \node[Bullet] (T2') at (-6,-1.5) {};
      \node[Bullet] (T2'') at (-4,-1.5) {};
      \node[Bullet] (n1) at (2,0) {};
      \node[Bullet] (n2) at (1,-1) {};
      \node[Bullet] (n3) at (3,-1) {};
      \node[Bullet] (n4) at (2,-2) {};
      \node[Bullet] (n5) at (4,-2) {};

      \draw[black] (T1) -- (-9,0.5);
      \draw[black] (T1) -- (T1');
      \draw[black] (T1) -- (T1'');
      \draw[black] (T2) -- (-5,0.5);
      \draw[black] (T2) -- (T2');
      \draw[black] (T2) -- (T2'');
      \draw[black] (T1) node[left] {$(0)~$};
      \draw[black] (T1') node[below] {$T_1'$};
      \draw[black] (T1'') node[below] {$T_1''$};
      \draw[black] (T2') node[below] {$T_2'$};
      \draw[black] (T2'') node[below] {$T_2'$};
      \draw[black] (T2) node[left] {$(0)~$};
      \draw[black] (-9,-2) node[] {$<$};
      \draw[black] (-7,-2) node[] {$=$};
      \draw[black] (-5,-2) node[] {$=$};
      \draw[black] (-2.5,-1) node[right] {$\longmapsto$};
      \draw[black] (n1) -- (2,1);
      \draw[black] (n1) -- (n2);
      \draw[black] (n1) -- (n3);
      \draw[black] (n3) -- (n4);
      \draw[black] (n3) -- (n5);
      \draw[black] (n1) node[right] {$~\xleftarrow{} ~$no inputs};
      \draw[black] (n1) node[left] {$(0)~$};
      \draw[black] (n2) node[below] {$T_1'$};
      \draw[black] (3,-1.2) node[below] {$T_1''$};
      \draw[black] (n3) node[right] {~$(+1)$};
      \draw[black] (n4) node[below] {$T_2'$};
      \draw[black] (n5) node[below] {$T_2'$};
    \end{tikzpicture}
    \]

    This is well-defined since $T_1',T_2'\in \sT^{>0}$ and $(T_1''*T_2')*T_2'\in \sT^{\mathrm{nm}}$. The injectivity immediately follows as $T_1'$ and $T_2'$ are uniquely determined.

    Every tree in the image has the root of weight 0 and valency 3. Its dominating component is of the form $((T*T)*T)[1]$ for some $T\in \sT^{>0}$.
\end{enumerate}

Define $\Phi:\sP'\xrightarrow{} \sT_{n,k}/S_n$ by $\Phi|_{\sP_i}=\Phi_i$ for $i=1,\cdots,11$. By construction, we see that $U_{\Phi(T_1,T_2)}=U_{T_1}.U_{T_2}$ for all $(T_1,T_2)\in\sP$. As each $\Phi_i$ is injective, it suffices to show that the images $\bar{\sT}_i:=\mathrm{Im}(\Phi_i)\subset \sT_{n,k}/S_n$ are mutually disjoint and disjoint from $\bar{\sT}_0=\mathrm{Im}(\varphi_1)\sqcup \mathrm{Im}(\varphi_2)\sqcup \mathrm{Im}(\varphi_3)$.

\begin{lemma}
Each of the following unions is disjoint.
\ben
    \item $\bar{\sT}_2 \cup \bar{\sT}_4 \cup \bar{\sT}_5 \cup \bar{\sT}_{11} \cup \mathrm{Im}(\varphi_2)\subset \sT^{val=3}_{n,k}/S_n$,
    \item $\bar{\sT}_1 \cup \bar{\sT}_7 \cup \bar{\sT}_9 \cup \bar{\sT}_{10}\cup \mathrm{Im}(\varphi_1) \subset \sT^{>0,val\geq 4}_{n,k}/S_n$,
    \item $\bar{\sT}_3 \cup \bar{\sT}_6 \cup \bar{\sT}_8 \cup \mathrm{Im}(\varphi_3) \subset \sT^{0,val\geq4}_{n,k}/S_n$.
\een
\end{lemma}
\begin{proof}(1) We first see that $\bar{\sT}_2$, $\bar{\sT}_4$, $\bar{\sT}_5$ and $\mathrm{Im}(\varphi_2)$ are mutually disjoint as follows. Since the trees in $\bar{\sT}_4$ and $\mathrm{Im}(\varphi_2)$ have inputs at the roots while the trees in $\bar{\sT}_2$ and $\bar{\sT}_5$ do not, $\bar{\sT}_2$ and $\bar{\sT}_5$ are disjoint from $\bar{\sT}_4$ and $\mathrm{Im}(\varphi_2)$. We have $\bar{\sT}_2\cap \bar{\sT}_5=\emptyset$ as the weight of the root of the dominating component is 1 for each tree in $\bar{\sT}_2$ and is greater than 1 for each tree in $\bar{\sT}_5$. The disjointness of $\bar{\sT}_4$ and $\mathrm{Im}(\varphi_2)$ comes from $\sP_4\cap \sR=\emptyset$.

On the other hand, $\bar{\sT}_{11}$ is disjoint from the others since the trees in $\bar{\sT}_{11}$ have the dominating component of the form $((T*T)*T)[1]$ with $T\in \sT^{>0}$, which has the root of weight greater than 1 and does not have the dominating component. Note that every tree in $\Bar{\sT}_2\cup \mathrm{Im}(\varphi_2)$ has a dominating component with the root of weight $1$, and the dominating component of each tree in $\Bar{\sT}_4\cup \Bar{\sT}_5$ thought of as a weighted rooted tree has a dominating component.

(2) Observe that the valency of the root of trees in $\bar{\sT}_9\cup \mathrm{Im}(\varphi_1)$ is at least 6, that of trees in $\bar{\sT}_7 $ is 4 and that of trees in $\bar{\sT}_{10}$ is 5. For $\bar{\sT}_9$ and $\mathrm{Im}(\varphi_1)$, each tree in $\bar{\sT}_9$ has at most 2 copies of the maximal component while each tree in $\mathrm{Im}(\varphi_1)$ has three maximal components. Thus we see that $\bar{\sT}_7 $, $\bar{\sT}_9$, $\bar{\sT}_{10}$ and $\mathrm{Im}(\varphi_1)$ are mutually disjoint.

To check that $\bar{\sT}_1$ is disjoint from the others, observe first that trees in $\bar{\sT}_1$ have  dominating components while those in $\mathrm{Im}(\varphi_1)$ do not. Lastly, we check that $\bar{\sT}_1$ is disjoint from $\bar{\sT}_7$, $\bar{\sT}_9$ and $\bar{\sT}_{10}$. A tree in $\bar{\sT}_1$ can be written as $T=(T''*T')[1]$, where $T'$ is the dominating component and $T''\le T'$.
However, if we write $T=(T''*T')[1]$ for a tree in $\bar{\sT}_7$, $\bar{\sT}_9$ and $\bar{\sT}_{10}$ so that $T'$ is the dominating component, we can check $T'<T''$ from the construction. For example, if $T\in \bar{\sT}_{10}$ has the dominating component $T'$, in the notation of the construction above, $T'=T_1''$ and $T''=T_1'\cdot T_2'\cdot T_2'$. Clearly, $T'<T''$ in this case. Other cases are similar.

(3) The sets $\bar{\sT}_6$ and $\mathrm{Im}(\varphi_3)$ are disjoint from $\bar{\sT}_3$ and $\bar{\sT}_8$ because the weight of the root of the dominating component is greater than 1 for a tree in the former, and is 1 or there is no dominating component for a tree in the latter.
Moreover, $\bar{\sT}_6$ and $\mathrm{Im}(\varphi_3)$ are disjoint as $(1\cdot 1,T \cdot T)\notin \sP_6$. Finally, to show $\bar{\sT}_3$ and $ \bar{\sT}_8$ are disjoint,
let $T \in \bar{\sT}_3$. Then $T$ can be written as $T=T''*(T'[1])$ and has the dominating component $T'[1]$. By construction, we have $T''\le T'$ in this case. However, if $T$ is also in $\sP_8$, since $T$ has a dominating component, we see that $T'=T_1<T_2=T''$ where $T=\Phi_8(T_1,T_2)$. The last inequality is indeed strict since $T_1\in \sT^{val\geq 4}$ and $T_2\in \sT^{val=3}$. \end{proof}

Consequently, $\Phi$ is injective and factors through $(\sT_{n,k}/S_n)\setminus \bar{\sT}_0$.
We thus have the second equivariant embedding in \eqref{59} whose cokernel is a permutation representation.

\begin{remark} \label{Rem.Subseq}
Let $\bar{\sT}:=\bigsqcup_{i=0}^{11}\bar{\sT}_i$. Let $\bar{\sT}^{(1)}, \bar{\sT}^{(2)}\subset \sT_{n,k}/S_n$ be the subsets consisting of all the weighted rooted trees of the form $(T\cdot T \cdot T\cdot T)[1]$ and $(T\cdot T\cdot 1^2)[1]$ respectively for some $T\in \sT^{>0}$ (cf. Lemma~\ref{lem:reps}). Then the proof of Theorem~\ref{Cor.SubseqDeg} in fact shows
\begin{equation} \label{Eq.SubseqDeg}
    A^{k-1}(\overline{\mathcal{M}}_{0,n})\oplus A^{k}(\overline{\mathcal{M}}_{0,n})=U_{\bar{\sT}^c}\oplus U_{\bar{\sT}^{(1)}}\oplus U_{\bar{\sT}^{(2)}}^{\oplus 2}\oplus \bigoplus_{T\in (\bar{\sT}')^c}\left(s_{(2)}\circ U_T\right).
\end{equation}
where $\bar{\sT}^c:=\left(\sT_{n,k}/S_n\right)\setminus \bar{\sT}$ and $(\bar{\sT}')^c:=\left(\sT_{\frac{n}{2},\frac{k-1}{2}}/S_{\frac{n}{2}}\right)\setminus \bar{\sT}'$. Note that the last three summands vanish unless both $n$ and $k-1$ are even.
\end{remark}

\subsection{Proof of Corollary~\ref{Cor.SingleDeg}} \label{Subsec.Proof.SingleDeg}
Applying the proof of Theorem \ref{Cor.SubseqDeg} above, we can calculate the $S_n$-representation $A^k(\overline{\mathcal{M}}_{0,n})$ for $k=1,2$ and $3$.
Although there are much simpler proofs in some of these cases, we use the systematic approach developed above and Remark \ref{Rem.Subseq}.

\subsubsection{}
When $k=1$, due to the restriction on weights, $\sP=\sP_3$ and it consists of pairs of weighted rooted trees with 1 vertex, where $a\leq n-a$ if the numbers of inputs are $a, n-a$ respectively. Hence under $\Phi$, we obtain weighted rooted trees whose weighted bare trees have two vertices and have weight 1 and at most $\frac{n}{2}$ inputs at the root.
\[\begin{tikzpicture} [scale=.5,auto=left,every node/.style={scale=0.8}]
      \tikzset{Bullet/.style={circle,draw,fill=black,scale=0.5}}
      \node[Bullet] (T1) at (-8,-1) {};
      \node[Bullet] (T2) at (-5,-1) {};
      \node[Bullet] (n1) at (1,-0.5) {};
      \node[Bullet] (n2) at (1,-1.5) {};

      \draw[black] (T1) -- (-8,0);
      \draw[black] (T2) -- (-5,0);
      \draw[black] (-8,-1.2) node[below] {$a$};
      \draw[black] (T1) node[left] {$(0)$};
      \draw[black] (-5,-1.2) node[below] {$n-a$};
      \draw[black] (T2) node[left] {$(0)$};
      \draw[black] (-6.8,-1) node[below] {$\leq$};
      \draw[black] (-2.5,-1) node[right] {$\longmapsto$};
      \draw[black] (n1) -- (1,0.5);
      \draw[black] (n1) -- (n2);
      \draw[black] (n1) node[right] {$~a~(2\leq a\leq \frac{n}{2})$};
      \draw[black] (n1) node[left] {$(0)$};
      \draw[black] (n2) node[right] {$~n-a$};
      \draw[black] (n2) node[left] {$(1)$};
    \end{tikzpicture}
    \]

On the other hand, $\sT_{n,1}$ consists of all the weighted rooted trees of weight 1, having at most two vertices (Example~\ref{ex:lowp}). Therefore we have
\[\begin{tikzpicture} [scale=.5,auto=left,every node/.style={scale=1}]
      \tikzset{Bullet/.style={circle,draw,fill=black,scale=0.5}}
      \node[] (n) at (0,0) {};
      \node[] at (0,-1) {};
      \draw[black] (n) node[] {$U_{(\sT_{n,1}/S_n)\setminus \bar{\sT}}=$};
    \end{tikzpicture}
\begin{tikzpicture} [scale=.5,auto=left,every node/.style={scale=0.8}]
      \tikzset{Bullet/.style={circle,draw,fill=black,scale=0.5}}
      \node[Bullet] (n1) at (0,0) {};
      \node[] at (0,-1) {};

      \draw[black] (n1) -- (0,1);
      \draw[black] (0,-0.2) node[below] {$n$};
      \draw[black] (n1) node[left] {$(1)$};
    \end{tikzpicture}
    \begin{tikzpicture} [scale=.5,auto=left,every node/.style={scale=1}]
      \tikzset{Bullet/.style={circle,draw,fill=black,scale=0.5}}
      \node[] (n) at (0,0) {};
      \node[] at (0,-1) {};
      \draw[black] (n) node[] {$+$};
    \end{tikzpicture}
\begin{tikzpicture} [scale=.5,auto=left,every node/.style={scale=0.8}]
      \tikzset{Bullet/.style={circle,draw,fill=black,scale=0.5}}
      \node[Bullet] (n1) at (1,-1) {};
      \node[Bullet] (n2) at (1,-2) {};

      \draw[black] (n1) -- (1,0);
      \draw[black] (n1) -- (n2);
      \draw[black] (n1) node[right] {$~1$};
      \draw[black] (n1) node[left] {$(0)$};
      \draw[black] (n2) node[right] {$~n-1$};
      \draw[black] (n2) node[left] {$(1)$};
    \end{tikzpicture}
    \begin{tikzpicture} [scale=.5,auto=left,every node/.style={scale=1}]
      \tikzset{Bullet/.style={circle,draw,fill=black,scale=0.5}}
      \node[] (n) at (0,0) {};
      \node[] at (0,-1) {};
      \draw[black] (n) node[] {$+$};
    \end{tikzpicture}
\begin{tikzpicture} [scale=.5,auto=left,every node/.style={scale=0.8}]
      \tikzset{Bullet/.style={circle,draw,fill=black,scale=0.5}}
      \node[Bullet] (n1) at (1,-1) {};
      \node[Bullet] (n2) at (1,-2) {};

      \draw[black] (n1) -- (1,0);
      \draw[black] (n1) -- (n2);
      \draw[black] (n1) node[right] {$~a~ (\frac{n}{2}<a\leq n-3)$};
      \draw[black] (n1) node[left] {$(0)$};
      \draw[black] (n2) node[right] {$~n-a$};
      \draw[black] (n2) node[left] {$(1)$};
    \end{tikzpicture}
    \begin{tikzpicture} [scale=.5,auto=left,every node/.style={scale=1}]
      \tikzset{Bullet/.style={circle,draw,fill=black,scale=0.5}}
      \node[] (n) at (0,0) {};
      \node[] at (0,-1) {};
    \end{tikzpicture}
    \]
\[\,\,\,\,=U_n \oplus U_{1,n-1}\oplus \bigoplus_{\frac{n}{2}<a \leq n-3}U_{a,n-a}.\]
Since $(\bar{\sT}')^c=\sT_{\frac{n}{2},0}/S_{\frac{n}{2}}$ consists of one element corresponding to the trivial representation $U_{\frac{n}{2}}$ for even $n$ and since $A^0(\overline{\mathcal{M}}_{0,n})=U_n$, \eqref{Eq.SubseqDeg} reads as
\[A^1(\overline{\mathcal{M}}_{0,n})=U_{1,n-1}\oplus \bigoplus_{\frac{n}{2}< a \leq n-3}U_{a,n-a}\oplus \left( s_{(2)}\circ U_{\frac{n}{2}}\right) =\bigoplus_{a\geq 4 \text{ even}}U_{a,n-a}\]
for all $n$, where we set $U_{\frac{n}{2}}=0$ for odd $n$. For even $n$, the last equality follows from the identity
\begin{equation} \label{Identity.OddEven}
    \bigoplus_{a\leq \frac{n}{2} \text{ even}}U_{a,n-a}=\left(s_{(2)}\circ U_{\frac{n}{2}}\right) \oplus \bigoplus_{a>\frac{n}{2} \text{ odd}}U_{a,n-a}.
\end{equation}
This proves Corollary~\ref{Cor.SingleDeg} (1).

\subsubsection{}
When $k=2$, the sets $\sQ, \sR, \bar{\sT}_0, \bar{\sT}^{(1)}, \bar{\sT}^{(2)}$ and $\bar{\sT}'$ are empty, and $\sP'=\sP_1\sqcup \sP_3 \sqcup \sP_4$ consists of pairs $(T_1,T_2)$ of weighted rooted trees where each $T_1$ has weight 0 and only one vertex and $T_2$ has weight 1 and at most two vertices. Choose an order on weighted rooted trees of weight 1 as
\[
    \begin{tikzpicture} [scale=.5,auto=left,every node/.style={scale=0.8}]
      \tikzset{Bullet/.style={circle,draw,fill=black,scale=0.5}}
      \node[Bullet] (n1) at (0,0) {};
      \node[] at (0,-1) {};

      \draw[black] (n1) -- (0,1);
      \draw[black] (n1) node[left] {$(1)$};
    \end{tikzpicture}
    \begin{tikzpicture} [scale=.5,auto=left,every node/.style={scale=0.8}]
      \tikzset{Bullet/.style={circle,draw,fill=black,scale=0.5}}
      \node[] (n) at (0,0) {};
      \node[] at (0,-1) {};
      \draw[black] (n) node[] {$~<$};
    \end{tikzpicture}
    \begin{tikzpicture} [scale=.5,auto=left,every node/.style={scale=0.8}]
      \tikzset{Bullet/.style={circle,draw,fill=black,scale=0.5}}
      \node[Bullet] (n1) at (1,-1) {};
      \node[Bullet] (n2) at (1,-2) {};

      \draw[black] (n1) -- (1,0);
      \draw[black] (n1) -- (n2);
      \draw[black] (n1) node[right] {$ $};
      \draw[black] (n1) node[left] {$(0)$};
      \draw[black] (n2) node[right] {$ $};
      \draw[black] (n2) node[left] {$(1)$};
    \end{tikzpicture}
    \begin{tikzpicture} [scale=.5,auto=left,every node/.style={scale=1}]
      \tikzset{Bullet/.style={circle,draw,fill=black,scale=0.5}}
      \node[] (n) at (0,0) {};
      \node[] at (0,-0.5) {};
      \draw[black] (n) node[above] {$.$};
    \end{tikzpicture}
\]
Then $\Phi_1, \Phi_3, \Phi_4$ generate all the weighted trees whose associated weighted bare trees are of type (2) and (4) in \eqref{List.Baretrees.Deg2} as follows.
\[
    \begin{tikzpicture} [scale=.5,auto=left,every node/.style={scale=0.8}]
      \tikzset{Bullet/.style={circle,draw,fill=black,scale=0.5}}
      \node[Bullet] (T1) at (-4.5,-1) {};
      \node[Bullet] (T2) at (-3,-1) {};
      \node[Bullet] (n1) at (0,-0.5) {};
      \node[Bullet] (n2) at (0,-1.5) {};
      \node[] at (0,-2.3) {};

      \draw[black] (T1) -- (-4.5,0);
      \draw[black] (T2) -- (-3,0);
      \draw[black] (T1) node[below] {$(0)$};
      \draw[black] (T2) node[below] {$(1)$};
      \draw[black] (-2.5,-1) node[right] {$\longmapsto$};
      \draw[black] (n1) -- (0,0.5);
      \draw[black] (n1) -- (n2);
      \draw[black] (n1) node[left] {$(1)$};
      \draw[black] (n2) node[left] {$(1)$};
    \end{tikzpicture}
    \begin{tikzpicture} [scale=.5,auto=left,every node/.style={scale=1}]
      \tikzset{Bullet/.style={circle,draw,fill=black,scale=0.5}}
      \node[] (n) at (0,0) {};
      \node[] at (0,-0.5) {};
      \draw[black] (n) node[above] {$,$};
    \end{tikzpicture}
    \quad
    \begin{tikzpicture} [scale=.5,auto=left,every node/.style={scale=0.8}]
      \tikzset{Bullet/.style={circle,draw,fill=black,scale=0.5}}
      \node[Bullet] (T1) at (-5,-1) {};
      \node[Bullet] (T2) at (-3,-0.5) {};
      \node[Bullet] (T2') at (-3,-1.5) {};
      \node[Bullet] (n1) at (1,0) {};
      \node[Bullet] (n2) at (1,-1) {};
      \node[Bullet] (n3) at (1,-2) {};

      \draw[black] (T1) -- (-5,0);
      \draw[black] (T2) -- (-3,0.5);
      \draw[black] (T2) -- (T2');
      \draw[black] (T1) node[below] {$(0)$};
      \draw[black] (T2) node[left] {$(0)$};
      \draw[black] (T2) node[right] {\small $~\geq 2$};
      \draw[black] (T2') node[left] {$(1)$};
      \draw[black] (-1.5,-1) node[right] {$\longmapsto$};
      \draw[black] (n1) -- (1,1);
      \draw[black] (n1) -- (n2);
      \draw[black] (n2) -- (n3);
      \draw[black] (n1) node[left] {$(0)$};
      \draw[black] (n1) node[right] {\small $~\geq2$};
      \draw[black] (n2) node[left] {$(1)$};
      \draw[black] (n3) node[left] {$(1)$};
    \end{tikzpicture}
    \begin{tikzpicture} [scale=.5,auto=left,every node/.style={scale=1}]
      \tikzset{Bullet/.style={circle,draw,fill=black,scale=0.5}}
      \node[] (n) at (0,0) {};
      \node[] at (0,-0.5) {};
      \draw[black] (n) node[above] {$,$};
    \end{tikzpicture}
    \quad
    \begin{tikzpicture} [scale=.5,auto=left,every node/.style={scale=0.8}]
      \tikzset{Bullet/.style={circle,draw,fill=black,scale=0.5}}
      \node[Bullet] (T1) at (-5,-1) {};
      \node[Bullet] (T2) at (-3,-0.5) {};
      \node[Bullet] (T2') at (-3,-1.5) {};
      \node[Bullet] (n1) at (0.5,0) {};
      \node[Bullet] (n2) at (0.5,-1) {};
      \node[Bullet] (n3) at (0.5,-2) {};

      \draw[black] (T1) -- (-5,0);
      \draw[black] (T2) -- (-3,0.5);
      \draw[black] (T2) -- (T2');
      \draw[black] (T1) node[below] {$(0)$};
      \draw[black] (T2) node[left] {$(0)$};
      \draw[black] (T2) node[right] {$~1$};
      \draw[black] (T2') node[left] {$(1)$};
      \draw[black] (-2,-1) node[right] {$\longmapsto$};
      \draw[black] (n1) -- (0.5,1);
      \draw[black] (n1) -- (n2);
      \draw[black] (n2) -- (n3);
      \draw[black] (n1) node[left] {$(0)$};
      \draw[black] (n1) node[right] {$~1$};
      \draw[black] (n2) node[left] {$(1)$};
      \draw[black] (n3) node[left] {$(1)$};
    \end{tikzpicture}
    \begin{tikzpicture} [scale=.5,auto=left,every node/.style={scale=1}]
      \tikzset{Bullet/.style={circle,draw,fill=black,scale=0.5}}
      \node[] (n) at (0,0) {};
      \node[] at (0,-1) {};
      \draw[black] (n) node[] {$.$};
    \end{tikzpicture}
    \]
The number without parentheses next to a vertex represents the number of inputs attached to that vertex. By \eqref{Eq.SubseqDeg}, $A^1(\overline{\mathcal{M}}_{0,n})\oplus A^2(\overline{\mathcal{M}}_{0,n})=U_{(\sT_{n,2}/S_n)\setminus \bar{\sT}}$ is generated by the weighted rooted tree of the first, the third and the last types in \eqref{List.Baretrees.Deg2};
    \begin{equation} \label{Eq.SubseqDeg2}
            \begin{tikzpicture} [scale=.5,auto=left,every node/.style={scale=1}]
      \tikzset{Bullet/.style={circle,draw,fill=black,scale=0.5}}
      \node[] (n) at (0,0) {};
      \node[] at (0,-1) {};
      \draw[black] (n) node[] {$A^1(\overline{\mathcal{M}}_{0,n})\oplus A^2(\overline{\mathcal{M}}_{0,n})=$};
    \end{tikzpicture}
    \begin{tikzpicture} [scale=.5,auto=left,every node/.style={scale=0.8}]
      \tikzset{Bullet/.style={circle,draw,fill=black,scale=0.5}}
      \node[Bullet] (n1) at (0,0) {};
      \node[] at (0,-1) {};

      \draw[black] (n1) -- (0,1);
      \draw[black] (0,-0.2) node[below] {};
      \draw[black] (n1) node[left] {$(2)$};
    \end{tikzpicture}
    \begin{tikzpicture} [scale=.5,auto=left,every node/.style={scale=1}]
      \tikzset{Bullet/.style={circle,draw,fill=black,scale=0.5}}
      \node[] (n) at (0,0) {};
      \node[] at (0,-1) {};

      \draw[black] (n) node[] {$+$};
    \end{tikzpicture}
\begin{tikzpicture} [scale=.5,auto=left,every node/.style={scale=0.8}]
      \tikzset{Bullet/.style={circle,draw,fill=black,scale=0.5}}
      \node[Bullet] (n1) at (1,-1) {};
      \node[Bullet] (n2) at (1,-2) {};

      \draw[black] (n1) -- (1,0);
      \draw[black] (n1) -- (n2);
      \draw[black] (n1) node[right] {$ $};
      \draw[black] (n1) node[left] {$(0)$};
      \draw[black] (n2) node[right] {$ $};
      \draw[black] (n2) node[left] {$(2)$};
    \end{tikzpicture}
    \begin{tikzpicture} [scale=.5,auto=left,every node/.style={scale=1}]
      \tikzset{Bullet/.style={circle,draw,fill=black,scale=0.5}}
      \node[] (n) at (0,0) {};
      \node[] at (0,-1) {};

      \draw[black] (n) node[] {$+$};
    \end{tikzpicture}
\begin{tikzpicture} [scale=.5,auto=left,every node/.style={scale=0.8}]
      \tikzset{Bullet/.style={circle,draw,fill=black,scale=0.5}}
      \node[Bullet] (n1) at (1,-1) {};
      \node[Bullet] (n2) at (0,-2) {};
      \node[Bullet] (n3) at (2,-2) {};

      \draw[black] (n1) -- (1,0);
      \draw[black] (n1) -- (n2);
      \draw[black] (n1) -- (n3);
      \draw[black] (n1) node[right] {$ $};
      \draw[black] (n1) node[left] {$(0)$};
      \draw[black] (n2) node[right] {$ $};
      \draw[black] (n2) node[left] {$(1)$};
      \draw[black] (n3) node[right] {$ $};
      \draw[black] (n3) node[right] {$(1)$};
    \end{tikzpicture} \qquad \qquad
    \end{equation}
    \[\qquad \qquad =U_n\oplus \bigoplus_{4\leq a \leq n-1}U_{a,n-a}\oplus \bigoplus_{3\leq a<b}U_{a,b,n-a-b}\oplus \bigoplus_{a\geq 3}U_{a,a,n-2a}^{S_2}\]
    where the range of $a$ or $b$ in each subscript of the direct sums reflects the valency conditions on the weighted rooted trees for a given associated weighted bare tree. Therefore we have
    \[A^2(\overline{\mathcal{M}}_{0,n})=\bigoplus_{a \geq 5 \text{ odd}}U_{a,n-a}\oplus \bigoplus_{3\leq a<b}U_{a,b,n-a-b}\oplus \bigoplus_{a\geq 3}U_{a,a,n-2a}^{S_2}.\]
This proves Corollary~\ref{Cor.SingleDeg} (2).

\subsubsection{}
We compute $A^2(\overline{\mathcal{M}}_{0,n})\oplus A^3(\overline{\mathcal{M}}_{0,n})$ and then subtract $A^2(\overline{\mathcal{M}}_{0,n})$ from it. For the former, we first identify the sets $\bar{\sT}_i$ and the complement of their union in $\sT_{n,3}/S_n$. Note that there are 13 types of weighted rooted trees in $\sT_{n,3}/S_3$ according to the associated weighted bare trees, which are listed as follows.
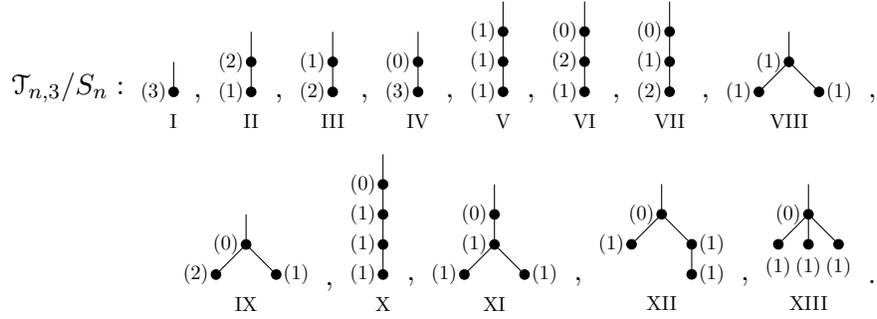
\begin{equation*}

\]
To determine $\bigsqcup_{i>0}\bar{\sT}_i$, choose any order on weighted rooted trees satisfying
\[
    %
    \]
if $b>d$ when $a+b=c+d$. Here the left inequality is between weighted bare trees, and $a,b,c,d$ on the right denote the number of inputs attached to the corresponding vertices. Note that $\sP_5, \sP_7$ and $\sP_{11}$ are empty, so are $\bar{\sT}_{5}, \bar{\sT}_{7}$ and $\bar{\sT}_{11}$.
\begin{enumerate}
    \item $\bar{\sT}_{1}$ consists of 3 types of weighted rooted trees as there are 3 types of pairs in $\sP_1$. For each type, $\Phi_1$ is as follows.
    \[\qquad
    %
    \]
where the first two types form the full lists of weighted rooted trees of types III and V respectively. The weighted rooted trees of the last type are those of type II whose roots have at most $\frac{n}{2}$ inputs.

    \item $\bar{\sT}_2$ consists of weighted rooted trees of one type obtained by
\[
    %
\]
which are precisely the weighted rooted trees of type XII with no inputs at the roots.

\item $\bar{\sT}_3$ consists of 4 types of weighted rooted trees as there are 4 types of pairs in $\sP_3$. For each type, $\Phi_3$ is as follows.
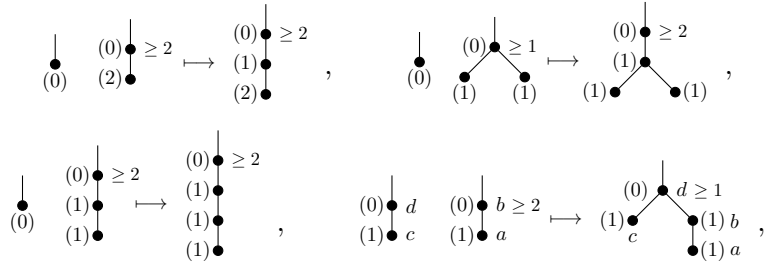
\[
    %
    \]
    where the first three types correspond to the trees of types VII, XI and X having at least 2 inputs at the roots respectively. On the other hand, the last type corresponds to those of type XII, which satisfy $d\geq1$ and either $a+b>c+d$ or $a+b=c+d$ and $b\geq d$ where $a,b,c,d$ denote the number of inputs at the vertices as in the figure. Indeed, this is the necessary and sufficient condition for $T_1\leq T_2$ for a pair $(T_1,T_2)$ of the last type.

    \item $\bar{\sT}_4$ consists of 4 types of weighted rooted trees as there are 4 types of pairs in $\sP_4$. For each type, $\Phi_4$ is as follows.
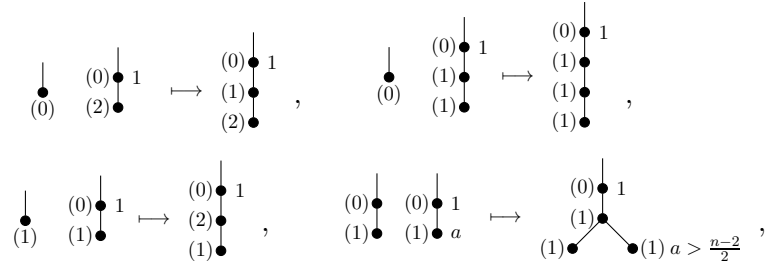
\[
    %
\]
where the first three types correspond to the trees of types VII, X and VI having precisely 1 input at the roots, respectively. On the other hand, the last type corresponds to those of type XI, one of whose two non-root vertices at the ends has more than $\frac{n-2}{2}$ inputs. Indeed, for a pair $(T_1,T_2)\in \sP_4$ of the last type, the condition $T_1<T_2$ is equivalent to $a>\frac{n-2}{2}$ while $T_1,T_2$ cannot be identical in $\sP_4$.

    \item $\bar{\sT}_6$ and $\bar{\sT}_8$ consist of weighted rooted trees of one type obtained by
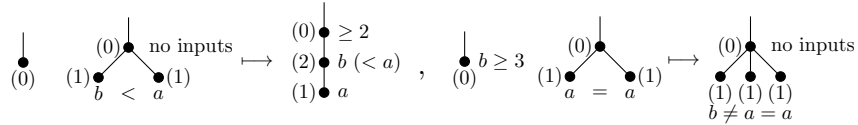
\[\qquad \quad
    %
\]
respectively. Hence, $\bar{\sT}_6$ consists of trees of type VI with at least 2 inputs at the roots, which satisfy $a>b$ for the numbers $a, b$ of inputs at the vertices of weights 1, 2 respectively. $\bar{\sT}_8$ consists of the trees of type XIII having no inputs at the roots and only two of whose non-root vertices have the same number of inputs.

    \item Each of $\bar{\sT}_9, \bar{\sT}_{10}$ consists of at most one weighted rooted tree of type VIII obtained by
\[\qquad
    %
\]
respectively. $\bar{\sT}_9$ is empty unless $n$ is a multiple of 3 and at least 9 and $\bar{\sT}_{10}$ is empty when $n$ is odd.
\end{enumerate}
Therefore $\bar{\sT}\setminus \bar{\sT}_0$ contains all the weighted rooted trees of types III, V, VII, X and no trees of types I, IV, IX. For the other types, it contains some weighted rooted trees of such types as described above but not all of such.

On the other hand, $\bar{\sT}_0$ consists of the first two trees of the following
\[
    %
\]
while $\bar{\sT}^{(2)}$, $\bar{\sT}'$ are singletons of the third and the last trees respectively when $n$ is even. When $n$ is odd, all are empty. Furthermore, $\bar{\sT}^{(1)}=\emptyset$. Hence by \eqref{Eq.SubseqDeg},
\begin{equation*}
            %
\]
where for types II, VI, VIII, XI, XII, XIII the corresponding summands are sums over weighted rooted trees in $\sT_{n,p}/S_n$ of given types satisfying the following further conditions:
\begin{enumerate}
    \item[II$'$:] $a<b$, i.e. $a<\frac{n}{2}$,
    \item [VI$'$:] $a\leq b$, $c\geq 2$ and $(a,b,c)\neq(\frac{n-2}{2},\frac{n-2}{2},2)$,
    \item[VIII$'$:] $(a,b,c)\neq (\frac{n}{3},\frac{n}{3},\frac{n}{3})$,
    \item[XI$'$:] $a \leq b\leq \frac{n-2}{2}$, $d=1$ and $(a,b,c,d)\neq (\frac{n-2}{2},\frac{n-2}{2},1,1)$,
    \item[XII$'$:] Either (i) $a+b<\frac{n}{2}$ and $d\geq 1$ or (ii) $a+b=\frac{n}{2}$ and $b<d$,
    \item[XIII$'$:] (i) $d\geq 1$, (ii) $d=0$ and $a<b<c$, or (iii) $(a,b,c,d)=(\frac{n}{3},\frac{n}{3},\frac{n}{3},0)$.
\end{enumerate}
We denote by $V_\mathrm{I}, V_{\mathrm{II}'}$, etc. the summands indexed by types I, II$'$, etc. We subtract $A^2(\overline{\mathcal{M}}_{0,n})$ from $A^2(\overline{\mathcal{M}}_{0,n})\oplus A^3(\overline{\mathcal{M}}_{0,n})$. By \eqref{Eq.SubseqDeg2}, we compute
\begin{equation*}
    \begin{tikzpicture} [scale=.5,auto=left,every node/.style={scale=1}]
      \tikzset{Bullet/.style={circle,draw,fill=black,scale=0.5}}
      \node[] (n) at (0,0) {};
      \node[] at (0,-1) {};
      \draw[black] (n) node[] {$A^2(\overline{\mathcal{M}}_{0,n})=$};
    \end{tikzpicture}
    \begin{tikzpicture} [scale=.4,auto=left,every node/.style={scale=0.7}]
      \tikzset{Bullet/.style={circle,draw,fill=black,scale=0.5}}
      \node[Bullet] (n1) at (0,0) {};
      \node[] at (0,-1) {};

      \draw[black] (n1) -- (0,1);
      \draw[black] (0,-0.2) node[below] {$ $};
      \draw[black] (n1) node[left] {$(2)$};
    \end{tikzpicture}
    \begin{tikzpicture} [scale=.5,auto=left,every node/.style={scale=1}]
      \tikzset{Bullet/.style={circle,draw,fill=black,scale=0.5}}
      \node[] (n) at (0,0) {};
      \node[] at (0,-1) {};

      \draw[black] (n) node[] {$+$};
    \end{tikzpicture}
\begin{tikzpicture} [scale=.4,auto=left,every node/.style={scale=0.7}]
      \tikzset{Bullet/.style={circle,draw,fill=black,scale=0.5}}
      \node[Bullet] (n1) at (1,-1) {};
      \node[Bullet] (n2) at (1,-2) {};

      \draw[black] (n1) -- (1,0);
      \draw[black] (n1) -- (n2);
      \draw[black] (n1) node[right] {$ $};
      \draw[black] (n1) node[left] {$(0)$};
      \draw[black] (n2) node[right] {$ $};
      \draw[black] (n2) node[left] {$(2)$};
    \end{tikzpicture}
    \begin{tikzpicture} [scale=.5,auto=left,every node/.style={scale=1}]
      \tikzset{Bullet/.style={circle,draw,fill=black,scale=0.5}}
      \node[] (n) at (0,0) {};
      \node[] at (0,-1) {};

      \draw[black] (n) node[] {$+$};
    \end{tikzpicture}
\begin{tikzpicture} [scale=.4,auto=left,every node/.style={scale=0.7}]
      \tikzset{Bullet/.style={circle,draw,fill=black,scale=0.5}}
      \node[Bullet] (n1) at (1,-1) {};
      \node[Bullet] (n2) at (0,-2) {};
      \node[Bullet] (n3) at (2,-2) {};

      \draw[black] (n1) -- (1,0);
      \draw[black] (n1) -- (n2);
      \draw[black] (n1) -- (n3);
      \draw[black] (n1) node[right] {$ $};
      \draw[black] (n1) node[left] {$(0)$};
      \draw[black] (n2) node[right] {$ $};
      \draw[black] (n2) node[left] {$(1)$};
      \draw[black] (n3) node[right] {$ $};
      \draw[black] (n3) node[right] {$(1)$};
    \end{tikzpicture}
    \begin{tikzpicture} [scale=.5,auto=left,every node/.style={scale=1}]
      \tikzset{Bullet/.style={circle,draw,fill=black,scale=0.5}}
      \node[] (n) at (0,0) {};
      \node[] at (0,-1) {};
      \draw[black] (n) node[] {$-~A^1(\overline{\mathcal{M}}_{0,n})$};
    \end{tikzpicture}
    \qquad\qquad\qquad \qquad \quad
    \end{equation*}
    \begin{equation*}\quad \,\,\,
    \begin{tikzpicture} [scale=.5,auto=left,every node/.style={scale=1}]
      \tikzset{Bullet/.style={circle,draw,fill=black,scale=0.5}}
      \node[] (n) at (0,0) {};
      \node[] at (0,-1) {};
      \draw[black] (n) node[] {$=$};
    \end{tikzpicture}
    \begin{tikzpicture} [scale=.4,auto=left,every node/.style={scale=0.7}]
      \tikzset{Bullet/.style={circle,draw,fill=black,scale=0.5}}
      \node[Bullet] (n1) at (0,0) {};
      \node[] at (0,-1) {};

      \draw[black] (n1) -- (0,1);
      \draw[black] (0,-0.2) node[below] {$ $};
      \draw[black] (n1) node[left] {$(3)$};
    \end{tikzpicture}
    \begin{tikzpicture} [scale=.5,auto=left,every node/.style={scale=1}]
      \tikzset{Bullet/.style={circle,draw,fill=black,scale=0.5}}
      \node[] (n) at (0,0) {};
      \node[] at (0,-1) {};

      \draw[black] (n) node[] {$+$};
    \end{tikzpicture}
    \begin{tikzpicture} [scale=.4,auto=left,every node/.style={scale=0.7}]
      \tikzset{Bullet/.style={circle,draw,fill=black,scale=0.5}}
      \node[Bullet] (n1) at (1,-1) {};
      \node[Bullet] (n2) at (1,-2) {};

      \draw[black] (n1) -- (1,0);
      \draw[black] (n1) -- (n2);
      \draw[black] (n1) node[right] {$ $};
      \draw[black] (n1) node[left] {$(0)$};
      \draw[black] (n2) node[right] {$ $};
      \draw[black] (n2) node[left] {$(3)$};
    \end{tikzpicture}
    \begin{tikzpicture} [scale=.5,auto=left,every node/.style={scale=1}]
      \tikzset{Bullet/.style={circle,draw,fill=black,scale=0.5}}
      \node[] (n) at (0,0) {};
      \node[] at (0,-1) {};

      \draw[black] (n) node[] {$+$};
    \end{tikzpicture}
    \begin{tikzpicture} [scale=.4,auto=left,every node/.style={scale=0.7}]
      \tikzset{Bullet/.style={circle,draw,fill=black,scale=0.5}}
      \node[Bullet] (n1) at (1,-1) {};
      \node[Bullet] (n2) at (1,-2) {};

      \draw[black] (n1) -- (1,0);
      \draw[black] (n1) -- (n2);
      \draw[black] (n1) node[right] {$~n-4$};
      \draw[black] (n1) node[left] {$(0)$};
      \draw[black] (n2) node[right] {$~4$};
      \draw[black] (n2) node[left] {$(2)$};
    \end{tikzpicture}
    \begin{tikzpicture} [scale=.5,auto=left,every node/.style={scale=1}]
      \tikzset{Bullet/.style={circle,draw,fill=black,scale=0.5}}
      \node[] (n) at (0,0) {};
      \node[] at (0,-1) {};

      \draw[black] (n) node[] {$+$};
    \end{tikzpicture}
\begin{tikzpicture} [scale=.4,auto=left,every node/.style={scale=0.7}]
      \tikzset{Bullet/.style={circle,draw,fill=black,scale=0.5}}
      \node[Bullet] (n1) at (1,-1) {};
      \node[Bullet] (n2) at (0,-2) {};
      \node[Bullet] (n3) at (2,-2) {};

      \draw[black] (n1) -- (1,0);
      \draw[black] (n1) -- (n2);
      \draw[black] (n1) -- (n3);
      \draw[black] (n1) node[right] {$~$no inputs};
      \draw[black] (n1) node[left] {$(0)$};
      \draw[black] (n2) node[right] {$ $};
      \draw[black] (n2) node[left] {$(1)$};
      \draw[black] (n3) node[right] {$ $};
      \draw[black] (n3) node[right] {$(1)$};
    \end{tikzpicture}\begin{tikzpicture} [scale=.5,auto=left,every node/.style={scale=1}]
      \tikzset{Bullet/.style={circle,draw,fill=black,scale=0.5}}
      \node[] (n) at (0,0) {};
      \node[] at (0,-1) {};

      \draw[black] (n) node[] {$+$};
    \end{tikzpicture}
\begin{tikzpicture} [scale=.4,auto=left,every node/.style={scale=0.7}]
      \tikzset{Bullet/.style={circle,draw,fill=black,scale=0.5}}
      \node[Bullet] (n1) at (1,-1) {};
      \node[Bullet] (n2) at (0,-2) {};
      \node[Bullet] (n3) at (2,-2) {};

      \draw[black] (n1) -- (1,0);
      \draw[black] (n1) -- (n2);
      \draw[black] (n1) -- (n3);
      \draw[black] (n1) node[right] {$ $};
      \draw[black] (n1) node[left] {$(1)$};
      \draw[black] (n2) node[right] {$ $};
      \draw[black] (n2) node[left] {$(1)$};
      \draw[black] (n3) node[right] {$ $};
      \draw[black] (n3) node[right] {$(1)$};
    \end{tikzpicture}
    \begin{tikzpicture} [scale=.5,auto=left,every node/.style={scale=1}]
      \tikzset{Bullet/.style={circle,draw,fill=black,scale=0.5}}
      \node[] (n) at (0,0) {};
      \node[] at (0,-1) {};
      \draw[black] (n) node[] {$-~A^1(\overline{\mathcal{M}}_{0,n})$};
    \end{tikzpicture}
\end{equation*}
\[\quad=V_\mathrm{I}+V_\mathrm{IV}+U_{4,n-4}+\left(V_{\mathrm{II}'}+s_{(2)}\circ U_{\frac{n}{2}}\right)+\left(V_{\mathrm{VIII}'}+U_{\frac{n}{3},\frac{n}{3},\frac{n}{3}}^{S_2}\right)-A^1(\overline{\mathcal{M}}_{0,n}).\]
Hence we have
\begin{equation*}
    \begin{split}
        A^3(\overline{\mathcal{M}}_{0,n})=&~ A^1(\overline{\mathcal{M}}_{0,n})-U_{4,n-4}+V_{\mathrm{VI}'}+V_{\mathrm{IX}}+V_{\mathrm{XI}'}+V_{\mathrm{XII}'}+V_{\mathrm{XIII}'}\\
        & -U_{\frac{n}{3},\frac{n}{3},\frac{n}{3}}^{S_2}+U_{\frac{n-2}{2},\frac{n-2}{2},2}^{S_2}+\bigoplus_{3\leq a\leq \frac{n}{2}-2}s_{(2)}\circ U_{a,\frac{n}{2}-a}.
    \end{split}
\end{equation*}
Since $A^1(\overline{\mathcal{M}}_{0,n})-U_{4,n-4}=\bigoplus_{a\geq 6 \text{ even}}U_{a,n-a}$, this proves  Corollary \ref{Cor.SingleDeg} (3) as desired.

The claim that $A^3(\overline{\mathcal{M}}_{0,n})$ is a permutation representation for $n=3k\geq 9$
follows from the identity
\[U_{k,k,k}^{S_2}\oplus \bigoplus_{a<k \text{ odd}}U_{a,k,2k-a}=\bigoplus_{a\leq k \text{ even}}U_{a,k,2k-a},\]
which is obtained by taking the induced representation of \eqref{Identity.OddEven}. Indeed,
\[U_{2,k,2k-2} \quad \text{ and } \quad \bigoplus_{4\leq a\leq k, \text{ even}}U_{a,k,2k-a}\]
appear as the summands of $V_{\mathrm{VI}'}$ and $V_{\mathrm{IX}}$ respectively. This completes our proof of Corollary \ref{Cor.SingleDeg}.

\bigskip

\section{K-theoretic cuspidal block}\label{sec:cuspidalblock} The above investigation on $S_n$-representations on $A^*(\overline{\mathcal{M}}_{0,n})$ has an application on an $S_n$-equivariant decomposition of the Grothendieck group $K_0(\overline{\mathcal{M}}_{0,n})$ of coherent sheaves on $\overline{\mathcal{M}}_{0,n}$, which is geometrically meaningful.

We will define the K-theoretic cuspidal block $K_{0}(\overline{\mathcal{M}}_{0,n})_{cusp}$ of $\overline{\mathcal{M}}_{0,n}$, which is the K-theoretic analogue of the cuspidal block defined on the derived category by Castravet and Tevelev \cite{CT1}. We will see that $K_0(\overline{\mathcal{M}}_{0,n})$ admits an $S_n$-equivariant decomposition by $K_{0}(\overline{\mathcal{M}}_{0,k})_{cusp}$ with $k \leq n$. We will give a closed formula for the $S_n$-character of $K_{0}(\overline{\mathcal{M}}_{0,n})_{cusp}$ and check that it is a permutation representation for $n\leq 8$.

In this section, we only consider Grothendieck groups with rational coefficients.

\subsection{K-theoretic cuspidal block of $\overline{\mathcal{M}}_{0,n}$}
In \cite{CT1}, Castravet and Tevelev defined the \emph{cuspidal block}
\[\mathcal{D}^b_{cusp}(\overline{\mathcal{M}}_{0,n}) \subset \mathcal{D}^b(\overline{\mathcal{M}}_{0,n})\]
of $\overline{\mathcal{M}}_{0,n}$ as the full triangulated subcategory of the bounded derived category $\mathcal{D}^b(\overline{\mathcal{M}}_{0,n})$ of coherent sheaves on $\overline{\mathcal{M}}_{0,n}$ consisting of \emph{cuspidal} objects. By definition, a cuspidal object is the object $E$ satisfying
\[R\pi_{i*}E=0 \quad \text{~for~} i=1,\cdots, n\] where $\pi_i:\overline{\mathcal{M}}_{0,n}\xrightarrow{}\overline{\mathcal{M}}_{0,n-1}$ is the morphism which forgets the $i$-th marking.

They showed that $\mathcal{D}^b(\overline{\mathcal{M}}_{0,n})$ admits a semi-orthogonal decomposition by the pullbacks of $\mathcal{D}^b_{cusp}(\overline{\mathcal{M}}_{0,k})$ with $k\leq n$ in an $S_n$-equivariant way \cite[Proposition 1.6]{CT1}, and raised a question whether there exists an $S_n$-invariant full exceptional collection in $\mathcal{D}^b_{cusp}(\overline{\mathcal{M}}_{0,n})$ \cite[Question 1.5]{CT1}. An affirmative answer for this question will give a new $S_n$-equivariant decomposition of $\mathcal{D}^b(\overline{\mathcal{M}}_{0,n})$.

Let us consider K-theoretic analogues.
\begin{definition}  The \emph{K-theoretic cuspidal block} of $\mzn$ is defined as
\[K_{0}(\overline{\mathcal{M}}_{0,n})_{cusp}:=\{E\in K_0(\overline{\mathcal{M}}_{0,n})~|~(\pi_i)_!E=0 \text{ all }i\}.\]
\end{definition}

For a finite set $S$, we denote by $\overline{\mathcal{M}}_{0,S}$ the moduli space of stable rational curves with $|S|$ points marked by $S$. We identify $\overline{\mathcal{M}}_{0,n}$ with $\overline{\mathcal{M}}_{0,[n]}$. For each $J\subset [n]$ with $|J|\leq n-4$, we denote by $\pi_J:\overline{\mathcal{M}}_{0,n}\xrightarrow{} \overline{\mathcal{M}}_{0,[n]\setminus J}$ the forgetful morphism which forgets the markings by $J$. We let $\pi_\emptyset =\mathrm{id}$ for convenience. The following is essentially proved in \cite[\S3]{CT1}.
\begin{proposition} \label{Prop.Cusp}
There is an $S_n$-equivariant isomorphism
\begin{equation}
    K_0(\overline{\mathcal{M}}_{0,n})=\Q\oplus \bigoplus_{J\subset [n], |J|\leq n-4}\pi_J^* K_{0}(\overline{\mathcal{M}}_{0,[n]\setminus J})_{cusp}.
\end{equation}
where $\Q$ is generated by the structure sheaf $\mathcal{O}_{\overline{\mathcal{M}}_{0,n}}$.
\end{proposition}
\begin{proof} First note that $(\pi_i)_!$ is a left inverse of $\pi_i^*$ for each $i$. Hence we have the decomposition
\beq\label{70}
K_0(\overline{\mathcal{M}}_{0,n})= \pi_i^*K_0(\overline{\mathcal{M}}_{0,n-1})\oplus \{E\in K_0(\overline{\mathcal{M}}_{0,n}) ~|~(\pi_i)_!E=0\}.\eeq
Now the assertion follows by taking intersections of such decompositions for $i=1,\cdots,n$, which make sense
because
\beq\label{71}
\pi_j^*(\pi_j)_!\pi_i^*K_0(\overline{\mathcal{M}}_{0,n-1}) \subset \pi_i^*K_0(\overline{\mathcal{M}}_{0,n-1}), \quad
\text{for all }1\leq i,j \leq n.\eeq
Note that \eqref{71} follows from the equality $(\pi_j)_!\pi_i^*=\pi_i^*(\pi_j)_!$ which holds by
 the flat base change theorem and the fact that $(\pi_i)_!\mathcal{O}_{\overline{\mathcal{M}}_{0,n}}=\mathcal{O}_{\overline{\mathcal{M}}_{0,n-1}}$.
\end{proof}

\subsection{$S_n$-character of the cuspidal block}
Using Proposition~\ref{Prop.Cusp}, we compute the $S_n$-character of $K_{0}(\overline{\mathcal{M}}_{0,n})_{cusp}$.
The following transitive permutation representations will be useful.

Let $\sX_n$ be the set of increasing sequences $x=(x_1,\cdots,x_r)$ of integers with $4\leq x_1<x_2<\cdots<x_r= n$.
\begin{definition}
For $x=(x_1,\cdots,x_r)\in \sX_n$, we define two $S_n$-representations
\begin{equation*}
    \begin{split}
    & U_x:=  \begin{cases}
    U_{x_1,x_2-x_1,\cdots,x_r-x_{r-1}} &\text{ if } r\geq2\\
    U_n &\text{ if } r=1,
    \end{cases} \\
    & M_x:=  \begin{cases}
    K_0(\overline{\mathcal{M}}_{0,x_1}).U_{x_2-x_1,\cdots,x_r-x_{r-1}} &\text{ if } r\geq2\\
    K_0(\overline{\mathcal{M}}_{0,n}) &\text{ if } r=1.
    \end{cases}
    \end{split}
\end{equation*}
\end{definition}
For each $x=(x_1,\cdots,x_r)\in \sX_n$, let $l(x):=r-1$ denote the \emph{length} of $x$. Then $K_{0}(\overline{\mathcal{M}}_{0,n})_{cusp}$ is written as a signed sum of $U_x$ and $M_x$ for various $x\in \sX_n$ as follows.
\begin{proposition} For $n\geq4$,
    \[K_0(\overline{\mathcal{M}}_{0,n})_{cusp}=\sum_{x\in \sX_n}(-1)^{l(x)}\Big(M_x-U_x\Big).\]
\end{proposition}
\begin{proof} Since $K_0(\overline{\mathcal{M}}_{0,n})\cong A^*(\overline{\mathcal{M}}_{0,n})$ by Grothendieck-Riemann-Roch, it follows directly from Proposition~\ref{Prop.Cusp} that
\begin{equation} \label{Eq.Cusp}
    K_0(\overline{\mathcal{M}}_{0,n})_{cusp}=K_0(\overline{\mathcal{M}}_{0,n})-U_n-\sum_{i=4}^{n-1}U_{n-i}.K_0(\overline{\mathcal{M}}_{0,i})_{cusp}
\end{equation}
Now a repeated application of \eqref{Eq.Cusp} proves the assertion. \end{proof}

For $n\leq 8$, we compute the $S_n$-representations on $K_0(\overline{\mathcal{M}}_{0,n})_{cusp}$ in Table \ref{Table.Cusp} and find that they are permutation representations.

\begin{table}[h]
\centering
\label{Table.Cusp}
\begin{tabular}{clc}
\toprule
$n$ & \qquad \qquad \qquad $K_0(\overline{\mathcal{M}}_{0,n})_{cusp}$ \\
\midrule
4 &  \quad $U_4$\\
5  & \quad $U_5$\\
6  & \quad $U_6^{\oplus 2}\oplus U_{3,3}^{S_2}$\\
7  & \quad $U_7^{\oplus 2}\oplus U_{3,4}^{\oplus 2}$\\
8  & \quad $U_8^{\oplus 3}\oplus U_{3,5}^{\oplus 3} \oplus U_{4,4}\oplus (U_{4,4}^{S_2})^{\oplus 4}\oplus U_{3,3,2}^{S_2}$\\
\bottomrule
\noalign{\smallskip}\noalign{\smallskip}
\end{tabular}
\caption{$K_0(\overline{\mathcal{M}}_{0,n})_{cusp}$ for $n\leq 8$}
\end{table}

The following question is a K-theoretic analogue of \cite[Question 1.5]{CT1}.
\begin{question}
    Is $K_0(\overline{\mathcal{M}}_{0,n})_{cusp}$ a permutation representation of $S_n$ for every $n\geq4$?
\end{question}

An affirmative answer to this question would give supporting evidence for the existence of an $S_n$-invariant full exceptional collection on the cuspidal block $\mathcal{D}^b_{cusp}(\overline{\mathcal{M}}_{0,n})$, as well as a new $S_n$-equivariant decomposition of $K_0(\overline{\mathcal{M}}_{0,n})$.

\appendix

\section{Comparison with \cite{CGK}} \label{App.CGK}

In this section, we compare \eqref{22} with the blowup construction of $\overline{\mathcal{M}}_{0,n+1}$ from $\overline{\mathcal{M}}_{0,n}$ in \cite{CGK}. We explain in detail that the construction in \cite{CGK} is not $S_n$-equivariant and the blowup is different from \eqref{22}. See Example~\ref{CGK.ex} for a case of small $n$.

In \cite{CGK}, Chen, Gibney and Krashen constructed the moduli space $T_n:=T_{1,n}$ of $n$-pointed rooted trees of projective lines, and proved the following theorems.
\begin{theorem} \cite[Proposition 3.4.3]{CGK} \label{CGKthm}
	$T_{n}\cong \overline{\mathcal{M}}_{0,n+1}$ for $n\geq 2$.
\end{theorem}
By definition, $T_{n}$ is the sublocus of the Fulton-MacPherson compactification $\PP^1[n]$ of the configuration space of $n$ marked points in $\PP^1$ where the $n$ marked points collide to a given point $x$ in $\PP^1$ (\cite[Definition~3.3.2]{CGK}). Equivalently, under the identification $\PP^1[n]\cong \overline{\mathcal{M}}_{0,n}(\PP^1,1)$,
\beq \label{CGK.2}T_{n}= \bigcap_{i=1}^n ev_i^{-1}(x)\eeq
where $ev_i:\overline{\mathcal{M}}_{0,n}(\PP^1,1) \to \PP^1$ is the evaluation map at the $i$-th marking.

Let $f:(C, p_1, \ldots, p_n)\to \PP^1$ be a stable map in $T_n$ under the identification \eqref{CGK.2}. Suppose that $C$ has irreducible components $C', D_1, \ldots, D_r$ with $f_*([C'])=1$ and $f_*([D_i])=0$. Note that none of the marked points $p_i$ may lie on $C'$ and $C'$ has a unique node attached to $\cup_i D_i$. By forgetting $C'$, we get an $(n+1)$-pointed stable curve where the $(n+1)$st marking is given by the attaching node of $C'$. This gives the isomorphism in Theorem \ref{CGKthm}. The $(n+1)$st marking corresponds to the \emph{root} on an $n$-pointed rooted tree of projective lines (\cite[Definition 2.0.1]{CGK}).

Fix $n\geq 2$. Let $N=\{1,\ldots,n\}$ and $N^+=N\cup \{n+1\}$. For a subset $S\subset N$, let $S^+=S\cup \{n+1\}$. Denote by $\bullet$ the root on each $n$-pointed rooted tree of projective lines.
We identify
\begin{equation*}
	\begin{split}
		T_{n}&\cong \overline{\mathcal{M}}_{0,N\cup\{\bullet\}} ~(\cong \overline{\mathcal{M}}_{0,n+1}) \ \ \text{ and}\\
		T_{n+1}&\cong \overline{\mathcal{M}}_{0,N^+\cup\{\bullet\}}~(\cong \overline{\mathcal{M}}_{0,n+2}).
	\end{split}
\end{equation*}

\begin{theorem} \cite[Theorem 3.3.1]{CGK}
	The forgetful morphism
	\beq \label{CGK.1} \overline{\mathcal{M}}_{0,N^+\cup\{\bullet\}}\lra \overline{\mathcal{M}}_{0,N\cup\{\bullet\}}\eeq
	which forgets the $(n+1)$st marking, admits the factorization
	\beq \label{fact.CGK} T_{n+1}=F^n_{n}\xrightarrow{b_{n-1}}F^{n-1}_{n}\xrightarrow{b_{n-2}}\cdots \xrightarrow{~b_1~}F^1_{n}\xrightarrow{~b_0~}F^0_{n}=T_{n},\eeq
	where
	\begin{enumerate}
            \item for $0\leq i\leq n$, $F_n^i$ is smooth and it has subvarieties $F_n^i(S)$ indexed by $S\subset N^+$ with $\lvert S\rvert\ge 2$,
            \item the morphism $b_0$ is a $\PP^1$-bundle,
		\item for $1\leq i \leq n-1$, the morphism $b_i$ is the blowup along the disjoint union of $F_n^i(S^+)$ for $S\subset N$ with $\lvert S\rvert=n-i$,
		\item $F_n^i(S^+) \cong \overline{\mathcal{M}}_{0, S\cup\{\star\}}\times \overline{\mathcal{M}}_{0,S^c\cup\{\star, \bullet\}}$ if $\lvert S\rvert=n-i\ne 1$,
            \item $F_n^{n-1}(S^+) \cong \overline{\mathcal{M}}_{0,N \cup\{\bullet\}}$ if $\lvert S\rvert=1$.
	\end{enumerate}
\end{theorem}
The factorization \eqref{fact.CGK} is obtained as the restriction to $T_{n+1}$ of the construction of the Fulton-MacPherson configuration space in \cite{FM}
\[\PP^1[n+1]=:\PP^1[n,n]\xrightarrow{~\rho_n~}\PP^1[n,n-1]\xrightarrow{\rho_{n-1}}\cdots \xrightarrow{\rho_1}\PP^1[n,0]:=\PP^1[n]\times \PP^1, \]
where the composition $\PP^1[n+1]\to \PP^1[n]\times \PP^1$ is given by the forgetful map of the $(n+1)$st marking and the evaluation map $ev_{n+1}$. Here, the image $T_n$ of $T_{n+1}$ sits inside $\PP^1[n]\times\{x\}$ for $x\in \PP^1$. See \cite[\S3.1]{CGK}.

The above map $\PP^1[n+1]\to \PP^1[n]\times \PP^1$ can be viewed as the contraction of the universal curve over $\PP^1[n]\cong \overline{\mathcal{M}}_{0,n}(\PP^1,1)$ to the degree 1 component. 
Furthermore, for each $k$, the composition $\rho_{n-k+1}\circ \cdots \circ \rho_n$ contracts all the connected degree 0 subcurves $D$ of the universal curve $C$ such that $D$ has precisely $k$ markings and $C-D$ is connected, to the attaching node $D\cap \overline{C-D}$.

Likewise, for each $k$, the composition $b_{n-k}\circ \cdots \circ b_{n-1}$ contracts all the connected subcurves $D$ of the universal curve $C$ over $T_n\cong \overline{\mathcal{M}}_{0,N\cup\{\bullet\}}$ such that $D$ has precisely $k$ markings in $N$ but not the root $\bullet$ and $C-D$ is connected, to the attaching node $D\cap \overline{C-D}$.

Note that in our case, $\rho_n$ and $b_{n-1}$ are isomorphisms since $\dim ~\PP^1=1$.

\medskip

It is obvious by construction that \eqref{fact.CGK} is equivariant under the action of $S_N\cong S_n$ which permutes the markings indexed by $N$ and fixes the root $\bullet$. However, it is not equivariant under the action of $S_{N\cup\{\bullet\}}\cong S_{n+1}$, which is the full automorphism group of $T_n$ for $n\geq 4$.

\medskip
\noindent\textbf{Claim.} \eqref{fact.CGK} is not equivariant under the action of $S_{N\cup\{\bullet\}}\cong S_{n+1}$.
\begin{proof}
	By the above description, the composition
	\[b_1\circ \cdots \circ b_{n-1}: \overline{\mathcal{M}}_{0,N^+\cup\{\bullet\}}\lra F^1_n\]
    contracts the universal curve over $\overline{\mathcal{M}}_{0,N\cup\{\bullet\}}$ to the irreducible component containing the root $\bullet$. Since this component is not preserved by the $S_{N\cup\{\bullet\}}$-action, \eqref{fact.CGK} is not $S_{N\cup\{\bullet\}}$-equivariant.
\end{proof}

\smallskip

In particular, \eqref{fact.CGK} is different from the factorization \eqref{22}
\[
\overline{\mathcal{M}}_{0,n+2}\cong \fQ^{\d=\infty}\xrightarrow{~\fq_2~ } \fQ^{\d_3^+}\xrightarrow{~\fq_3~ }\cdots \xrightarrow{~\fq_{\ell-1}~ }\fQ^{\d_\ell^+} \xrightarrow{~\fq_\ell~} \fQ^{\d=0^+}\xrightarrow{~\fp~}\overline{\mathcal{M}}_{0,n+1}~\]
in Theorem~\ref{26} with $\ell=\lfloor \frac{n}{2}\rfloor$, as this is $S_{n+1}$-equivariant.
Here, $\fp$
parametrizes the \emph{balanced components} of the universal curve over $\overline{\mathcal{M}}_{0,n+1}$. See Remark~\ref{rmk:delta0}.
\begin{example}[$n=4$]\label{CGK.ex}
	Let $N=\{1,2,3,4\}$ and $N^+=N\cup \{5\}$ so that
	\[T_{5}\cong \overline{\mathcal{M}}_{0,N^+\cup \{\bullet\}} \cong \overline{\mathcal{M}}_{0,6} \and T_4\cong \overline{\mathcal{M}}_{0,N\cup\{\bullet\}}\cong \overline{\mathcal{M}}_{0,5}.\]
	 For each subset $I\subset N^+\cup\{\bullet\}$ with $2\leq \lvert I\rvert \leq 4$, denote by
	 \[D_{I,I^c} \cong \overline{\mathcal{M}}_{0,I\cup\{\star\}}\times \overline{\mathcal{M}}_{0,I^c\cup\{\star\}}\]
	 the boundary divisor of $\overline{\mathcal{M}}_{0,N^+\cup \{\bullet\}}$, which generically parametrizes nodal curves with two components, one with markings of $I$ and the other with markings of $I^c=N^+\cup\{\bullet\}-I$. Here we denote by $\star$ the attaching node.
	 \begin{itemize}
	 	\item In \eqref{fact.CGK}, the composition $T_5\to F^1_4$ contracts precisely the boundary divisors
	 		\[D_{I\cup\{5\},(N-I)\cup\{\bullet\}} \quad \text{ for }~ I\subset N,~2\leq \lvert I\rvert \leq 3.\]
	 		In particular, \eqref{fact.CGK} is not $S_{N\cup \{\bullet\}}$-equivariant, although one can easily check that it is $S_N$-equivariant.
	 	\item In \eqref{22}, the composition $\overline{\mathcal{M}}_{0,6}\to \fQ^{\d=0^+}$ contracts precisely the boundary divisors
	 		\[D_{I\cup\{5\},(N-I)} \quad \text{ for }I\subset N\cup\{\bullet\}, ~ \lvert I \rvert =2.\]
	 		Note that this also contracts such divisors with $I=\{i, \bullet\}$ for $i\in N$. 
	 \end{itemize}
	Therefore, two above contractions onto $\PP^1$-bundles over $\overline{\mathcal{M}}_{0,5}$ are different, and \eqref{fact.CGK} is not equivariant with respect to the actions of $S_{N\cup\{\bullet\}}\cong S_5$.
\end{example}

\clearpage
\section{Character of $A^k(\overline{\mathcal{M}}_{0,n})$} \label{App.A}
We list the characters of the $S_n$-representation $A^k(\overline{\mathcal{M}}_{0,n})$ for $n\leq 11$. Note that $\ch(A^0(\overline{\mathcal{M}}_{0,n}))=s_{(n)}$ for all $n$ and $\ch(A^{n-3-k}(\overline{\mathcal{M}}_{0,n}))=\ch(A^k(\overline{\mathcal{M}}_{0,n}))$.
\medskip
\begin{center}
{\small
\begin{tabular}{|c|c|l|}
\hline
$n$ &$k$ & \parbox[c][1.5em]{11cm}{$\ch_n(A^k(\overline{\mathcal{M}}_{0,n}))$}\\
\hline
5 & 1 &
\parbox[c][1.5em]{11cm}{$s_{(5)} + s_{(4, 1)}$
} \\
\hline
6 & 1 &
\parbox[c][1.5em]{11cm}{$2 s_{(6)} + s_{(4, 2)} + s_{(5, 1)}$
}  \\
\hline
\multirow{2}{*}{$7$} & $1$ &
\parbox[c][1.5em]{11cm}{$2 s_{(7)} + s_{(4, 3)} + s_{(5, 2)} + 2 s_{(6, 1)}$ }\\
\cline{2-3}
& 2 &
\parbox[c][1.5em]{11cm}{$4 s_{(7)} + 2 s_{(4, 3)} + 3 s_{(5, 2)} + 3 s_{(6, 1)} + s_{(4, 2, 1)}$
} \\
\hline
\multirow{2}{*}{$8$} &  $1$ &
\parbox[c][1.5em]{11cm}{$3 s_{(8)} + s_{(4, 4)} + s_{(5, 3)} + 2 s_{(6, 2)} + 2 s_{(7, 1)}$
} \\
\cline{2-3}
 &  2 &
\parbox[c][1.5em]{11cm}{$6 s_{(8)} + 3 s_{(4, 4)} + 5 s_{(5, 3)} + 7 s_{(6, 2)} + 6 s_{(7, 1)} + s_{(4, 2, 2)} +
 2 s_{(4, 3, 1)} + 2 s_{(5, 2, 1)} + s_{(6, 1, 1)}$
}\\
\hline
\multirow{3}{*}{$9$} &  $1$ &
\parbox[c][1.5em]{11cm}{$3 s_{(9)} + s_{(5, 4)} + 2 s_{(6, 3)} + 2 s_{(7, 2)} + 3 s_{(8, 1)} $
}\\
\cline{2-3}
 &  2 &
\parbox[c][2.7em]{11cm}{$ 9 s_{(9)} + 7 s_{(5, 4)} + 11 s_{(6, 3)} + 12 s_{(7, 2)} + 10 s_{(8, 1)} +
 2 s_{(4, 3, 2)} + 3 s_{(4, 4, 1)} + 2 s_{(5, 2, 2)} + 4 s_{(5, 3, 1)} +
 5 s_{(6, 2, 1)} + 2 s_{(7, 1, 1)}$
}\\
\cline{2-3}
&  3 &
\parbox[c][2.7em]{11cm}{$11 s_{(9)} + 12 s_{(5, 4)} + 20 s_{(6, 3)} + 19 s_{(7, 2)} +
 16 s_{(8, 1)} + s_{(3, 3, 3)} + 5 s_{(4, 3, 2)} +
 6 s_{(4, 4, 1)} + 5 s_{(5, 2, 2)} + 11 s_{(5, 3, 1)} +
 10 s_{(6, 2, 1)} + 6 s_{(7, 1, 1)} + s_{(4, 2, 2, 1)} +
 s_{(4, 3, 1, 1)} + s_{(5, 2, 1, 1)} $
}\\
\hline

\multirow{3}{*}{$10$} &  $1$ &
\parbox[c][1.5em]{11cm}{$ 4 s_{(10)} + 2 s_{(6, 4)} + 2 s_{(7, 3)} + 3 s_{(8, 2)} + 3 s_{(9, 1)}$
}\\
\cline{2-3}
 &  2 &
\parbox[c][2.7em]{11cm}{$12 s_{(10)} + 4 s_{(5, 5)} + 15 s_{(6, 4)} + 18 s_{(7, 3)} + 19 s_{(8, 2)} +
 15 s_{(9, 1)} + s_{(4, 3, 3)} + 3 s_{(4, 4, 2)} + 4 s_{(5, 3, 2)} +
 6 s_{(5, 4, 1)} + 4 s_{(6, 2, 2)} + 9 s_{(6, 3, 1)} + 8 s_{(7, 2, 1)} +
 4 s_{(8, 1, 1)} $
}\\
\cline{2-3}
&  3 &
\parbox[c][5.1em]{11cm}{$21 s_{(10)} + 12 s_{(5, 5)} + 41 s_{(6, 4)} + 48 s_{(7, 3)} +
 45 s_{(8, 2)} + 32 s_{(9, 1)} + 6 s_{(4, 3, 3)} +
 13 s_{(4, 4, 2)} + 21 s_{(5, 3, 2)} + 26 s_{(5, 4, 1)} +
 17 s_{(6, 2, 2)} + 36 s_{(6, 3, 1)} + 30 s_{(7, 2, 1)} +
 14 s_{(8, 1, 1)} + s_{(3, 3, 3, 1)} + s_{(4, 2, 2, 2)} +
 4 s_{(4, 3, 2, 1)} + 3 s_{(4, 4, 1, 1)} + 4 s_{(5, 2, 2, 1)} +
 6 s_{(5, 3, 1, 1)} + 4 s_{(6, 2, 1, 1)} + 2 s_{(7, 1, 1, 1)} $
}\\
\hline

\multirow{3}{*}{$11$} &  $1$ &
\parbox[c][1.5em]{11cm}{$4 s_{(11)} + s_{(6, 5)} + 2 s_{(7, 4)} + 3 s_{(8, 3)} + 3 s_{(9, 2)} + 4 s_{(10, 1)} $
}\\
\cline{2-3}
 &  2 &
 \parbox[c][3.9em]{11cm}{$ 16 s_{(11)} + 13 s_{(6, 5)} + 24 s_{(7, 4)} + 28 s_{(8, 3)} + 27 s_{(9, 2)} +
 21 s_{(10, 1)} + 2 s_{(4, 4, 3)} + 2 s_{(5, 3, 3)} + 6 s_{(5, 4, 2)} +
 3 s_{(5, 5, 1)} + 8 s_{(6, 3, 2)} + 13 s_{(6, 4, 1)} + 6 s_{(7, 2, 2)} +
 14 s_{(7, 3, 1)} + 13 s_{(8, 2, 1)} + 6 s_{(9, 1, 1)}$
 }\\
\cline{2-3}
&  3 &
\parbox[c][6.3em]{11cm}{$ 32 s_{(11)} + 54 s_{(6, 5)} + 92 s_{(7, 4)} + 101 s_{(8, 3)} +
 84 s_{(9, 2)} + 59 s_{(10, 1)} + 15 s_{(4, 4, 3)} +
 22 s_{(5, 3, 3)} + 47 s_{(5, 4, 2)} + 28 s_{(5, 5, 1)} +
 60 s_{(6, 3, 2)} + 78 s_{(6, 4, 1)} + 38 s_{(7, 2, 2)} +
 88 s_{(7, 3, 1)} + 66 s_{(8, 2, 1)} + 31 s_{(9, 1, 1)} +
 s_{(3, 3, 3, 2)} + 3 s_{(4, 3, 2, 2)} + 5 s_{(4, 3, 3, 1)} +
 10 s_{(4, 4, 2, 1)} + 3 s_{(5, 2, 2, 2)} + 16 s_{(5, 3, 2, 1)} +
 14 s_{(5, 4, 1, 1)} + 12 s_{(6, 2, 2, 1)} + 17 s_{(6, 3, 1, 1)} +
 13 s_{(7, 2, 1, 1)} + 4 s_{(8, 1, 1, 1)}$
 }\\
\cline{2-3}
&  4 &
\parbox[c][7.5em]{11cm}{$44 s_{(11)} + 85 s_{(6, 5)} + 145 s_{(7, 4)} + 152 s_{(8, 3)} +
 125 s_{(9, 2)} + 81 s_{(10, 1)} + 30 s_{(4, 4, 3)} +
 41 s_{(5, 3, 3)} + 90 s_{(5, 4, 2)} + 49 s_{(5, 5, 1)} +
 111 s_{(6, 3, 2)} + 139 s_{(6, 4, 1)} + 70 s_{(7, 2, 2)} +
 149 s_{(7, 3, 1)} + 111 s_{(8, 2, 1)} + 47 s_{(9, 1, 1)} +
 2 s_{(3, 3, 3, 2)} + 9 s_{(4, 3, 2, 2)} + 14 s_{(4, 3, 3, 1)} +
 23 s_{(4, 4, 2, 1)} + 8 s_{(5, 2, 2, 2)} + 41 s_{(5, 3, 2, 1)} +
 33 s_{(5, 4, 1, 1)} + 27 s_{(6, 2, 2, 1)} + 41 s_{(6, 3, 1, 1)} +
 29 s_{(7, 2, 1, 1)} + 10 s_{(8, 1, 1, 1)} + s_{(3, 3, 3, 1, 1)} +
 s_{(4, 2, 2, 2, 1)} + 2 s_{(4, 3, 2, 1, 1)} +
 s_{(4, 4, 1, 1, 1)} + 2 s_{(5, 2, 2, 1, 1)} +
 3 s_{(5, 3, 1, 1, 1)} + s_{(6, 2, 1, 1, 1)} + s_{(7, 1, 1, 1, 1)} $
 }\\
\hline
\end{tabular}
}
\end{center}

\clearpage

\section{Character of $A^k(\PP^1[n])$} \label{App.B}
We list the characters of the $S_n$-representation $A^k(\PP^1[n])$ for $n\leq 11$. Note that $\ch(A^0(\PP^1[n]))=s_{(n)}$ for all $n$ and $\ch(A^{n-k}(\PP^1[n]))=\ch(A^k(\PP^1[n]))$.
\medskip
\begin{center}
{\small
\begin{tabular}{|c|c|l|}
\hline
$n$ &$k$ & \parbox[c][1.5em]{11cm}{$\ch_n(A^k(\PP^1[n]))$} \\
\hline
3 & 1 &
\parbox[c][1.5em]{11cm}{ $ 2 s_{(3)} + s_{(2, 1)} $
} \\
\hline
\multirow{2}{*}{$4$}  & 1 &
\parbox[c][1.5em]{11cm}{$ 3 s_{(4)} + 2 s_{(3, 1)}  $ }\\
\cline{2-3}
& 2 &
\parbox[c][1.5em]{11cm}{$5 s_{(4)} + s_{(2, 2)}+ 3 s_{(3, 1)} $
} \\
\hline

\multirow{2}{*}{$5$} & $1$ &
\parbox[c][1.5em]{11cm}{$ 4 s_{(5)}+ s_{(3, 2)}+ 3 s_{(4, 1)} $ }\\
\cline{2-3}
& 2 &
\parbox[c][1.5em]{11cm}{$9 s_{(5)}+ 4 s_{(3, 2)}+ 8 s_{(4, 1)}+ s_{(3, 1, 1)}$
} \\
\hline

\multirow{3}{*}{$6$} & $1$ &
\parbox[c][1.5em]{11cm}{$5 s_{(6)}+ s_{(3, 3)}+ 2 s_{(4, 2)}+ 4 s_{(5, 1)}$ }\\
\cline{2-3}
& 2 &
\parbox[c][1.5em]{11cm}{$ 15 s_{(6)}+ 4 s_{(3, 3)}+ 12 s_{(4, 2)}+ 15 s_{(5, 1)}+ 2 s_{(3, 2, 1)}+
 3 s_{(4, 1, 1)}$
} \\
\cline{2-3}
& 3 &
\parbox[c][1.5em]{11cm}{$ 20 s_{(6)}+ 8 s_{(3, 3)}+ 18 s_{(4, 2)}+ 24 s_{(5, 1)}+ 4 s_{(3, 2, 1)}+
 6 s_{(4, 1, 1)}$
} \\
\hline

\multirow{3}{*}{$7$} & $1$ &
\parbox[c][1.5em]{11cm}{$ 6 s_{(7)}+ 2 s_{(4, 3)}+ 3 s_{(5, 2)}+ 5 s_{(6, 1)} $ }\\
\cline{2-3}
& 2 &
\parbox[c][1.5em]{11cm}{$22 s_{(7)}+ 15 s_{(4, 3)}+ 23 s_{(5, 2)}+ 26 s_{(6, 1)}+ s_{(3, 2, 2)}+
 3 s_{(3, 3, 1)}+ 7 s_{(4, 2, 1)}+ 6 s_{(5, 1, 1)}$
} \\
\cline{2-3}
& 3 &
\parbox[c][2.7em]{11cm}{$40 s_{(7)}+ 37 s_{(4, 3)}+ 54 s_{(5, 2)}+ 56 s_{(6, 1)}+ 4 s_{(3, 2, 2)}+
 10 s_{(3, 3, 1)} + 23 s_{(4, 2, 1)} + 19 s_{(5, 1, 1)} + s_{(3, 2, 1, 1)} +
 s_{(4, 1, 1, 1)} $
} \\
\hline

\multirow{4}{*}{$8$} &  $1$ &
\parbox[c][1.5em]{11cm}{$7 s_{(8)} + s_{(4, 4)} + 3 s_{(5, 3)} + 4 s_{(6, 2)} + 6 s_{(7, 1)} $
} \\
\cline{2-3}
 &  2 &
\parbox[c][2.7em]{11cm}{$31 s_{(8)} + 13 s_{(4, 4)} + 30 s_{(5, 3)} + 40 s_{(6, 2)} + 39 s_{(7, 1)} +
 2 s_{(3, 3, 2)} + 4 s_{(4, 2, 2)} + 11 s_{(4, 3, 1)} + 14 s_{(5, 2, 1)} +
 11 s_{(6, 1, 1)} $
}\\
\cline{2-3}
 &  3 &
\parbox[c][3.9em]{11cm}{$70 s_{(8)} + 43 s_{(4, 4)} + 107 s_{(5, 3)} + 126 s_{(6, 2)} + 113 s_{(7, 1)} +
 13 s_{(3, 3, 2)} + 23 s_{(4, 2, 2)} + 56 s_{(4, 3, 1)} + 70 s_{(5, 2, 1)} +
 49 s_{(6, 1, 1)} + 2 s_{(3, 2, 2, 1)} + 3 s_{(3, 3, 1, 1)} +
 8 s_{(4, 2, 1, 1)} + 4 s_{(5, 1, 1, 1)}  $
}\\
\cline{2-3}
 &  4 &
\parbox[c][3.9em]{11cm}{$ 94 s_{(8)} + 67 s_{(4, 4)} + 155 s_{(5, 3)} + 185 s_{(6, 2)} + 156 s_{(7, 1)} +
 21 s_{(3, 3, 2)} + 39 s_{(4, 2, 2)} + 90 s_{(4, 3, 1)} + 112 s_{(5, 2, 1)} +
 75 s_{(6, 1, 1)} + 4 s_{(3, 2, 2, 1)} + 7 s_{(3, 3, 1, 1)} +
 15 s_{(4, 2, 1, 1)} + 9 s_{(5, 1, 1, 1)} $
}\\
\hline

\multirow{4}{*}{$9$} &  $1$ &
\parbox[c][1.5em]{11cm}{$8 s_{(9)} + 2 s_{(5, 4)} + 4 s_{(6, 3)} + 5 s_{(7, 2)} + 7 s_{(8, 1)}  $
}\\
\cline{2-3}
 &  2 &
\parbox[c][2.7em]{11cm}{$ 41 s_{(9)} + 31 s_{(5, 4)} + 53 s_{(6, 3)} + 60 s_{(7, 2)} + 56 s_{(8, 1)} +
 s_{(3, 3, 3)} + 8 s_{(4, 3, 2)} + 10 s_{(4, 4, 1)} + 8 s_{(5, 2, 2)} +
 22 s_{(5, 3, 1)} + 25 s_{(6, 2, 1)} + 17 s_{(7, 1, 1)}  $
}\\
\cline{2-3}
&  3 &
\parbox[c][3.9em]{11cm}{$  115 s_{(9)} + 151 s_{(5, 4)} + 244 s_{(6, 3)} + 251 s_{(7, 2)} + 201 s_{(8, 1)} +
 11 s_{(3, 3, 3)} + 71 s_{(4, 3, 2)} + 74 s_{(4, 4, 1)} + 69 s_{(5, 2, 2)} +
 164 s_{(5, 3, 1)} + 169 s_{(6, 2, 1)} + 102 s_{(7, 1, 1)} + s_{(3, 2, 2, 2)} +
 7 s_{(3, 3, 2, 1)} + 13 s_{(4, 2, 2, 1)} + 22 s_{(4, 3, 1, 1)} +
 25 s_{(5, 2, 1, 1)} + 13 s_{(6, 1, 1, 1)}  $
}\\
\cline{2-3}
&  4 &
\parbox[c][5.1em]{11cm}{$192 s_{(9)} + 317 s_{(5, 4)} + 497 s_{(6, 3)} + 494 s_{(7, 2)} + 370 s_{(8, 1)} +
 28 s_{(3, 3, 3)} + 181 s_{(4, 3, 2)} + 177 s_{(4, 4, 1)} + 173 s_{(5, 2, 2)} +
 391 s_{(5, 3, 1)} + 390 s_{(6, 2, 1)} + 222 s_{(7, 1, 1)} + 4 s_{(3, 2, 2, 2)} +
 25 s_{(3, 3, 2, 1)} + 44 s_{(4, 2, 2, 1)} + 73 s_{(4, 3, 1, 1)} +
 82 s_{(5, 2, 1, 1)} + 42 s_{(6, 1, 1, 1)} + s_{(3, 2, 2, 1, 1)} +
 s_{(3, 3, 1, 1, 1)} + 3 s_{(4, 2, 1, 1, 1)} + s_{(5, 1, 1, 1, 1)}  $
}\\
\hline
\multirow{2}{*}{$10$} &  $1$ &
\parbox[c][1.5em]{11cm}{$9 s_{(10)} + s_{(5, 5)} + 3 s_{(6, 4)} + 5 s_{(7, 3)} + 6 s_{(8, 2)} + 8 s_{(9, 1)} $
}\\
\cline{2-3}
 &  2 &
\parbox[c][2.7em]{11cm}{$53 s_{(10)} + 19 s_{(5, 5)} + 59 s_{(6, 4)} + 80 s_{(7, 3)} + 86 s_{(8, 2)} +
 75 s_{(9, 1)} + 4 s_{(4, 3, 3)} + 8 s_{(4, 4, 2)} + 16 s_{(5, 3, 2)} +
 24 s_{(5, 4, 1)} + 14 s_{(6, 2, 2)} + 39 s_{(6, 3, 1)} + 38 s_{(7, 2, 1)} +
 25 s_{(8, 1, 1)} $
}\\
\hline
\end{tabular}
\clearpage

\begin{tabular}{|c|c|l|}
\hline
$n$ &$k$ & \parbox[c][1.5em]{11cm}{$\ch_n(A^k(\PP^1[n]))$}\\
\hline
\multirow{3}{*}{$10$}
&  3 &
\parbox[c][5.1em]{11cm}{ $175 s_{(10)} + 134 s_{(5, 5)} + 366 s_{(6, 4)} + 474 s_{(7, 3)} + 443 s_{(8, 2)} +
 331 s_{(9, 1)} + 57 s_{(4, 3, 3)} + 99 s_{(4, 4, 2)} + 204 s_{(5, 3, 2)} +
 254 s_{(5, 4, 1)} + 159 s_{(6, 2, 2)} + 381 s_{(6, 3, 1)} + 338 s_{(7, 2, 1)} +
 187 s_{(8, 1, 1)} + 4 s_{(3, 3, 2, 2)} + 7 s_{(3, 3, 3, 1)} +
 6 s_{(4, 2, 2, 2)} + 42 s_{(4, 3, 2, 1)} + 33 s_{(4, 4, 1, 1)} +
 39 s_{(5, 2, 2, 1)} + 66 s_{(5, 3, 1, 1)} + 63 s_{(6, 2, 1, 1)} +
 29 s_{(7, 1, 1, 1)} $
}\\
\cline{2-3}
&  4 &
\parbox[c][7.5em]{11cm}{ $361 s_{(10)} + 380 s_{(5, 5)} + 1032 s_{(6, 4)} + 1272 s_{(7, 3)} +
 1132 s_{(8, 2)} + 767 s_{(9, 1)} + 218 s_{(4, 3, 3)} + 369 s_{(4, 4, 2)} +
 747 s_{(5, 3, 2)} + 866 s_{(5, 4, 1)} + 561 s_{(6, 2, 2)} + 1253 s_{(6, 3, 1)} +
 1060 s_{(7, 2, 1)} + 538 s_{(8, 1, 1)} + 24 s_{(3, 3, 2, 2)} +
 41 s_{(3, 3, 3, 1)} + 39 s_{(4, 2, 2, 2)} + 228 s_{(4, 3, 2, 1)} +
 160 s_{(4, 4, 1, 1)} + 204 s_{(5, 2, 2, 1)} + 338 s_{(5, 3, 1, 1)} +
 300 s_{(6, 2, 1, 1)} + 136 s_{(7, 1, 1, 1)} + 2 s_{(3, 2, 2, 2, 1)} +
 8 s_{(3, 3, 2, 1, 1)} + 14 s_{(4, 2, 2, 1, 1)} + 19 s_{(4, 3, 1, 1, 1)} +
 20 s_{(5, 2, 1, 1, 1)} + 8 s_{(6, 1, 1, 1, 1)} $
}\\
\cline{2-3}
& 5 &
\parbox[c][7.5em]{11cm}{ $454 s_{(10)} + 542 s_{(5, 5)} + 1430 s_{(6, 4)} + 1758 s_{(7, 3)} +
 1526 s_{(8, 2)} + 1014 s_{(9, 1)} + 332 s_{(4, 3, 3)} + 550 s_{(4, 4, 2)} +
 1122 s_{(5, 3, 2)} + 1272 s_{(5, 4, 1)} + 826 s_{(6, 2, 2)} +
 1824 s_{(6, 3, 1)} + 1520 s_{(7, 2, 1)} + 756 s_{(8, 1, 1)} +
 42 s_{(3, 3, 2, 2)} + 66 s_{(3, 3, 3, 1)} + 64 s_{(4, 2, 2, 2)} +
 372 s_{(4, 3, 2, 1)} + 260 s_{(4, 4, 1, 1)} + 332 s_{(5, 2, 2, 1)} +
 538 s_{(5, 3, 1, 1)} + 478 s_{(6, 2, 1, 1)} + 210 s_{(7, 1, 1, 1)} +
 4 s_{(3, 2, 2, 2, 1)} + 16 s_{(3, 3, 2, 1, 1)} + 28 s_{(4, 2, 2, 1, 1)} +
 38 s_{(4, 3, 1, 1, 1)} + 40 s_{(5, 2, 1, 1, 1)} + 16 s_{(6, 1, 1, 1, 1)} $
}\\
\hline

\multirow{3}{*}{$11$} &  $1$ &
\parbox[c][1.5em]{11cm}{$10 s_{(11)} + 2 s_{(6, 5)} + 4 s_{(7, 4)} + 6 s_{(8, 3)} + 7 s_{(9, 2)} + 9 s_{(10, 1)} $
}\\
\cline{2-3}
 &  2 &
 \parbox[c][3.9em]{11cm}{$ 66 s_{(11)} + 50 s_{(6, 5)} + 92 s_{(7, 4)} + 115 s_{(8, 3)} + 115 s_{(9, 2)} +
 98 s_{(10, 1)} + 5 s_{(4, 4, 3)} + 8 s_{(5, 3, 3)} + 19 s_{(5, 4, 2)} +
 15 s_{(5, 5, 1)} + 28 s_{(6, 3, 2)} + 46 s_{(6, 4, 1)} + 21 s_{(7, 2, 2)} +
 59 s_{(7, 3, 1)} + 55 s_{(8, 2, 1)} + 34 s_{(9, 1, 1)} $
 }\\
\cline{2-3}
&  3 &
\parbox[c][6.3em]{11cm}{$255 s_{(11)} + 420 s_{(6, 5)} + 725 s_{(7, 4)} + 823 s_{(8, 3)} + 721 s_{(9, 2)} +
 508 s_{(10, 1)} + 96 s_{(4, 4, 3)} + 158 s_{(5, 3, 3)} + 336 s_{(5, 4, 2)} +
 228 s_{(5, 5, 1)} + 458 s_{(6, 3, 2)} + 618 s_{(6, 4, 1)} + 308 s_{(7, 2, 2)} +
 739 s_{(7, 3, 1)} + 601 s_{(8, 2, 1)} + 311 s_{(9, 1, 1)} + 5 s_{(3, 3, 3, 2)} +
 21 s_{(4, 3, 2, 2)} + 38 s_{(4, 3, 3, 1)} + 62 s_{(4, 4, 2, 1)} +
 18 s_{(5, 2, 2, 2)} + 121 s_{(5, 3, 2, 1)} + 111 s_{(5, 4, 1, 1)} +
 89 s_{(6, 2, 2, 1)} + 157 s_{(6, 3, 1, 1)} + 127 s_{(7, 2, 1, 1)} +
 56 s_{(8, 1, 1, 1)}  $
 }\\
\cline{2-3}
&  4 &
\parbox[c][10em]{11cm}{$624 s_{(11)} + 1571 s_{(6, 5)} + 2621 s_{(7, 4)} + 2811 s_{(8, 3)} +
 2288 s_{(9, 2)} + 1446 s_{(10, 1)} + 525 s_{(4, 4, 3)} + 853 s_{(5, 3, 3)} +
 1741 s_{(5, 4, 2)} + 1082 s_{(5, 5, 1)} + 2274 s_{(6, 3, 2)} +
 2808 s_{(6, 4, 1)} + 1441 s_{(7, 2, 2)} + 3209 s_{(7, 3, 1)} +
 2431 s_{(8, 2, 1)} + 1136 s_{(9, 1, 1)} + 48 s_{(3, 3, 3, 2)} +
 202 s_{(4, 3, 2, 2)} + 321 s_{(4, 3, 3, 1)} + 489 s_{(4, 4, 2, 1)} +
 166 s_{(5, 2, 2, 2)} + 960 s_{(5, 3, 2, 1)} + 809 s_{(5, 4, 1, 1)} +
 661 s_{(6, 2, 2, 1)} + 1102 s_{(6, 3, 1, 1)} + 838 s_{(7, 2, 1, 1)} +
 342 s_{(8, 1, 1, 1)} + s_{(3, 2, 2, 2, 2)} + 13 s_{(3, 3, 2, 2, 1)} +
 16 s_{(3, 3, 3, 1, 1)} + 20 s_{(4, 2, 2, 2, 1)} + 81 s_{(4, 3, 2, 1, 1)} +
 49 s_{(4, 4, 1, 1, 1)} + 70 s_{(5, 2, 2, 1, 1)} + 98 s_{(5, 3, 1, 1, 1)} +
 80 s_{(6, 2, 1, 1, 1)} + 30 s_{(7, 1, 1, 1, 1)}  $
 }\\
 \cline{2-3}
&  5 &
\parbox[c][11.2em]{11cm}{$ 971 s_{(11)} + 2958 s_{(6, 5)} + 4857 s_{(7, 4)} + 5085 s_{(8, 3)} +
 4000 s_{(9, 2)} + 2411 s_{(10, 1)} + 1146 s_{(4, 4, 3)} + 1857 s_{(5, 3, 3)} +
 3721 s_{(5, 4, 2)} + 2234 s_{(5, 5, 1)} + 4779 s_{(6, 3, 2)} +
 5698 s_{(6, 4, 1)} + 2949 s_{(7, 2, 2)} + 6388 s_{(7, 3, 1)} +
 4692 s_{(8, 2, 1)} + 2100 s_{(9, 1, 1)} + 127 s_{(3, 3, 3, 2)} +
 525 s_{(4, 3, 2, 2)} + 800 s_{(4, 3, 3, 1)} + 1196 s_{(4, 4, 2, 1)} +
 426 s_{(5, 2, 2, 2)} + 2343 s_{(5, 3, 2, 1)} + 1916 s_{(5, 4, 1, 1)} +
 1578 s_{(6, 2, 2, 1)} + 2569 s_{(6, 3, 1, 1)} + 1908 s_{(7, 2, 1, 1)} +
 752 s_{(8, 1, 1, 1)} + 4 s_{(3, 2, 2, 2, 2)} + 46 s_{(3, 3, 2, 2, 1)} +
 54 s_{(3, 3, 3, 1, 1)} + 70 s_{(4, 2, 2, 2, 1)} + 274 s_{(4, 3, 2, 1, 1)} +
 160 s_{(4, 4, 1, 1, 1)} + 233 s_{(5, 2, 2, 1, 1)} + 324 s_{(5, 3, 1, 1, 1)} +
 260 s_{(6, 2, 1, 1, 1)} + 96 s_{(7, 1, 1, 1, 1)} + s_{(3, 2, 2, 2, 1, 1)} +
 3 s_{(3, 3, 2, 1, 1, 1)} + 5 s_{(4, 2, 2, 1, 1, 1)} +
 6 s_{(4, 3, 1, 1, 1, 1)} + 6 s_{(5, 2, 1, 1, 1, 1)} + 2 s_{(6, 1, 1, 1, 1, 1)} $
 }\\
\hline
\end{tabular}
}
\end{center}

\bibliographystyle{amsplain}

\end{document}